\documentclass{amsart}
\usepackage{amsmath,amssymb,latexsym}
\usepackage{mathpazo}
\usepackage[mathscr]{eucal}

\usepackage{pb-diagram}
\usepackage[all]{xy}
\usepackage{verbatim}

\newcommand{\calH}{\mathcal{H}}
\newcommand{\id}{\operatorname{id}}
\newcommand{\Aop}{A^{\scriptscriptstyle{\rm op}}}

\newcommand{\der}{\operatorname{Der}}
\newcommand{\Hom}{\operatorname{Hom}}
\newcommand{\V}{\mathscr{V}}
\newcommand{\J}{\mathscr{J}}

\newcommand{\G}{\mathsf{G}}

\newcommand{\calC}{\mathcal{C}}
\newcommand{\frakg}{\mathfrak{g}}
\newcommand{\frakm}{\mathfrak{m}}

\newcommand{\elle}{{\scriptscriptstyle{L}}}            % `runterskalierte Groszbuchstaben
\newcommand{\erre}{{\scriptscriptstyle{R}}}            % 

\newcommand{\rmref}[1]{{\rm (}\ref{#1}{\rm )}}

\newcommand{\ga}{\alpha}
\newcommand{\gb}{\beta}

\newcommand{\gd}{\delta}
\newcommand{\gD}{\Delta}

\newcommand{\gl}{\lambda}
\newcommand{\gL}{\Lambda}

\newcommand{\gO}{\Omega}

\newcommand{\gs}{\sigma}

\newcommand{\eps}{\epsilon}

\newcommand{\C}{{\mathbb{C}}}

\newcommand{\R}{{\mathbb{R}}}

\newcommand{\End}{\operatorname{End}}

\newcommand{\Der}{{\operatorname{Der}}}

\newcommand{\pr}{{\rm pr} \,}

\newcommand{\Tor}{\operatorname{Tor}}
\newcommand{\Ext}{\operatorname{Ext}}
\newcommand{\Cotor}{\operatorname{Cotor}}

\newcommand{{\qqquad}}{{\quad\quad\quad}}

\newcommand{{\bull}}{{\scriptscriptstyle{\bullet}}}

\newcommand{\rcp}[1]{{#1}^{-1}} 

\newcommand{\Br}{{\rm Bar}}                                   % Bar-Aufloesung
\newcommand{\Cobar}{\operatorname{Cobar}}                     % ---- dito -----

\newcommand{\Ae}{{A^\mathrm{e}}}

\newcommand{{\op}}{{\scriptscriptstyle{\mathrm{op}}}}

\newcommand{{\Aopl}}{{A^{\rm op}_\pl}}

\newcommand{{\gog}}{{G \rightrightrrows G_0}}                
\newcommand{{\rra}}{\rightrightarrows}                        %

\newcommand{\cinfc}[1]{C_{\rm c}^\infty(#1)}                %  C-unendlich Schnitte mit kompaktem Traeger
\newcommand{\cinf}[1]{C^\infty(#1)}                         %  C-unendlich(Param)
\newcommand{\pl}{\partial}                                    %  dell
\newcommand{\qttr}[5]{ {}^{#2} \! {}_{#4} {#1}^{#3}_{#5} }    %  Vierfachindex

\def\kasten#1{\mathop{\mkern0.5\thinmuskip
\vbox{\hrule
      \hbox{\vrule
            \hskip#1
            \vrule height#1 width 0pt
            \vrule}%
      \hrule}%
\mkern0.5\thinmuskip}}

\newcommand{\bx}{{\kasten{6pt}}}

\numberwithin{equation}{section}
\theoremstyle{plain}
        \newtheorem{theorem}{Theorem}[section]
        \newtheorem{lemma}[theorem]{Lemma}
        \newtheorem{proposition}[theorem]{Proposition}
        \newtheorem{corollary}[theorem]{Corollary}

\theoremstyle{definition}
        \newtheorem{definition}[theorem]{Definition}
        \newtheorem{remark}[theorem]{Remark}
        \newtheorem{example}[theorem]{Example}

\title{The cyclic theory of Hopf algebroids}
\author{Niels Kowalzig and Hessel Posthuma} 

\begin{document}
%\date{\today}
%
%

\begin{abstract}
We give a systematic description of the cyclic cohomology theory of Hopf algebroids in terms of its associated 
category of modules. 
Then we {introduce} a dual cyclic homology theory by applying cyclic duality
to the underlying cocyclic object. We derive general structure theorems for these theories in
the special cases of commutative and cocommutative Hopf algebroids. 
Finally, we compute the cyclic theory in examples 
associated to Lie-Rinehart algebras and \'etale groupoids.
\end{abstract}

\address{Niels Kowalzig: Utrecht University, 
Department of Mathematics,
P.O. Box 80.010,
3508TA Utrecht,
The Netherlands}

\email{N.Kowalzig@uva.nl}

\address{Hessel Posthuma: University of Amsterdam,
Korteweg-de Vries Institute for Mathematics, 
P.O. Box 94.248,
1090GE Amsterdam, 
The Netherlands}

\email{H.B.Posthuma@uva.nl}

\keywords{Hopf algebroids, Hopf-cyclic (co)homology, cyclic duality, Lie-Rinehart algebras, groupoids}
   
\subjclass[2010]{16T05, 16E40; 16T15, 19D55, 57T30, 58B34}

\maketitle

\tableofcontents

%\newpage

%\vspace*{-.5cm}
\section*{Introduction}

In geometry, groupoids are a joint generalisation of both spaces and
groups. 
As such they provide
a generalised symmetry concept 
that has found many applications in the theory of foliations, group actions, etc. 
In particular, the cohomology of the classifying spaces of (Lie) groupoids are the natural domain for 
the characteristic classes associated to such geometric structures.
Symmetries in noncommutative geometry, i.e.\ the noncommutative 
analogue of group actions, are encoded by the action or coaction of some
Hopf algebra on some algebra or coalgebra, which plays the r\^ole of a ``noncommutative space''.

{\em Hopf algebroids} are the noncommutative generalisation of groupoids
and as such provide a concept of generalised symmetries in noncommutative
geometry: 
they generalise Hopf algebras
to noncommutative base algebras. However, there exists more than one
definition. Originally introduced as cogroupoid objects in the category of
commutative algebras (see e.g.\ \cite{Rav:CCASHGOS}), the main
difficulty of defining Hopf algebroids stems from the fact that
the involved tensor category of bimodules is not symmetric, 
so that a straightforward generalisation of the corresponding notion 
for Hopf algebras does not make sense. 

Thinking of a Hopf algebra as a bialgebra equipped with an antipode,
the first step, the generalisation  
to so-called
{\em bialgebroids} (or {\em $\times_A$-bialgebras}) 
is unambiguous: 
this is a bialgebra object in the tensor category of bimodules over 
a (noncommutative) base algebra (cf.\ \cite{Swe:GOSA, Tak:GOAOAA, Lu:HAAQG, Schau:BONCRAASTFHB, Xu:QG, BrzMil:BBAD}). 

%First formulated in a still commutative
%setting in \cite{Swe:GOSA},
%given in full generality in 
%\cite{Tak:GOAOAA} and independently developed in \cite{Lu:HAAQG, Schau:BONCRAASTFHB, Xu:QG}, 
%all definitions are shown to be equivalent in \cite{BrzMil:BBAD}. 

Approaches begin to differ when adding the antipode.  
The first general definition appeared in \cite{Lu:HAAQG}, 
where an auxiliary structure 
(a section of a certain projection map) was needed.
Motivated by cyclic cohomology, as we discuss below, 
a closely related notion of {\em para}-Hopf algebroid was introduced in \cite{KhaRan:PHAATCC}.

In this paper we will consider the alternative definition of
\cite{BoeSzl:HAWBAAIAD, Boe:AANOHA}, 
which, 
roughly speaking, consists of introducing \textit{two} bialgebroid
structures on a given algebra, 
called left and right bialgebroid (cf.\ \cite{KadSzl:BAODTEAD}), 
and views the antipode as mapping the left structure to the right
one. 
This setup avoids the somewhat {\em ad hoc} choice of a section 
and makes the definition
completely symmetric. 
Also we will show in \S\ref{Examples} 
that Lie groupoids and Lie algebroids (or rather Lie-Rinehart algebras)
lead to natural examples of such structures.
However, the immediate
generalisation of a Hopf algebra to a noncommutative base ring is, strictly speaking, rather 
given by a {\em $\times_A$-Hopf algebra} \cite{Schau:DADOQGHA}, while Hopf algebroids 
in the sense of \cite{BoeSzl:HAWBAAIAD, Boe:AANOHA} generalise Hopf algebras equipped 
with a character (i.e.\ with a possibly ``twisted'' antipode \cite{Cra:CCOHA, ConMos:CCAHAS}). 
For reasons to be explained in Remark \ref{hopf-times}, we will refer to $\times_A$-Hopf algebras as {\em left Hopf algebroids}.

Cyclic cohomology for Hopf algebras, {\em Hopf-cyclic cohomology},  
is the noncommutative analogue of Lie algebra homology (which is recovered
when applied to universal enveloping algebras of Lie algebras). 
It was launched in the work of Connes and Moscovici \cite{ConMos:HACCATTIT} on the transversal
index theorem for foliations and defined in general
in \cite{Cra:CCOHA} (cf.\ also \cite{ConMos:CCAHAS}). 
A universal framework suited to describe all examples of
cyclic (co)homology arising from Hopf algebras (up to cyclic duality)
was given in \cite{Kay:TUHCT}, based on a construction of para-(co)cyclic objects in
symmetric monoidal categories 
in 
terms 
of 
(co)monoids.

The generalisation of Hopf-cyclic cohomology to noncommutative base
rings, i.e.\ to Hopf algebroids, has been less
explored. For instance, the general machinery from \cite{Kay:TUHCT} does not apply
to this context (because the relevant category of modules is not
symmetric and in general not even braided).
It appeared for the first time 
in the particular example of the ``extended''
Hopf algebra governing the transversal geometry of foliations
in \cite{ConMos:DCCAHASITG}. In this context,
certain bialgebroids (in fact, left Hopf algebroids) carrying a
cocyclic structure arise naturally. Extending this construction to
general Hopf algebroids is not straightforward: 
for
example, the notion of Hopf algebroid in \cite{Lu:HAAQG} is 
not well-suited to the problem. This led in \cite{KhaRan:PHAATCC} to the definition of
para-Hopf algebroids, in which the antipode of
\cite{Lu:HAAQG} is replaced by a {\em para}-antipode. 
Its axioms are principally designed for the cocyclic structure to be 
well-defined adapting the Hopf algebra case. 
However, the resulting para-antipode axioms
appear quite complicated and do not resemble the original symmetric Hopf algebra axioms.
In particular, {\em guessing} an
antipode (and hence the cyclic operator) in concrete examples remains
intricate. 

In \cite{BoeSte:CCOBAAVC} a general cyclic theory for bialgebroids and
{\em left} Hopf algebroids (in terms of so-called (co)monads) is developed that
works in an arbitrary category, and hence embraces the construction in
\cite{Kay:TUHCT} for symmetric monoidal categories.

In this paper we shall show that the cyclic cohomology theory 
for Hopf algebroids in the sense of
\cite{BoeSzl:HAWBAAIAD, Boe:AANOHA} 
is actually naturally defined 
and explain how it fits into the monoidal category of modules and the cyclic cohomology of coalgebras, generalising the corresponding Hopf 
algebra approach from \cite{Cra:CCOHA, ConMos:CCAHAS}.

Besides the cyclic cohomology, we develop a dual cyclic
homology theory 
by, roughly speaking, applying cyclic duality to the underlying cocyclic object. 
This generalises the dual theory for Hopf algebras
\cite{Cra:CCACCFF,KhaRan:ANCMFHA} 
and is more related to a certain category of comodules 
(over one of the underlying bialgebroid structures). 
It should be stressed that this homology theory is not 
simply the $\Hom$-dual of the cohomology theory mentioned above; 
it can give interesting results even when the cyclic cohomology is
trivial, cf.\ \S \ref{etale} for an example. 
Generally, in each of the classes
of examples we consider, one of the two cyclic theories
does not furnish new information compared to the respective
Hochschild 
theory, whereas the other one does. However, these examples 
are in some sense ``extremal''  with
respect to {\em primitive} and {\em (weakly) grouplike} elements---we
do not pursue this any further here.
\\

\noindent {\bf Outline.}
This paper is set up as follows: in \S \ref{hopfalgebroid} we review
the 
definition of a Hopf algebroid as in \cite{Boe:AANOHA, BoeSzl:HAWBAAIAD}
and give a brief description of the associated monoidal categories of
modules and comodules. 
We then give a systematic 
derivation of the cyclic cohomology complexes using coinvariant
localisation in the category 
of modules over the Hopf algebroid 
(\S\ref{basic} and \S\ref{coinvariants}). The dual homology is
constructed in \S\ref{dual} 
by applying the notion of duality in Connes' cyclic category, after
the cochain spaces have been mapped isomorphically
into the category of certain comodules by
means of a
{\em Hopf-Galois map} (cf.\ \cite{Schau:DADOQGHA}) associated to the Hopf algebroid.

The remainder of section \ref{cyclictheory} is devoted to some
ramifications of the 
theory. We identify the Hochschild theory 
as certain derived functors (\S\ref{schondreiuhr}) and prove structure
theorems which allow to express the cyclic theory of commutative and
cocommutative Hopf algebroid in terms of their respective Hochschild
theory (\S\ref{baldvier}). 
This generalises a similar approach for Hopf algebras \cite{KhaRan:ANCMFHA}.

Section \ref{Examples} is devoted to examples: we discuss Hopf
algebroids 
arising from \'etale groupoids, Lie-Rinehart algebras 
(or Lie algebroids), and jet spaces of Lie-Rinehart algebras. 
In all these examples, the left bialgebroid structure 
has been described before in the literature, and we add both the 
right structure and the antipode. 
For Lie-Rinehart algebras this leads to the following remarkable
conclusion: 
the universal enveloping algebra of a Lie-Rinehart algebra 
has a canonical left Hopf algebroid structure 
(in particular it is a left bialgebroid), and a full Hopf algebroid
structure 
depends on the choice of a certain flat right connection 
(cf.\ \cite{Hue:LRAGAABVA}) 
on the base algebra. 
However, its dual jet space does carry a Hopf algebroid 
structure, free of choices.

Finally, we compute the cyclic homology and cohomology in 
all 
these examples and find that it generalises well-known 
Lie groupoid and Lie algebroid resp.\ Lie-Rinehart homology and
cohomology theories. In particular, it generalises corresponding results in Hopf algebra
theory \cite{ConMos:HACCATTIT, Cra:CCACCFF, Cra:CCOHA, KhaRan:ANCMFHA}.
\\

\noindent {\bf Acknowledgements.} 
We would like to thank Andrew Baker, Gabriella B\"ohm, and Marius Crainic for stimulating discussions and comments.
This research was supported by NWO through the GQT cluster (N.K.) and a Veni grant (H.P.).

\section{Hopf Algebroids}
\label{hopfalgebroid}
\subsection{Preliminaries}
In this paper, the term ``ring'' always means ``unital and associative
ring'', and we fix a commutative ground ring $k$. Throughout the
paper, we
work in the symmetric monoidal category of $k$-modules. 
For a $k$-algebra $A$, its opposite is denoted by $\Aop$, the enveloping
algebra by $\Ae:=A\otimes_k \Aop$, and the category of left
$A$-modules by $\mathsf{Mod}(A)$.
The category of $\Ae$-modules, that is, $(A,A)$-bimodules with symmetric
action of $k$, is monoidal by means of the tensor 
product $\otimes_A$ over $A$. 
An {\em $A$-algebra} is a monoid in this category, i.e.\ 
an $(A,A)$-bimodule $U$ equipped with $(A,A)$-bimodule
morphisms $\mu:U\otimes_AU\to A$ and $\eta:A\to U$ satisfying the
usual associativity and unitality axioms. 
Likewise, the notion of an {\em $A$-coalgebra} is defined as a comonoid in
the category of $\Ae$-modules. These notions also appear under the
name {\em $A$-ring} and {\em $A$-coring} in the literature, 
see e.g.\ \cite{Boe:HA, BrzWis:CAC}.

\subsection{Bialgebroids} (cf.\ \cite{Tak:GOAOAA})
Bialgebroids are a generalisation of bialgebras. An important subtlety
 is that the algebra and coalgebra structure are defined in 
different monoidal categories. Let $A$ and $\calH$ be (unital) $k$-algebras, and
suppose we have homomorphisms $s\!\!:\!\!A\!\!\to\!\!\calH$ and $t\!\!:\!\!\Aop\!\!\to\!\!\calH$ whose 
images commute in $\calH$: 
this 
structure is equivalent to
the structure of an $\Ae$-algebra on $\calH$. 
Such objects are also called {\em $(s,t)$-rings} over $A$, whereas $s$
and $t$ are referred to as {\em source} and {\em target} maps.
Multiplication in $\calH$ from the left equips $\calH$ with the following $(A,A)$-bimodule structure
\begin{equation}
\label{bimod-lmod}
a_1\cdot h\cdot a_2:=s(a_1)t(a_2)h, \qqquad a_1,a_2\in A, \ h\in\calH.
\end{equation}
With respect to this bimodule structure we define the tensor product $\otimes_A$. 
Inside $\calH\otimes_A\calH$, there is a subspace \label{taki} called the 
{\em Takeuchi product}:
\[
\calH \times_A \calH :=  \{\textstyle\sum_i h_i \otimes_A h'_i \in \calH \otimes_A \calH \mid \sum_i h_i t_l(a)
\otimes h'_i = \sum_i h_i \otimes h'_i s_l(a), \ \forall a \in A\}.
\]
This is a unital algebra via factorwise multiplication and even an
$(s,t)$-ring again.
\begin{definition}
\label{left-bialg}
Let $A_l$ be a $k$-algebra.
A {\em left bialgebroid} over $A_l$ or {\em $A_l$-bialgebroid} is an $(s_l,t_l)$-ring $\calH_l$
equipped with the structure of an $A_l$-coalgebra
$(\Delta_l,\epsilon_l)$ 
with respect to the $(A_l,A_l)$-bimodule structure \rmref{bimod-lmod}, subject to the following conditions:
\begin{enumerate}
\item
the (left) coproduct $\Delta_l:\calH_l\to\calH_l\otimes_{A_l}\calH_l$
  maps into $\calH_l\times_{A_l} \calH_l$ 
and defines a morphism $\Delta_l:\calH_l\to\calH_l \times_{A_l} \calH_l$ of
  unital $k$-algebras;
\item
the (left) counit has the property
\begin{equation*}
%\label{leftcou}
\epsilon_l(hh') = \epsilon_l(h s_l(\epsilon_l h')) =   \epsilon_l(h
t_l(\epsilon_l h')), \qquad\qquad \mbox{for all} \ h,h'\in\calH_l.
\end{equation*}
\end{enumerate} 
\end{definition}
\noindent We shall indicate such a left bialgebroid by
$(\calH_l,A_l, s_l,t_l,\Delta_l,\epsilon_l)$, or simply 
by $\calH_l$.

Given any $(s, t)$-ring $\calH$, besides the $(A,A)$-bimodule structure \eqref{bimod-lmod}, 
one could choose the one coming from the {\em right} action of $\calH$ on 
itself:
\begin{equation}
\label{bimod-rmod}
a_1\cdot h \cdot a_2:=h t(a_1) s(a_2), \qqquad a_1,a_2\in A, \ h\in\calH.
\end{equation}
Proceeding analogously as above, this leads to the notion of a \textit{right}
bialgebroid $(\calH_r,A_r,s_r,t_r,\Delta_r,\epsilon_r)$, where the underlying algebra is 
denoted by $A_r$. We shall not write out the details, but rather refer
to \cite{KadSzl:BAODTEAD,Boe:HA}. For example, the corresponding {\em right} counit
$\eps_r:\calH_r\to A_r$ satisfies in this case 
\begin{equation*}
%\label{rightcou}
\eps_r(hh')=\eps_r(s_r(\eps_rh)h')=\eps_r(t_r(\eps_r h)h'), \qquad
\qquad \mbox{for all} \ h, h' \in \calH_r.
\end{equation*}
We will use Sweedler notation with subscripts
$\Delta_l(h)=h_{(1)}\otimes h_{(2)}$ for left coproducts, whereas
right coproducts
 are indicated
by superscripts: $\Delta_r(h)=h^{(1)}\otimes h^{(2)}$.

\subsection{Hopf algebroids}
A Hopf algebroid is now, roughly speaking, an algebra equipped with a left and a right
bialgebroid structure together with an antipode mapping from the left bialgebroid to the right.
This idea leads to the following definition: 
\begin{definition}[cf.\ \cite{BoeSzl:HAWBAAIAD}]
\label{def-hopf-algbd}
\label{HAlgd}
A {\em Hopf algebroid} is given by a triple $(\calH_l,\calH_r,S)$, where 
$\calH_l = (\calH_l,A_l,s_l,t_l,\Delta_l,\epsilon_l)$ is a left
$A_l$-bialgebroid and  $\calH_r = 
(\calH_r,A_r,s_r,t_r,\Delta_r,\epsilon_r)$ is a right $A_r$-bialgebroid on the
same $k$-algebra $\calH$, and $S:\calH\to\calH$ is a $k$-module
 map subject to the conditions:
\begin{enumerate}
\item 
the images of $s_l$ and $t_r$, as well as $t_l$ and $s_r$, coincide:
\begin{equation}
\label{Subr}
s_l \epsilon_l t_r = t_r, \quad t_l \epsilon_l s_rr = s_r, \quad s_r \epsilon_r
t_l = t_l,\quad t_r \epsilon_r s_l = s_l;
\end{equation}
\item
{\em twisted coassociativity} holds:
\begin{equation}
\label{TwCoAssoc}
(\Delta_l \otimes \id_\calH) \Delta_r = (\id_\calH \otimes \Delta_r) \Delta_l
\quad \mbox{and} \quad 
(\Delta_r \otimes \id_\calH) \Delta_l = (\id_\calH \otimes \Delta_l) \Delta_r;
\end{equation}
\item
for all $a_1\in A_l$, $a_2\in A_r$ and $h\in\calH$ we have
\[
S(t_l(a_1)ht_r(a_2))=s_r(a_2)S(h)s_l(a_1);
\]
\item
the {\em antipode axioms} are fulfilled:
\begin{equation}
\label{TwAp}
\mu_\calH(S \otimes \id_\calH)\Delta_l = s_r   \epsilon_r 
\quad \mbox{and} \quad
\mu_\calH(\id_\calH \otimes S)\Delta_r = s_l  \epsilon_l,
\end{equation}
\end{enumerate}
where $\mu_\calH$ denotes multiplication in $\calH$.
\end{definition}
Although we do not need this for all constructions in this paper, we
shall from now on assume 
that the antipode $S$ is invertible.
\begin{remark} The axioms above have the following implications (cf. \cite{BoeSzl:HAWBAAIAD, Boe:HA}):
\begin{enumerate}
\item 
Applying $\epsilon_r$ to the first two and $\epsilon_l$ to the second pair of
identities in \eqref{Subr}, one obtains that $A_l$ and $A_r$ are
anti-isomorphic as $k$-algebras, i.e.,
\begin{equation}
\label{ABIso}
\begin{array}{lcllcl}
\phi:= \epsilon_r \circ s_l: \Aop_l &\stackrel{\cong}{\longrightarrow}& A_r, \quad &  
\phi^{-1}:= \epsilon_l\circ t_r: A_r &\stackrel{\cong}{\longrightarrow}& \Aop_l, \\
\theta:= \epsilon_r\circ t_l: A_l &\stackrel{\cong}{\longrightarrow}& \Aop_r, \quad   &
\theta^{-1}:= \epsilon_l \circ s_r: \Aop_r &\stackrel{\cong}{\longrightarrow}& A_l.
\end{array}
\end{equation}
When $S^2=\id$, i.e., when the antipode is involutive, it follows from \eqref{SHomId} below that $\theta=\phi$,
so there is a canonical way to identify $A^\op_l$ with $A_r$.
\item
The antipode is an anti-algebra and anti-coalgebra morphism (between different coalgebras) and
satisfies 
\begin{equation}
\label{bleistift}
{\rm tw}\circ (S\otimes S) \gD_l= \gD_r S,  \qquad {\rm tw} \circ (S\otimes S) \gD_r =
\gD_l S, 
\end{equation}
where ${\rm tw}:\calH\otimes_k\calH\to\calH\otimes_k\calH$ is the tensor flip permuting the two factors 
(one can check that the maps above do respect the $(A_l,A_l)$-, resp.\ $(A_r,A_r)$-bimodule structure).
Likewise, one has for the inverse:
\[
{\rm tw}\circ (S^{-1}\!\otimes S^{-1})  \gD_l= \gD_r  S^{-1}, \qquad {\rm tw}\circ (S^{-1}\!\otimes S^{-1}) \gD_r
=\gD_l S^{-1}.
\]
\item
We have the identities 
\begin{equation}
\label{SHomId}
\begin{array}{rclrclrclrcl}
s_r \eps_r s_l &\!\!\!\! =&\!\!\!\!  S s_l \ &  s_l  \eps_l s_r &\!\!\!\! =&\!\!\!\!  S s_r  \ &  s_r \eps_r
\, t_l &\!\!\!\! =&\!\!\!\!  S^{-1} s_l   \ &  s_l  \eps_l \, t_r &\!\!\!\! =&\!\!\!\!  S^{-1} s_r
\\
t_r \eps_r s_l &\!\!\!\! =&\!\!\!\!  S t_l \ & t_l  \eps_l s_r &\!\!\!\! =&\!\!\!\!  S t_r  \ &  t_r \eps_r
\, t_l &\!\!\!\! =&\!\!\!\!  S^{-1} t_l   \ & t_l  \eps_l \, t_r &\!\!\!\! =&\!\!\!\!  S^{-1} t_r
\\
\eps_r s_l \eps_l &\!\!\!\! =&\!\!\!\!  \eps_r S \ & \eps_l s_r \eps_r &\!\!\!\! =&\!\!\!\!  \eps_l S \ & \eps_r \, t_l
\eps_l &\!\!\!\! =&\!\!\!\!  \eps_r S^{-1} \ & \eps_l \, t_r \eps_r &\!\!\!\! =&\!\!\!\!  \eps_l S^{-1}.
\\
\end{array}
\end{equation}
\end{enumerate}
\end{remark}
\noindent We now collect a list of basic identities involving the
antipode, the multiplication and the (left or right) comultiplication that we need later in explicit computations. 
All can be verified directly from the axioms.
\begin{lemma}
\label{SCoUConv}
For a Hopf algebroid $\calH$ with invertible antipode, the following
identities hold true:
\begin{equation*}
\label{SRel}
\begin{array}{rclrcl}
%\label{InvS1}
\mu_\calH(S \otimes s_l \eps_l)\gD_l &=& S,
\quad& 
\mu_\calH(s_r \eps_r \otimes S)\gD_r &=& S, \\
%\label{SCoU}
\mu_{\calH^{\op}}(S^2 \otimes t_l \eps_l S^2)\gD_l &=& S^2, \quad&  
\mu_{\calH^{\op}} (t_r \eps_r S^2 \otimes S^2)\gD_r &=& S^2, \\
%\label{SS2}
 \mu_{\calH^{\op}}(S^2 \otimes S)\gD_l &=& t_r \eps_r S^2,  \quad&
\mu_{\calH^{\op}} (S \otimes S^2)\gD_r &=& t_l \eps_l S^2,\\
&&&&& \\
%\label{InvAntiL}
\mu_{\calH^{\op}}(\id_\calH \otimes S^{-1})\gD_l &=& t_r \eps_r, 
\quad&  
\mu_{\calH^{\op}}(S^{-1} \otimes \id_\calH)\gD_r &=& t_l \eps_l, 
\\
\mu_{\calH^{\op}}(t_l \eps_l \otimes S^{-1})\gD_l &=& S^{-1},
\quad&  
\mu_{\calH^{\op}}(S^{-1} \otimes t_r \eps_r)\gD_r &=& S^{-1}, 
\\
%\label{SS-2}
 \mu_\calH(S^{-1} \otimes S^{-2})\gD_l &=& s_r \eps_r S^{-2},  \quad&
 \mu_\calH(S^{-2} \otimes S^{-1})\gD_r &=& s_l \eps_l S^{-2}.
\end{array}
\end{equation*}
Here $\mu_{\calH^{\op}}$ is the multiplication in the opposite of $\calH$.
\end{lemma}

\subsection{Modules and comodules}
\label{mod-comod}
Let $\calH=(\calH_l,\calH_r,S)$ be a Hopf algebroid. 
In this section we discuss several categories of modules and comodules attached to 
$\calH$, together with some basic properties.
\subsubsection{Left modules}(cf.\ \cite{Schau:BONCRAASTFHB})
A left module over $\calH$ or left $\calH$-module $M$ is simply a left module
over the underlying $k$-algebra $\calH$. We denote the structure map
usually by $(h,m) \mapsto h \cdot m$ and the category of
left $\calH$-modules by $\mathsf{Mod}(\calH)$.
The left bialgebroid structure $\calH_l$ induces the following structure on this category:
first, using the left $A^{\rm e}_l$-algebra structure, 
any module $M\in\mathsf{Mod}(\calH)$ carries an underlying $(A_l,A_l)$-bimodule structure 
by
\begin{equation}
\label{verbiegen}
a_1\cdot m\cdot a_2:=s_l(a_1) \cdot t_l(a_2) \cdot m,
\end{equation}
for all $a_1,a_2\in A_l$ and $m\in M$. This defines a forgetful functor 
\[
\mathsf{Mod}(\calH)\to \mathsf{Mod}(A^{\rm e}_l).
\]
Second, the left coproduct defines a monoidal structure on
$\mathsf{Mod}(\calH)$ by $(M,N)\mapsto M\otimes_{A_l} N$, 
equipped with the 
$\calH$-module structure 
\begin{equation}
\label{mod-tensor}
h\cdot (m\otimes n):=h_{(1)} \cdot m\otimes h_{(2)} \cdot n, \qqquad h\in\calH,~m\in M, \ n\in N.
\end{equation}
The fundamental
theorem of Schauenburg \cite[Thm.\ 5.1]{Schau:BONCRAASTFHB} 
states that conversely such tensor structure on 
$\mathsf{Mod}(\calH)$ is equivalent to a left bialgebroid structure on $\calH$.
The unit object in $\mathsf{Mod}(\calH)$ is given by $A_l$ with left $\calH$-action defined 
by 
$\calH \to \End_k (A_l), h \mapsto \{ a \mapsto \eps_l(hs_l(a))\}$. 
With this, $\mathsf{Mod}(\calH)$ is a monoidal tensor category.

\subsubsection{Right modules}
The category of \textit{right} $\calH$-modules has a similar tensor
\mbox{structure} by exploring the 
\textit{right} bialgebroid structure. 
Its unit object is given by $A_r$ equipped with a right $\calH$-module
structure induced by
the right counit:
$\calH \to \End_k (A_r), \ h \mapsto \{ a \mapsto \eps_r(s_r(a)h)\}$. 
We write $\mathsf{Mod}(\calH^{\op})$ for this tensor category. The antipode defines a functor from 
$\mathsf{Mod}(\calH)$ to $\mathsf{Mod}(\calH^{\op})$ because it is an
anti-homomorphism. When it is involutive, 
this is obviously an 
equivalence of categories.
\subsubsection{Coinvariant localisation}
%In case we have a left 
%and a right module over $\calH$, we can take the tensor product $\otimes_\calH$ over $\calH$. 
There is an important functor 
$(-)_{\scriptscriptstyle{\rm
    coinv}}:\mathsf{Mod}(\calH)\to\mathsf{Mod}(k)$ from the
category of left $\calH$-modules into the
category of $k$-modules
called {\em coinvariant localisation}, defined by
\[ 
M_{\scriptscriptstyle{\rm coinv}}:=A_r\otimes_\calH M,
\]
for $M\in\mathsf{Mod}(\calH)$. Equivalently,
$M_{\scriptscriptstyle{\rm coinv}}\cong M\slash I_r$, with $I_r$ the
$k$-module of coinvariants given by
\[
I_r:=\mbox{span}_k\{\eps_r(h)\cdot m   -h\cdot m,~h\in\calH,~m\in M\},
\]
where the $(A_r,A_r)$-bimodule structure on $M$ is defined by \rmref{verbiegen} via $\theta^{-1}\!:\!A_r\!\to \!A_l^{\op}$.
\begin{lemma} [Partial Integration]
\label{PartIntMod} 
Let $\calH$ be a Hopf algebroid as before, and $M, N \in \mathsf{Mod}(\calH)$. 
In $M\otimes_{A_l} N$  one has the identity
\[
h\cdot m \otimes n \equiv m \otimes (Sh)\cdot n \qquad \mbox{\rm mod}~I_r,
\]
for all $m \in M, n \in N $ and $h \in \calH$.
\end{lemma}
\begin{proof}
First observe that the induced $(A_r, A_r)$-bimodule structure on
$M \otimes_{A_l} N$ reads 
$$
a_1\cdot(m\otimes n)\cdot a_2:=S(s_r a_2) \cdot m\otimes s_r(a_1) \cdot n,
$$
for $a_1,a_2\in A_r$ and $m\in M$, $n\in N$.
Then one has
\begin{equation*}
\begin{array}{rcll}
m \otimes (Sh) \cdot n &\!\!\!\!=&\!\!\!\! m \otimes \big(s_r \eps_r(h^{(1)}) Sh^{(2)}\big)
 \cdot n & \mbox{by Lemma \ref{SCoUConv},} \\
 &\!\!\!\!=&\!\!\!\!   \eps_r(h^{(1)})\cdot\big( m \otimes (Sh^{(2)}) \cdot n\big) & \\
&\!\!\!\!\equiv&\!\!\!\! h^{(1)} \cdot \big(m \otimes (S h^{(2)}) \cdot n\big) 
& \mbox{mod}~I_r \\
&\!\!\!\!=&\!\!\!\!  h_{(1)} \cdot m \otimes  h^{(1)}_{(2)} \cdot 
\big(Sh^{(2)}_{(2)} \cdot n\big) & \mbox{by twisted coassociativity \rmref{TwCoAssoc},} \\
%&=  h_{(1)}m \otimes_A  (h^{(1)}_{(2)} Sh^{(2)}_{(2)})\tilde{m}
%\quad &\mbox{by $\tilde{M} \in \hmod$,} \\
&\!\!\!\!=&\!\!\!\!  h_{(1)} \cdot m \otimes \eps_l ( h_{(2)})\cdot n & \mbox{by \rmref{TwAp},} \\
%&=  h_{(1)} m \otimes_A \eps h_{(2)} \lact \tilde{m} & \\
%% &= h_{(1)} \cdot m \cdot \eps_l h_{(2)} \otimes n
%% \quad  \mbox{in the tensor product $\otimes_A$,}\\
&\!\!\!\!=&\!\!\!\! \big(t_l \eps_l
(h_{(2)})h_{(1)}\big)\cdot m \otimes_A n  & \\
&\!\!\!\!=&\!\!\!\! h\cdot m \otimes_A n, &
\end{array}
\end{equation*}
where the last identity is one of the comonoid identities of a
left bialgebroid.
\end{proof}
Considering $\calH$ as a module over itself with respect to
left multiplication, we get
\begin{proposition}
\label{ProjAntipCor}
For $M\in\mathsf{Mod}(\calH)$, there is a canonical isomorphism of $k$-modules
$$
%\vspace*{-.1cm}
(\calH \otimes_{A_l} M)_{\scriptscriptstyle{\rm coinv}}\stackrel{ \cong}{\longrightarrow} M,
%\vspace*{-.1cm}
$$
given by 
\begin{equation}
\label{almhof}
h\otimes m \longmapsto (Sh)\cdot m.
\end{equation}
\end{proposition}
\begin{proof}
Consider the map $M\to(\calH\otimes_{A_l} M)_{\scriptscriptstyle{\rm
    coinv}}$ 
induced by $m\mapsto 1\otimes m$. This clearly defines a right inverse 
to \rmref{almhof}. By the previous lemma it is also a left inverse. 
\end{proof}
\subsubsection{Right and left comodules over $\calH_r$} (cf.\ \cite{Schau:BONCRAASTFHB, Boe:GTFHA, BrzWis:CAC})
A right comodule over the underlying right bialgebroid $\calH_r$ 
is a right $A_r$-module $M$ equipped with a right $A_r$-module map
\[
\Delta_M:M\to M\otimes_{A_r}\calH, \quad m \mapsto m^{(0)} \otimes m^{(1)},
\]
satisfying the usual axioms for a coaction, where
the involved $(A_r, A_r)$-bimodule structure on $\calH_r$ 
is given by \rmref{bimod-rmod}.
We denote the category of right $\calH_r$-comodules 
by $\mathsf{Comod}_\erre(\calH_r)$. Any object
$M\in\mathsf{Comod}_\erre(\calH_r)$ carries, 
besides the right $A_r$-module
structure denoted $(m,a)\! \mapsto\! m\cdot a$, 
a commuting \textit{left} $A_r$-module structure defined by 
\begin{equation}
\label{right-A}
a\cdot m:=m^{(0)}\cdot \epsilon_r(s_r(a)m^{(1)}), \qquad a \in A_r.
\end{equation}
This yields a forgetful functor
$\mathsf{Comod}_\erre(\calH_r)\to\mathsf{Mod}(A_r^{\rm e})$. 
The category $\mathsf{Comod}_\erre(\calH_r)$ is monoidal with tensor structure $(M,N)\mapsto M\otimes_{A_r} N$ 
equipped with the comodule structure
\begin{equation}
\label{comod-tensor}
m\otimes n\mapsto m^{(0)}\otimes n^{(0)}\otimes m^{(1)} n^{(1)}.
\end{equation}
The unit is given by $A_r\in\mathsf{Comod}_\erre(\calH_r)$ 
equipped with coaction 
$a\mapsto 1\otimes s_r(a)$.
%
% \subsubsection{Left comodules over $\calH_r$}
%

A {\em left} comodule $N$ over $\calH_r$ is defined similarly as a
left $A_r$-module 
equipped with a morphism 
$\Delta_N:N\to\calH\otimes_{A_r}N$, $n \mapsto n^{(-1)} \otimes
n^{(0)}$
of left $A_r$-modules, where as before
$\calH_r$ is an $(A_r, A_r)$-bimodule by means of \rmref{bimod-rmod}.
Similarly as for right $\calH_r$-comodules, this leads to a monoidal category 
$\mathsf{Comod}_\elle(\calH_r)$ 
with unit $A_r$ equipped with the coaction 
$a\mapsto t_r(a)\otimes 1$.

\subsubsection{Comodules over $\calH_l$}
Likewise, the underlying {\em left} bialgebroid $\calH_l$ has
associated categories of left and right comodules which we will denote
by $\mathsf{Comod}_\elle(\calH_l)$ and $\mathsf{Comod}_\erre(\calH_l)$,
respectively. They have analogous structures as the category
$\mathsf{Comod}_\erre(\calH_r)$  
above. For left and right $\calH_l$-coactions, we shall use an analogous Sweedler notation as above, but with
lower indices.

\subsubsection{The cotensor product and invariants}
The {\em cotensor product} (cf.\ \cite{EilMoo:HAFICCPAIDF}) of a right $\calH_r$-comodule $M$
 and a left $\calH_r$-comodule $M'$ is defined as
 \[
 M\bx_{\calH_r} M':=\ker(\Delta_M\otimes \id_{M'} - \id_M \otimes\Delta_{M'})\subset M\otimes_{A_r}M'.
 \]
With this, the {\em space of invariants} of a, say, right $\calH_r$-comodule $M$ is defined to be
\begin{equation*}
\label{asta}
M^{\scriptscriptstyle{\rm inv}} :=M\bx_{\calH_r} A_r.
\end{equation*}
There is a canonical embedding $M^{\scriptscriptstyle{\rm inv}}\subset M$ as the subspace
\[
M^{\scriptscriptstyle{\rm inv}}\cong\{m\in M \mid \Delta_M(m)=m\otimes 1\}.
\]
Likewise, one defines invariants for a, say, left $\calH_l$-comodule
$N$ as $N^{\scriptscriptstyle{\rm inv}} \!=\! A_l \bx_{\calH_l} N \!\cong \! \{ n \in N
\mid \gD_N(n) = 1 \otimes n \}$.
The dual statement to Proposition \ref{ProjAntipCor} 
for these two kinds of invariants 
is now given by
\begin{proposition}
\label{inv-comod}
Let $M\in \mathsf{Comod}_\erre(\calH_r)$, and consider 
$M \otimes_{A_r}\calH$ 
as a right $\calH_r$-comodule by means of 
the right coproduct in $\calH$ and the
coaction {\rm \rmref{comod-tensor}}.
Then one has a canonical isomorphism of $k$-modules
\[
M\stackrel{\cong}{\longrightarrow} (M\otimes_{A_r}
\calH)^{\scriptscriptstyle{\rm inv}}, \quad m\mapsto m^{(0)}\otimes S (m^{(1)}).
\] 
Similarly, for $N\in \mathsf{Comod}_\elle(\calH_l)$, 
\begin{equation*}
\label{schwamm}
N\stackrel{\cong}{\longrightarrow} (\calH \otimes_{A_l} N)^{\scriptscriptstyle{\rm inv}}, 
\quad n\mapsto S (n_{(-1)})\otimes n_{(0)},
\end{equation*}
where $\calH \otimes_{A_l} N$ is considered as a left $\calH_l$-comodule 
by means of the left coproduct in $\calH$ and the monoidal structure of $\mathsf{Comod}_\elle(\calH_l)$
which is analogous to {\rm \rmref{comod-tensor}}. 
\end{proposition}
\begin{proof}
It is not difficult to see that both maps indeed map into the
space of invariants with respect to the coaction
\eqref{comod-tensor} and its analogue for $\mathsf{Comod}_\elle(\calH_l)$, respectively. 
To show that they are isomorphisms, define the two maps
\begin{eqnarray}
M\otimes_{A_r}\calH\to M,& & m\otimes h\mapsto m\cdot
\theta(\epsilon_l h) \nonumber \\
\label{tafelkreide}
\calH \otimes_{A_l } N \to N, & & h\otimes n\mapsto
\phi^{-1}(\epsilon_r h) \cdot n,
\end{eqnarray}
for the first and the second case, respectively, where $\phi$ and
$\theta$ are as in \rmref{ABIso}. 
Clearly, these define inverses for the respective maps above.
\end{proof}

\begin{remark}
For a left $A_l$-module $N$, the tensor product $\calH \otimes_{A_l} N$ is a
left $\calH_l$-comodule by the coaction $\gD_l \otimes \id_N$. Since the
space of
invariants of $\calH$ as a left $\calH_l$-comodule is precisely given
by $A_l$, we have the standard isomorphism 
\begin{equation}
\label{staub}
N \cong A_l \bx_{\calH_l} (\calH \otimes_{A_l} N), \quad n \mapsto 1
\otimes n,
\end{equation}
with inverse as in \rmref{tafelkreide}.
\end{remark}

\begin{remark}
In each of the tensor categories discussed in this section, the Hopf algebroid itself defines a canonical
object, by either the product or (left or right) coproduct. This
defines six---{\it a priori} different---bimodule structures on $\calH$: 
\begin{enumerate}
\item
$\calH$ is a left module over itself. As an object in
  $\mathsf{Mod}(\calH)$, 
this leads to the $(A_l,A_l)$-bimodule structure \eqref{bimod-lmod}.
\item
$\calH$ is a right module over itself. 
This leads to the $(A_r,A_r)$-bimodule structure given by \eqref{bimod-rmod}.
\item 
$\calH$ is a right $\calH_r$-comodule via the right
  comultiplication, i.e.\ an object in $\mathsf{Comod}_\erre(\calH_r)$. 
In this case, \eqref{right-A} leads to the $(A_r,A_r)$-bimodule structure 
\begin{equation}
\label{H-comod}
a_1\cdot h\cdot a_2:=s_r(a_1)hs_r(a_2).
\end{equation}
\item
As a left comodule over $\calH_r$ using $\Delta_r$, 
we get the $(A_r,A_r)$-bimodule structure $a_1\cdot h\cdot a_2:=t_r(a_2)ht_r(a_1)$.
\item
The left comultiplication gives a right comodule structure
  on $\calH$ over $\calH_l$. 
The associated $(A_l,A_l)$-bimodule structure 
reads $a_1\cdot h\cdot a_2:=t_l(a_2)ht_l(a_1)$.
\item
Finally, $\calH$ is a left comodule over $\calH_l$ using $\Delta_l$.
Similar to {\it iii}), this 
leads to the $(A_l,A_l)$-bimodule structure $a_1\cdot h\cdot a_2:=s_l(a_1)hs_l(a_2)$.
\end{enumerate} 
\end{remark}

\section{The Cyclic Theory}
\label{cyclictheory}
\subsection{Hopf-cyclic cohomology: the basic complexes}
\label{basic}
As before, let $(\calH_l,\calH_r,S)$ be a Hopf algebroid. We consider
$\calH$ as a left module over itself, 
which induces the $(A_l,A_l)$-bimodule structure \eqref{bimod-lmod}. With this we define
\[
C^n(\calH):=\underbrace{\calH\otimes_{A_l}\cdots\otimes_{A_l}\calH}_{\scriptscriptstyle{n~{\rm
      times}}}.
\]
For $n\geq 1$, define maps $\delta_i:C^n(\calH)\to C^{n+1}(\calH)$ by
\begin{equation}
\label{coface-hochschild}
\delta_i(h^1 \otimes\cdots \otimes h^n) := 
\begin{cases}
1 \otimes h^1 \otimes \cdots \otimes h^n \quad
&\mbox{if} \ i=0, \\
 h^1 \otimes\cdots \otimes \Delta_l h^i \otimes \cdots \otimes h^n \quad &\mbox{if} \
1 \leq i \leq n, \\
h^1 \otimes \cdots \otimes h^n \otimes 1 \quad &\mbox{if} \ i = n + 1. 
\end{cases}
\end{equation}
For $n=0$, put $C^0(\calH):=A_l$ and define
\begin{equation}
\label{CoFaceA}
\delta_i(a):=\begin{cases}t_l(a) \quad
& \mbox{if} \ i=0, \\
 s_l(a) \quad & \mbox{if} \ i =1.
 \end{cases}
\end{equation}
In the opposite direction we have codegeneracies $\sigma_i:C^n(\calH)\to C^{n-1}(\calH)$ given by
\begin{equation}
\label{CoDegHopf}
\sigma_i(h^1 \otimes \cdots \otimes h^n) = h^1 \otimes \cdots \otimes
\epsilon_l (h^{i+1}) \cdot h^{i+2} \otimes \cdots \otimes h^n, \quad 0 \leq i \leq n-1.
\end{equation}
One easily verifies that $(\delta_\bull,\sigma_\bull)$ equip
$C^\bull(\calH)$ with the structure of a cosimplicial space which
only depends on the underlying left bialgebroid structure of
$\calH$. For further use, introduce as usual the Hochschild differential by $b := \sum^{n+1}_{i=0}(-1)^i \gd_i$.
Next, define the cyclic operator $\tau_n:C^n(\calH)\to 
C^n(\calH)$ by
\begin{equation} 
\label{CyclicHopf}
\tau_n(h^1 \otimes \cdots \otimes h^n) = (Sh^1)\cdot(h^2
\otimes \cdots \otimes h^n \otimes 1),
\end{equation}
where the $\calH$-module structure on $C^n(\calH)$ is given as in \eqref{mod-tensor}.
\begin{theorem}
\label{CyclThm}
For a Hopf algebroid $(\calH_l,\calH_r,S)$ the formulae
above give $C^\bull(\calH)$ the structure of a cocyclic module if and
only if $S^2 = \id$. More specifically,
\[
\tau^{n+1}_n(h^1 \otimes \cdots \otimes h^n) = S^2 (h^1) \otimes\cdots \otimes
S^2 (h^n).
\]
\end{theorem}
\begin{remark}
This theorem was first proved by Connes and Moscovici in \cite{ConMos:DCCAHASITG}
in a special case using a characteristic 
map associated to a faithful trace. A more general
version (for the {\em if}-direction) appeared in \cite{KhaRan:PHAATCC} 
for so-called para-Hopf algebroids. 
\end{remark}

\subsection{The approach via coinvariants}
\label{coinvariants}
In this section we will prove Theorem \ref{CyclThm} using coinvariant localisation. This approach is
inspired by the analogous procedure for Hopf algebras as in \cite{Cra:CCOHA}.

Let us first define the fundamental cocyclic module associated to 
a Hopf algebroid
%$(\calH,\calH_l,\calH_r)$ 
arising from its underlying left
bialgebroid structure: define the $k$-module
\[
B^n(\calH):=A_r\otimes_{A_l^{\rm e}}C^{n+1}(\calH).
\]
Here, the right $A_l^{\rm e}$-module structure 
on $A_r$ is given using
$\theta:A_l\to A_r^\op$, whereas the left $A_l^{\rm e}$-module
structure on $C^{n+1}(\calH)$ is defined using $s_l$ and $t_l':= t_l \circ
\theta^{-1} \circ \phi$: the $k$-module $B^n(\calH)$ 
is isomorphic to a quotient of $C^{n+1}(\calH)$ by the $k$-module $I^n\subset C^{n+1}(\calH)$ defined by
\[
I^n:=\mbox{span}\{ s_l(a)h^0\otimes\cdots\otimes h^n-h^0\otimes\cdots\otimes t_l'(a)h^n,~a\in A_l,~h^i\in\calH\}.
\]
When $S^2=\id$, and therefore $\theta = \phi$,  
the right hand side is the cyclic tensor product (cf.\ \cite{Qui:CCAAE}) of
$\calH$ in the category of $(A_l,A_l)$-bimodules with respect to the
bimodule structure induced by the forgetful functor
$\mathsf{Mod}(\calH) \to \mathsf{Mod}(A_l^{\rm e})$.

Define the coface, codegeneracy, and cocyclic operators on $B^n(\calH)$ as follows:
\begin{equation}\begin{split}
\label{CoACOp}
\gd_i(h^0 \otimes h^1 \otimes \cdots \otimes h^n)  &= \left\{
\begin{array}{ll}
%h^0_{(1)} \otimes h^0_{(2)} \otimes h^1 \otimes \cdots \otimes
%h^n  \quad
%& \mbox{if} \ i = 0 \\
h^0 \otimes \cdots \otimes \gD_l h^i \otimes \cdots \otimes h^n &
 \quad \, \mbox{if} \ 0 \leq i \leq n, \\
 h^0_{(2)} \otimes h^1 \otimes \cdots \otimes  S^2(h^0_{(1)}) & \quad
 \, \mbox{if} \
i = n+ 1,
\end{array}\right. \\
\sigma_i(h^0 \otimes h^1 \otimes  \cdots \otimes  h^n ) &= h^0 \otimes
\cdots \otimes h^i \cdot \eps_l (h^{i+1}) \otimes
\cdots \otimes h^n, \ \ 0 \leq i \leq n-1, \\
\tau_n(h^0 \otimes h^1 \otimes  \cdots \otimes h^n ) &=  h^1 \otimes h^2 \otimes
\cdots \otimes h^n \otimes S^2(h^0).
\end{split}
\end{equation}
It is easy to verify that with these structure maps $B^\bull(\calH)$
is a para-cocyclic module, 
which is cocyclic if and only if $S^2=\id$.
 In this case this is just the canonical cocyclic module  
associated to the $A_l$-coalgebra
$(\calH_l,\Delta_l,\epsilon_l)$ arising from the underlying left
bialgebroid structure, and to which we will refer as
$\calH^\bull_{\scriptscriptstyle{{\rm coalg}},
  \natural}$. 
Likewise, \nolinebreak 
the underlying right
bialgebroid gives rise to a similar construction by means of $(\calH_r, \gD_r,
\eps_r)$.

On the other hand, we have $C^{n+1}(\calH)\in\mathsf{Mod}(\calH)$, and
we can apply the functor of coinvariants
to get $C^{n+1}(\calH)_{\scriptscriptstyle{\rm coinv}}
\cong
C^n(\calH)$ by Proposition \ref{ProjAntipCor}. Explicitly, 
this isomorphism
is implemented by the maps
\[
\vspace*{-.1cm}
\xymatrix{
C^n(\calH)\ar@<0.5ex>[rr]^{\Phi_{\scriptscriptstyle{\rm coinv}}}&&
\ar@<0.5ex>[ll]^{\Psi_{\scriptscriptstyle{\rm coinv}}} C^{n+1}(\calH),}
\vspace*{-.2cm}
\]
given by
\begin{eqnarray}
\label{castor}
\Phi_{\scriptscriptstyle{\rm coinv}}(h^1\otimes\cdots\otimes h^n) &:=&
1\otimes h^1\otimes\cdots\otimes h^n, \\
%\end{equation}
%and
%\begin{equation}
\label{pollux}
\Psi_{\scriptscriptstyle{\rm coinv}}(h^0\otimes\cdots\otimes h^n) &:=&
S(h^0)\cdot\left(h^1\otimes\cdots\otimes h^n\right).
\end{eqnarray}
Now observe that $I^n\subseteq\ker(\Psi_{\scriptscriptstyle{\rm coinv}})$,
so that we have a diagram
\[
\xymatrix{C^n(\calH)\ar[rr]^{\Phi_{\scriptscriptstyle{\rm coinv}}}&&C^{n+1}(\calH)\ar[dl]^{\pi}\\
&B^n(\calH)\ar[ul]^{\overline{\Psi}_{\scriptscriptstyle{\rm coinv}}}
}
\] 
where $\pi$ denotes the canonical projection and
$\overline{\Psi}_{\scriptscriptstyle{\rm coinv}}$ the induced map.
% which is induced by the injection
% $\Ae \to \calH$, $a_1 \otimes a_2 \mapsto s_l(a_1)t_l(a_2)$.
\begin{proposition}
The morphism $\overline{\Psi}_{\scriptscriptstyle{\rm coinv}}$ intertwines the
maps $\delta_i$, $\sigma_i$, and $\tau_n$ from
{\rm \rmref{coface-hochschild}--\rmref{CyclicHopf}} with the respective ones
from {\rm \rmref{CoACOp}}.
\end{proposition}
\begin{proof}
Consider first the cyclic operator
\[
\begin{split}
\tau_n\overline{\Psi}_{\scriptscriptstyle{\rm coinv}}&(h_0\otimes h_1\otimes\cdots\otimes h_n)
= \tau_n\big(S(h_0)\cdot(h_1\otimes\cdots\otimes h_n)\big)\\
&=S\big((Sh_0)_{(1)}h_1\big)\cdot\left((Sh_0)_{(2)}h_2\otimes\cdots\otimes (Sh)_{(n)}h_n\otimes 1\right)\\
&=S(h_1)\cdot S^2(h^{(n)}_0)\cdot\left( S(h^{(n-1)}_0)h_2\otimes\cdots\otimes S(h^{(1)}_0)h_n\otimes 1\right),
\end{split}
\]
where we used that $C^n(\calH)\in\mathsf{Mod}(\calH)$ with the module structure on tensor products 
given by \eqref{mod-tensor}, as well as the fact that the antipode $S$ is an anti-algebra homomorphism.
On the other hand,
\[
\overline{\Psi}_{\scriptscriptstyle{\rm coinv}} \tau_n\left(h_0\otimes
h^1\otimes\cdots\otimes h^n\right)
=S(h_1)\cdot\left(h_2\otimes\cdots\otimes h_n\otimes S^2h_0\right). 
\]
The statement therefore follows from the following:
\begin{lemma}
In $C^n(\calH)$ the following identity holds:
\[
S^2(h^{(n)}_0)\cdot\left( S(h^{(n-1)}_0)h_1\otimes\cdots\otimes
S(h^{(1)}_0)h_{n-1}\otimes 1\right)
=h_1\otimes\cdots\otimes h_{n-1}\otimes S^2 h_0.
\]
\end{lemma}
\begin{proof}
This is proved by induction:
first, for $n=2$ we have by 
the right comonoid identities, \rmref{SRel},
\rmref{SHomId}, and \rmref{TwAp} for any $h_1, h \in \calH$ in $C^2(\calH)$
\begin{equation}
\label{weihnachtsmann}
\begin{split}
h_1 \otimes Sh &= h_1 \otimes S\big(h^{(1)}s_r (\epsilon_r h^{(2)})\big)  \\
&= t_l \eps_l (Sh^{(2)}) h_1 \otimes Sh^{(1)} \\
&= s_r \eps_r h^{(2)} h_1 \otimes Sh^{(1)} \\
%
% &= (\mHop(S^{-1} \otimes \id)\gD_r Sh^{(2)}) \otimes_A Sh^{(1)} \\
% &= (\mHop(S^{-1} \otimes \id)(Sh^{(2)}_{(2)} \otimes^B Sh^{(2)}_{(1)}))
% \otimes_A Sh^{(1)} \\
%
&= Sh^{(2)}_{(1)} h^{(2)}_{(2)} h_1 \otimes Sh^{(1)} =
Sh^{(2)}_{(1)}h_{(2)} h_1 \otimes Sh^{(1)}_{(1)}.
\end{split}
\end{equation}
Applying this identity to $h :=S h_0$ proves the case
$n=2$. 
Assume now that the identity holds for $n-1$. Then we have, 
using 
%\rmref{ItApCoP}, 
\rmref{SRel}, \rmref{SHomId}, and \rmref{TwAp},
\[
\begin{split}
h_1 \otimes & \cdots \otimes h_{n} \otimes S^2 h_0 = h_1  \otimes
\left(S^2(h^{(n)}_0)\cdot\left(S h^{(n-1)}_0 h_2 \otimes
\cdots\otimes  S h^{(1)}_0 h_{n} \otimes 1\right)\right) \\
%&= h_1 \otimes S^2 (h^{(n)}_{(1)} Sh^{(n-1)} h_3 \otimes_A
%\cdots \otimes_A S^2 h^{(n)}_{(n-1)} Sh^{(1)} h_{n+1} \otimes_A  S^2h^{(n)}_{(n)} \\
 &= h_1 \otimes s_l \eps_l( S^2 {h_0}^{(n)}_{(1)}) S^2
	       {h_0}^{(n)}_{(2)} \cdot\left(S h^{(n-1)}_0 h_2 \otimes
\cdots\otimes  S h^{(1)}_0 h_{n} \otimes 1\right) \\
 &= s_r \epsilon_r(S {h_0}^{(n)}_{(1)}) h_1\otimes S^2
	       {h_0}^{(n)}_{(2)} \cdot\left(Sh^{(n-1)}_0 h_2 \otimes
\cdots\otimes  Sh^{(1)}_0 h_{n} \otimes 1\right)\\
%\cdots \otimes_A S^2 h^{(n)}_{(n)} Sh^{(1)} h_{n+1} \otimes_A  S^2 h^{(n)}_{(n+1)} \\
% 
% &= (\mHop(S \otimes S^2)\gD^r S^2 h^{(n)}_{(1)}) h_2 \otimes_A S^2 h^{(n)}_{(2)} Sh^{(n-1)} h_3 \otimes_A
% \cdots \otimes_A S^2 h^{(n)}_{(n)} Sh^{(1)} h_{n+1} \otimes_A  S^2 h^{(n)}_{(n+1)} \\
%  &= (\mHop(S^{-1} \otimes \id)(S^2 h^{(n)} \otimes^B S^2
% h^{(n+1)}_{(1)}) h_2 \otimes_A \\
% & \qqquad \otimes_A S^2 h^{(n+1)}_{(2)} Sh^{(n-1)} h_3 \otimes_A
% \cdots \otimes_A S^2 h^{(n+1)}_{(n)} Sh^{(1)} h_{n+1} \otimes_A  S^2
% h^{(n+1)}_{(n +1)} \\
 &= S^2 ({h_0}^{(n+1)}_{(1)}) S h^{(n)}_0 h_2 \otimes S^2
	       {h_0}^{(n+1)}_{(2)} \cdot\left(Sh^{(n-1)}_0 h_2 \otimes
\cdots\otimes  Sh^{(1)}_0 h_{n} \otimes 1\right)\\
&=S^2(h^{(n+1)}_0)\cdot\left(S h^{(n)}_0 h_2\otimes\cdots\otimes S
	       h^{(1)}_0 h_{n+1}\otimes 1\right).
%&=(\gD^{n}_\ell(S^2h^{(n+1)}))(Sh^{(n)}h_2 \otimes_A \cdots\otimes_A  Sh^{(1)} h_{n+1} \otimes_A 1)
\end{split}
\]
This completes the proof of the lemma.
\end{proof}
\begin{remark}
The identity \eqref{weihnachtsmann} for $h_1=1$ 
appears as an axiom in the definition of a para-Hopf algebroid in \cite{KhaRan:PHAATCC}.
\end{remark}
Hence the proposition is proved.
\end{proof}
\begin{proof} ({\em of Theorem} \ref{CyclThm})
Since $\overline{\Psi}_{\scriptscriptstyle{\rm coinv}}$ is surjective
with right inverse $\Phi_{\scriptscriptstyle{\rm coinv}}$, this proves Theorem \ref{CyclThm}.
\end{proof}

\begin{definition}
In case $S^2 = \id$, we denote by $\calH_\natural^\bull :=
C^\bull(\calH)$ 
the cocyclic module equipped with the operators
  \rmref{coface-hochschild}--\rmref{CyclicHopf} and by
  $(C^\bull(\calH), b, B)$ its associated mixed complex (cf.\ \cite{Con:NCDG, Kas:CHCAMC}).
  Its Hochschild and
  (periodic) cyclic cohomology groups are denoted by
  $HH^\bull(\calH)$, $HC^\bull(\calH)$ and $HP^\bull(\calH)$, and referred to as {\em Hopf-cyclic} cohomology groups.
\end{definition}

\begin{remark}
One may think of the forgetful functor
$\mathsf{Mod}(\calH)\to\mathsf{Mod}(A_l^{\rm e})$ as coming from the
morphism of Hopf algebroids $\iota: A_l^{\rm e}\to\calH$, $a_1\otimes
a_2\mapsto s_l(a_1)t_l(a_2)$; see Section \ref{pg} for the
description of the Hopf algebroid structure of $A_l^{\rm e}$.
From this point 
of view, $B^n(\calH)$ is simply the 
coinvariant localisation of $C^{n+1}(\calH)\in\mathsf{Mod}(\calH)$
with respect to $A_l^{\rm e}$. On the other hand, one can also directly
show that under the projection $B^n(\calH) \to
C^{n+1}(\calH)_{\scriptscriptstyle{\rm coinv}}$
induced by $\iota$, the operators \rmref{CoACOp} descend to
well-defined maps on
$C^{n+1}(\calH)_{\scriptscriptstyle{\rm coinv}}$,
turning it into a cocyclic module.
\end{remark}

\subsection{Dual Hopf-cyclic homology}
\label{dual}
\subsubsection{The chain complexes}
\vspace*{-.2cm}
In this section we consider $\calH$ as a right $\calH_r$-comodule 
with $(A_r, A_r)$-bimodule structure \eqref{H-comod} given by left and right multiplication with $s_r(a),~a\in A_r$.
Using the tensor structure of the category
$\mathsf{Comod}_\erre(\calH_r)$, 
we define
\[
C_n(\calH):=\underbrace{\calH\otimes_{A_r}\cdots\otimes_{A_r}\calH}_{\scriptscriptstyle{n~{\rm
      times}}}.
\]
Face and degeneracy operators can be introduced by
\begin{equation}
\label{DualHH}
\begin{split} 
\!\! d_i(h^1  \otimes   \cdots   \otimes   h^n)  &=  \left\{ \!\!
\begin{array}{ll}
\eps_r(h^1) \cdot h^2  \otimes   \cdots    \otimes
  h^n  
& \qquad\ \, \mbox{if} \ i = 0,  
\\
h^1  \otimes   \cdots \otimes   h^i h^{i+1}
 \otimes  \cdots   \otimes   h^n 
& \qquad\ \, \mbox{if} \ 1 \leq i \leq n-1,  
\\
h^1  \otimes   \cdots    \otimes  
h^{n-1} \cdot \eps_r (S^{-1} h^n) 
& \qquad\ \, \mbox{if} \ i = n, 
\end{array}\right. \\ 
\!\! s_i(h^1  \otimes   \cdots   \otimes   h^n)  &= \left\{ \!\!
\begin{array}{ll}
1  \otimes   h^1  \otimes   \cdots    \otimes h^n
& \!\! \mbox{if} \ i = 0,    
\\
h^1  \otimes   \cdots \otimes   h^i 
\otimes   1  \otimes   h^{i+1}
 \otimes  \cdots   \otimes   h^n 
& \!\! \mbox{if} \ 1 \leq i \leq n.  
%\\
%h^1  \otimes   \cdots    \otimes  
%h^{n}  \otimes   1 
%\\ \mbox{if} \ i = n. 
\end{array} \right.
\end{split}
\end{equation}
Elements of degree zero (of $A_r$, that is) are mapped to zero, i.e.,
$d_0(a) = 0$, $a \in A_r$.
To define a cyclic structure we assume the antipode $S$ to be invertible and define
\begin{equation}
%\begin{split}
\label{DualCyc}
t_n(h^1\otimes  \cdots  \otimes h^n) =
S^{-1}( h^{1}_{(2)} \cdots h^{n-1}_{(2)}  h^{n} )\otimes  h^{1}_{(1)} \otimes h^{2}_{(1)}
\otimes  \cdots \otimes  h^{n-1}_{(1)}. 
%& = S^{-1}( h^{1}_{(2)} \cdots h^{n-1}_{(2)}  h^{n}_{(2)} )\qttr
%\otimes {}{}B{} \,  
%h^{1}_{(1)} \qttr \otimes {}{}B{} \, h^{2}_{(1)} \qttr
%\otimes {}{}B{} \, \cdots \qttr \otimes {}{}B{} \, h^{n-1}_{(1)} t^\ell(\eps h^n_{(1)} ).
%\end{split}
\end{equation}
One easily verifies that this operator is well-defined. Below we shall prove that these $k$-modules and
maps are canonically isomorphic to the cyclic dual of the cocyclic module $C^\bull(\calH)$ of \S \ref{basic}.
This also proves that $C_\bull(\calH)$ is indeed a cyclic module.

\subsubsection{Cyclic duality} (cf.\ \cite{Con:CCEFE, Lod:CH}) We recall the notion of cyclic
duality. Let $\gL$ denote Connes' cyclic category. 
A cyclic module is a functor $\gL^\op \to\mathsf{Mod}(k)$, 
i.e.\ a contravariant functor
from $\gL$ to $\mathsf{Mod}(k)$; whereas a cocyclic module is a 
functor $\gL \to \mathsf{Mod}(k)$. 
Remarkably, there is a canonical equivalence $\gL \cong \gL^{\op}$
that allows one to construct a cocyclic
 module out of a cyclic module and {\it vice versa}. Explicitly, in the first direction this is
 done as follows:
let $Y = (Y^\bull, \gd_\bull, \gs_\bull, \tau_\bull)$ be a
para-cocyclic module with invertible operator $\tau$. Its {\em
  cyclic dual} is defined to be $\check{Y} := (\check{Y}_\bull, d_\bull, s_\bull,
t_\bull)$ where $\check{Y}_n := Y^n$ in degree $n$ and
\begin{equation*}
\begin{array}{lcrcll}
d_i := \gs_{i-1} & \!\!: & \check{Y}_n & \to & \check{Y}_{n-1}, & \quad 1 \leq i \leq n,
\\
d_0 := \gs_{n-1} \tau_n  & \!\!: & \check{Y}_n & \to & \check{Y}_{n-1}, &  \\
s_i := \gd_{i} & \!\!: & \check{Y}_n & \to & \check{Y}_{n+1}, & \quad 0 \leq i \leq n-1, \\
t_n :=  \tau^{-1}_n  & \!\!: & \check{Y}_n & \to & \check{Y}_{n}. &
\end{array}
\end{equation*}
It can be shown that $\check{Y}$ carries the structure of a
para-cyclic object in the category of $k$-modules and is cyclic if $Y$ is cocyclic. 

\subsubsection{The Hopf-Galois map and cyclic duality}
\label{adarte}
In this section we will prove that the cyclic module dual to $C^\bull(\calH)$ is canonically isomorphic 
to  $C_\bull(\calH)$. The explicit map implementing this isomorphism is given by 
generalising the Hopf-Galois map 
from
\cite[Thm.\ 3.5]{Schau:DADOQGHA} and its inverse from \cite{BoeSzl:HAWBAAIAD} for Hopf
algebroids.
\begin{lemma}
\label{hopf-galois}
For each $n \geq 0$, the $k$-modules $C_n(\calH)$ and $C^n(\calH)$ are
isomorphic by means of the {\em Hopf-Galois map}
$\varphi_n:C_n(\calH)\to C^n(\calH)$ defined inductively by $\varphi_1:=\id_\calH$ and
\[
\varphi_n(h^1\otimes\cdots\otimes h^n):=h^1\cdot\left(1\otimes\varphi_{n-1}(h^2\otimes\cdots\otimes h^n)\right), \qquad n \geq 2.
\]
For $n = 0$ one defines $\varphi_0:=\phi:A_l^\op \to A_r$.
\end{lemma}
\begin{proof}
The explicit formula for the Hopf-Galois map is given by
\begin{equation*}
\label{Hin}
h^1  \otimes  \cdots \otimes h^n \mapsto   h^1_{(1)} \otimes h^1_{(2)}
h^2_{(1)} \otimes h^1_{(3)} h^2_{(2)} h^3_{(1)} \otimes
\cdots \otimes h^1_{(n)} h^2_{(n-1)} \cdots h^{n-1}_{(2)}h^n.
\end{equation*}
Define its inverse $\varphi_n^{-1}:C^n(\calH)\to C_n(\calH)$ by
\begin{equation*}
\label{Her}
h_1 \otimes \cdots \otimes h_n \mapsto  h^{(1)}_1 \otimes S(h^{(2)}_1)  h^{(1)}_2  \otimes
S(h^{(2)}_2)   h^{(1)}_3  \otimes \cdots  \otimes
S(h^{(2)}_{n-1})  h_n.
\end{equation*}
To check that this is indeed an inverse, remark that one can decompose
$$
\varphi_{n+1} = (\id \otimes \varphi_n) \, (\varphi_{2} \otimes
\id^{\otimes n -1}) \quad \mbox{and} \quad \psi_{n+1} = (\psi_2 \otimes
\id^{\otimes n-1})   \, (\id \otimes \psi_n),
$$
and one easily verifies by induction that $\varphi_{n+1}$ and $\psi_{n+1}$ are
mutually inverse.
\end{proof}
\noindent To prove our main theorem about cyclic duality, we also need the
inverses of the 
cyclic operators on $C^\bull(\calH)$ and $C_\bull(\calH)$. Since
we assume $S$ to be invertible, we have:
\begin{lemma}
Let $n\geq1$.
The inverse of the cocyclic operator $\tau_n$  on $C^n(\calH)$ in
\eqref{CyclicHopf} is 
given 
by
\begin{equation}
\label{CyclHopfInv}
\tau^{-1}_n(h^1 \otimes \cdots \otimes h^n) = 
S^{-1}(h^n) \cdot(1 \otimes h^1
\otimes \cdots \otimes h^{n-1}).
\end{equation}
Likewise, the cyclic operator $t_n$ on $C_n(\calH)$ has inverse given by
\begin{equation}
\label{DCyclHopfInv}
t^{-1}_n(h_1 \otimes\cdots \otimes h_n) =
h_2^{(1)}  \otimes \cdots \otimes h_n^{(1)} \otimes
S(h_1 h_2^{(2)} \cdots h_n^{(2)})
\end{equation}
\end{lemma}
\begin{proof}
This can be verified directly, but we shall use induction. For $n=1$, by \eqref{CyclHopfInv} we have $\tau_1^{-1}=S^{-1}$, 
and the statement is clear. For $n\geq 2$, define the map
\[
\tilde{\varphi}(h\otimes h'):=h_{(1)}h'\otimes h_{(2)}.
\]
This defines a bijection of
$\calH\otimes_{A_l}\calH\!\in\!\mathsf{Comod}_\elle(\calH_l)$ to
$C^2(\calH)\!\in\!\mathsf{Mod}(\calH)$, with inverse
\[
\tilde{\varphi}^{-1}(h\otimes h')={h'}^{(2)}\otimes S^{-1}({h'}^{(1)})h.
\]
With these maps, one has 
\begin{eqnarray*}
\tau_{n+1} &=& ( \id^{\otimes n-1} \otimes \tilde{\varphi}) \ (\tau_n \otimes
\id) \\
\tau^{-1}_{n+1} &=& (\tau^{-1}_n \otimes \id) \ ( \id^{\otimes n-1} \otimes \tilde{\varphi}^{-1}).
\end{eqnarray*}
This proves the first statement. As for the second part, introduce
the maps
\[
\begin{split}
\tilde{\psi}(h\otimes h')&:=h_{(2)}h'\otimes h_{(1)}\\
\tilde{\psi}^{-1}(h\otimes h')&:={h'}^{(1)}\otimes S({h'}^{(2)})h.
\end{split}
\]
This time $\tilde{\psi}$ maps the tensor product $\calH \otimes_{A_r}
\calH \!\in\! \mathsf{Comod}_\erre(\calH_r)$ to the tensor \mbox{product}
$$
 \calH \otimes_{A_r} \calH :=  \calH \otimes_k \calH/{\rm span}_k
 \{ t^r(a) h \otimes_k h' -   h \otimes_k s^r (\phi \theta^{-1}(a)) h', \ a
 \in A_r  \},
$$
and one easily checks that $\tilde{\psi}^{-1}$ is its inverse, indeed.
Then one has
\begin{eqnarray*}
t_{n+1} &=& (t_n \otimes \id) \ ( \id^{\otimes n-1} \otimes \tilde{\psi}) \\
t^{-1}_{n+1} &=& (\id^{\otimes n-1} \otimes \tilde{\psi}^{-1}) \ (t^{-1}_n \otimes \id),
\end{eqnarray*}
and with this one proves the second equality.
\end{proof}

\begin{theorem}
\label{cyclic-duality}
Let $\calH$ be a Hopf algebroid with invertible antipode. The
Hopf-Galois map $\varphi:C_\bull(\calH) \to C^\bull(\calH)$ 
identifies $C_\bull(\calH)$ as the cyclic dual of the para-cocyclic module $C^\bull(\calH)$ of Theorem
{\rm \ref{CyclThm}}.
\end{theorem}
\begin{proof}
This is now a straightforward verification:
\[
\begin{split}
\tau^{-1}_n \varphi_n( &h^1 \otimes \cdots \otimes h^n) = \\
&=
\tau^{-1}_n(h^1_{(1)} \otimes h^1_{(2)}
h^2_{(1)} \otimes h^1_{(3)} h^2_{(2)} h^3_{(1)} \otimes 
\cdots \otimes  h^1_{(n)} h^2_{(n-1)} \cdots h^{n-1}_{(2)}h^n) \\
&= S^{-1}(h^1_{(n)} \cdots h^{n-1}_{(2)} h^n)\cdot
\left(1\otimes   h^1_{(1)} \otimes \cdots  \otimes
h^1_{(n-1)}h^2_{(n-2)} \cdots h^{n-1}_{(1)}\right),
\end{split}
\end{equation*}
and by coassociativity
\begin{equation*}
\begin{split}
 & \varphi_n t_n( h^1 \otimes\cdots\otimes h^n) = \varphi_n(S^{-1}(
  h^{1}_{(2)} \cdots h^{n-1}_{(2)}  h^{n}) \otimes  h^{1}_{(1)} 
\otimes h^{2}_{(1)} \otimes \cdots \otimes h^{n-1}_{(1)}) \\
&= S^{-1}(h^1_{(2)} \cdots h^{n-1}_{(2)} h^n)\cdot 
\left(1\otimes \varphi_{n-1}\left(h^{1}_{(1)} \otimes h^{2}_{(1)} \otimes \cdots \otimes h^{n-1}_{(1)})\right)\right)\\
&=S^{-1}(h^1_{(n)} \cdots h^{n-1}_{(2)} h^n)\cdot 
\left(1\otimes   h^1_{(1)} \otimes \cdots  \otimes   h^1_{(n-1)}h^2_{(n-2)} \cdots h^{n-1}_{(1)}\right).
\end{split}
\]
Hence, $\varphi_n \circ t_n = \tau^{-1}_n \circ \varphi_n$
or equivalently $t_n = \psi_n \circ \tau^{-1}_n \circ \varphi_n$. 
In the same fashion,
\begin{equation*}
\begin{split}
\gs_{n-1} & \tau_n  \varphi_n(  h^1 \otimes  \cdots  \otimes  h^n) 
=\gs_{n-1}\left(S(h^1_{(1)})\cdot\left(h^1_{(2)}h^2_{(1)}\otimes
\cdots\otimes h^1_{(n)}\cdots h^{n-1}_{(2)}h^n\otimes 1\right)\right)
\\
&= S(h^1_{(1)})\cdot \left(h^1_{(2)}h^2_{(1)}\otimes\cdots\otimes h^1_{(n)}\cdots h^{n-1}_{(2)}h^n\right)\\
&= \big(S(h^1_{(1)})h^1_{(2)}h^2\big)\cdot \left(1\otimes\varphi_{n-2}(h^3\otimes\cdots\otimes h^n)\right)\\
&= \big(s_r\epsilon_r(h^1)h^2\big)\cdot \left(1\otimes\varphi_{n-2}(h^3\otimes\cdots\otimes h^n)\right)\\
&=\varphi_{n-1}d_0(h^1\otimes\cdots\otimes h^n).
\end{split}
\end{equation*}
The remaining identities are left to the reader.
\end{proof}

\begin{corollary}
 $C_\bull(\calH)$ is para-cyclic, and cyclic if and only if  $S^2 = \id$.
\end{corollary}

\begin{definition}
In case $S^2 = \id$, let $\calH^\natural_\bull := C_\bull(\calH)$ 
denote the cyclic module equipped with the operators
  \rmref{DualHH}--\rmref{DualCyc}. Its respective Hochschild and
  (periodic) cyclic homology groups are denoted by
  $HH_\bull(\calH)$, $HC_\bull(\calH)$ and $HP_\bull(\calH)$, and
  referred to as {\em dual Hopf-cyclic} homology groups.
\end{definition}

\subsection{The approach via invariants}
\label{invariants}
As is clear from the explicit formulae, the dual cyclic homology is
closely related to the underlying algebra 
structure of the Hopf algebroid. 
To compare this homology to the usual cyclic homology of 
algebras, we use the following approach, which is dual to that of Section \ref{coinvariants}. Remarkably,
it only works in some special cases.

Let $\calH = (\calH_l,\calH_r, S)$ be a Hopf algebroid.
The standard cyclic module of $\calH$ as a $k$-algebra \cite{FeiTsy:AKT} is defined by
$
\calH_\bull^{{\scriptscriptstyle{{\rm alg}}},
  \sharp}:=\calH^{\otimes_k(\bullet +1)},
$
with face maps
\[
d_i(h^0  \otimes   \cdots   \otimes   h^n)  =
\begin{cases}
h^0  \otimes   \cdots    \otimes h^i h^{i+1} \otimes \cdots
  h^n & \mbox{if}
  \ 0 \leq i \leq n-1,
 \\
h^n h^0  \otimes  \cdots   \otimes   h^{n-1} & \mbox{if} \ i = n, \end{cases}
\]
and degeneracies
\[
s_i(h^0  \otimes   \cdots   \otimes   h^n)  = h^0  \otimes   \cdots \otimes   h^i 
\otimes   1  \otimes   h^{i+1}
 \otimes  \cdots   \otimes   h^n \quad \mbox{if} \ 0 \leq i \leq n.  
\]
Finally the cyclic structure is given by
\[
t_n(h^0\otimes\cdots\otimes h^n)=h^n\otimes h^0\otimes\cdots\otimes h^{n-1}.
\]
On the other hand, we have $C_{n+1}(\calH)\in\mathsf{Comod}_\erre(\calH_r)$. 
Recall that underlying the comodule structure is an $(A_r,A_r)$-bimodule, so we can define
\begin{equation*}
\label{suppe}
B_n(\calH):= C_{n+1}(\calH) \otimes_{A_r^{\rm e}} A_r.
\end{equation*}
This space is a quotient of $C_{n+1}(\calH)$ by the $k$-submodule $I_n\subset C_{n+1}(\calH)$
given by
\[
I_n:=\mbox{span}_k\{s_r(a)h^0\otimes\cdots\otimes h^n-h^0\otimes\cdots\otimes h^ns_r(a),~a\in A_r, h^i\in\calH.\}.
\]
One then easily observes that the canonical
projection $\calH_\bull^{{\scriptscriptstyle{{\rm alg}}}, \natural}\to B_\bull(\calH)$ 
equips the latter with the structure of a cyclic module given by the same formulae as above.

By Proposition \ref{inv-comod}, 
we have $C_n(\calH)\cong C_{n+1}(\calH)^{\scriptscriptstyle{\rm inv}}$ via the embedding
\begin{equation}
\label{invariant}
\Psi_{\scriptscriptstyle{\rm inv}}:h_1\otimes\cdots\otimes h_n\mapsto h_1^{(1)}\otimes\cdots\otimes h_n^{(1)}\otimes 
S(h_1^{(2)}\cdots h^{(2)}_n).
\end{equation}
Combined with the canonical projection $C_{n+1}(\calH)\to B_n(\calH)$, 
this leads to a morphism $\overline{\Psi}_{\scriptscriptstyle{\rm inv}}:C_\bull(\calH)\to 
B_\bull(\calH)$. 
Unfortunately, this is not a map of cyclic objects, let alone simplicial objects in general. 
However, there are Hopf algebroids for which this is true, cf.\ \S \ref{etale} for an example.
Also, the left inverse of the embedding above, given as
$\Phi_{\scriptscriptstyle{\rm inv}}:C_{n+1}(\calH)\to C_n(\calH)$, $h^1\otimes\cdots\otimes
h^{n+1} \mapsto h^1\otimes\cdots\otimes h^n\cdot\eps_r (S^{-1} h^{n+1})$ 
does not descend to the quotient $B_n(\calH)$. 
We therefore do not have a commutative diagram as for the coinvariant localisation and the 
cyclic cohomology theory.

\subsection{Hochschild theory with coefficients}
\label{schondreiuhr}
Both the Hochschild cohomology for bialgebroids as well as the dual homology for Hopf algebroids 
are part of a more general theory with coefficients that we now describe.
First, however, we introduce certain resolutions of the base algebras in the categories 
discussed in \S \ref{mod-comod}, 
defined by the left and right bialgebroid structure on $\calH$.

\subsubsection{Cobar resolution}
A straightforward generalisation of 
Theorem A.1.1.3 and Lemma A.1.2.2 
in \cite{Rav:CCASHGOS} to the noncommutative setting (cf.\ Section 2.4 in \cite{Kow:HAATCT})
shows that if $\calH$ is flat as right $A_l$-module \rmref{bimod-lmod}, the category
$\mathsf{Comod}_\elle(\calH_l)$ is abelian and has enough injectives. 
We call a (say, left) $\calH_l$-comodule $N$ {\em cofree} 
if there is a left $A_l$-module $M$ such that $N \cong \calH \otimes_{A_l} M$ as left $\calH_l$-comodules, 
and it is called {\em relative injective} if it is a direct summand in a cofree one.

The {\em cobar} resolution of $A_l$ in the category
$\mathsf{Comod}_\elle(\calH_l)$ 
generalises 
the well-known construction for bialgebras \cite{Doi:HCA} and for commutative bialgebroids in \cite{Rav:CCASHGOS}: 
define the
graded space
\[
\Cobar^n(\calH):=\underbrace{\calH\otimes_{A_l}\cdots\otimes_{A_l}\calH}_{\scriptscriptstyle{n+1
   ~ {\rm times}}},
\]
the tensor product being the one in the category $\mathsf{Mod}(\calH)$. Alternatively, we can view 
this as the cofree left $\calH_l$-comodule generated by $C^\bull(\calH)\in \mathsf{Mod}(\calH)$.
This allows us to view
$\Cobar^\bull(\calH)\in\mathsf{Comod}_\elle(\calH_l)$ by using the left
comultiplication on the first component.
 Introduce the following 
cosimplicial structure on $\Cobar^\bull(\calH)$: first, the coface operators 
$\delta_i':\Cobar^n(\calH)\to\Cobar^{n+1}(\calH)$ are given by
\[
\label{cobar-simpl}
\gd'_i(h^0 \otimes h^1 \otimes \cdots \otimes h^n) = 
\begin{cases}
h^0 \otimes \cdots \otimes \gD_l h^i \otimes \cdots
 \otimes h^n &\mbox{if} \ 
0 \leq i \leq n, 
\\
h^0 \otimes \cdots \otimes h^n \otimes 1 & \mbox{if} \  i = n + 1.
\end{cases} 
\]
The codegeneracies $\sigma_i':\Cobar^n(\calH)\to\Cobar^{n-1}(\calH)$ are:
\[
\gs_i'(h^0 \otimes\cdots \otimes h^n) = h^0 \otimes \cdots \otimes
\eps_l( h^{i+1}) \otimes \cdots \otimes h^n,  \qquad 0 \leq i \leq n-1.
\]
These maps are compatible with the left $\calH_l$-comodule structure
 on $\Cobar^\bull (\calH)$ but {\em not} with the
 left $\calH_l$-module structure.
The left $\calH_l$-coaction on $A_l$ given by the left source map $s_l:A_l\to\calH$ defines a coaugmentation for 
this cosimplicial object in $\mathsf{Comod}_\elle(\calH_l)$,
which yields a cosimplicial resolution of $A_l$: consider the associated cochain complex
\[
A_l\stackrel{s_l}{\longrightarrow} \Cobar^0(\calH)\stackrel{b'}{\longrightarrow}\Cobar^1(\calH)\stackrel{b'}
{\longrightarrow}\ldots,
\]
with differentials
$
b':=\sum_{i=0}^{n+1}(-1)^i\delta_i'.
$
It is easy to check that $b'$ is a morphism of left
$\calH_l$-comodules and that the maps $s^{n-1}: \Cobar^n(\calH) \to  \Cobar^{n-1}(\calH)$ 
given by $h^0 \otimes \cdots \otimes h^n \mapsto
\eps_l(h^0)\cdot h^1 \otimes \cdots \otimes h^n$ and $s^{-1}:
\calH = \Cobar^0(\calH) \to \Cobar^{-1}(\calH) := A_l$, $h \mapsto
\eps_l h$ define a contracting homotopy for the complex
$(\Cobar^\bull(\calH), b')$ over $A_l$, i.e., $s^n \circ b' + b' \circ
s^{n-1} = \id$. In particular, $A_l \stackrel{s_l}{\longrightarrow}
\Cobar^\bull(\calH)$ is a resolution of $A_l$ by cofree (hence relative
injective) left $\calH_l$-comodules: from \rmref{staub} follows
$\ker b'=\{h\in\calH,~\Delta_l(h)= h\otimes 1\} \cong A_l$, hence
exactness in degree zero. 

\subsubsection{The bar resolution}
Analogous to the standard case (see e.g.\ \cite{Wei:AITHA}), 
the bar complex gives 
a resolution of $A_l$ in the category $\mathsf{Mod}(\calH)$ of left
modules over $\calH$. We define 
\[
\mbox{Bar}_n(\calH):=\underbrace{\calH\otimes_{A_r}\cdots\otimes_{A_r}\calH}_{\scriptscriptstyle{n+1~
  {\rm times}}},
\]
where the tensor product is the one in
$\mathsf{Comod}_\erre(\calH_r)$, but we view $\Br_\bull(\calH)$ in $\mathsf{Mod}(\calH)$, 
the left $\calH$-action being given by
left multiplication on the first factor.
 The simplicial structure on this graded $\calH$-module is given by
 the face and degeneracy operators
\begin{equation*}
\begin{array}{rcll}
 d'_i(h^0  \otimes   \cdots   \otimes   h^n) & \!\!\!\!\! =& \!\!\!\!\!
\left\{ \!\!
\begin{array}{l}
h^0  \otimes   \cdots \otimes   h^i h^{i+1}
 \otimes  \cdots   \otimes   h^n \\
h^0  \otimes   \cdots    \otimes  
h^{n-1} \cdot \eps_r (S^{-1} h^n) 
\end{array}\right.  & \!\!\!\!\!\! \begin{array}{l} \mbox{if}
  \ 0 \leq i \leq n-1,  \\ \mbox{if} \ i = n, \end{array} \\
s'_i(h^0  \otimes   \cdots   \otimes   h^n)  & \!\!\!\!\! =& \!\!\!\!\! 
h^0  \otimes   \cdots \otimes   h^i 
\otimes   1  \otimes   h^{i+1}
 \otimes  \cdots   \otimes   h^n, & \!\!\!\!\!\! \begin{array}{l} \phantom{\mbox{if}}
  \ 0 \leq i \leq n. \end{array}  
\end{array}
\end{equation*}
This time these maps are morphisms of left $\calH$-modules, not of comodules.
The augmented bar complex is given by
\[
\ldots \stackrel{d'}{\longrightarrow}
\mbox{Bar}_1(\calH)\stackrel{d'}{\longrightarrow}\mbox{Bar}_0(\calH)
\stackrel{\epsilon_l}{\longrightarrow}A_l,
\]
with chain operator
$
d':=\sum_{i=0}^n(-1)^i d_i'.
$
It is a straightforward check that the bar complex is a contractible
resolution of $A_l$, where the extra degeneracy 
$$
s_{n}:\mbox{Bar}_{n}(\calH)\to\mbox{Bar}_{n+1}(\calH), \quad
s_n(h^0\otimes\cdots\otimes h^{n})
=1\otimes h^0\otimes\cdots\otimes
h^{n}
$$ 
for $n\geq 0$ and $s_{-1}: =t_l$ provide the contracting 
homotopies. 
Moreover, $\Br_\bull(\calH)$ is $\calH$-projective 
if $\calH$ is $A_r$-projective with respect to the left $A_r$-module
structure {\rm \rmref{H-comod}}, 
and in this case $\Br_\bull(\calH)
\stackrel{\eps_l}{\longrightarrow} A_l$ is a projective resolution of
$A_l$ in the category of left $\calH$-modules.

\subsubsection{The Hochschild theory as derived functors}
The main point is now:
\begin{theorem} 
\label{hoch-tor}
For any Hopf algebroid $\calH$ that is flat as right $A_l$-module \rmref{bimod-lmod}, 
there are natural isomorphisms 
$$
HH^\bull(\calH)\cong\Cotor^\bull_\calH(A_l,A_l),
$$
and if $\calH$ is projective as left $A_r$-module \rmref{H-comod},
$$
HH_\bull(\calH)\cong \Tor^\calH_\bull(A_r,A_l).
$$
\end{theorem}
\begin{proof}
Recall (cf.\ \cite{EilMoo:HAFICCPAIDF}) that $\Cotor^\bull_\calH(A_l,A_l)$ is the right derived
functor of the left cotensor product
$A_l\bx_{\calH_l}-:\mathsf{Comod}_\elle(\calH_l)\to \mathsf{Mod}(k)$,
where $A_l$ is seen as right $\calH_l$-comodule by means of $t_l: A_l
\to \calH_l$. 
That this derived 
functor can be computed by relative injective resolutions (like the
 cobar complex) follows from a straightforward generalisation of Lemmata A.1.2.8 and A.1.2.9 in \cite{Rav:CCASHGOS} to the noncommutative case. A little thought reveals that the space
$A_l\bx_{\calH_l}\Cobar^\bull(\calH)$ can be alternatively expressed
as $\{ h \otimes w \in \Cobar^\bull(\calH) \mid h \otimes w =
1_{\calH} \otimes s^l(\eps_l h)w \}$, where $w \in
C^\bull(\calH)$.
Using then the isomorphism
$f: A_l\bx_{\calH_l}\Cobar^\bull(\calH)\cong C^\bull(\calH)$ given by
\rmref{staub}, 
one easily checks that $b \circ f = f \circ (\id_{A_l} \otimes b')$, i.e.,
the induced 
differential coincides 
with that of the Hochschild complex. 
This proves the first 
isomorphism.

To prove the second isomorphism, 
use the bar resolution in $\mathsf{Mod}(\calH)$ to compute the left derived functor 
of $A_r\otimes_\calH -:\mathsf{Mod}(\calH)\to\mathsf{Mod}(k)$. 
% By
% Proposition \ref{ProjAntipCor} 
We have 
$A_r\otimes_\calH\mbox{Bar}_\bull(\calH)\cong C_\bull(\calH)$, and one
easily sees that
the differentials coincide.
\end{proof}
\begin{remark}
The definition of the Hochschild cohomology depends solely on the underlying left bialgebroid structure 
of $\calH$. This is because for any left \mbox{bialgebroid} $\calH_l$, the
base algebra $A_l$ carries canonical left and right 
$\calH_l$-coactions given by left source and target maps, respectively.
By contrast, the definition of dual Hochschild homology {\em does} depend
on the Hopf algebroid structure: although the base algebra $A_r$ of the underlying right bialgebroid
is naturally a right $\calH$-module, 
there is {\it a priori} no canonical {\em left} $\calH$-module structure defined on it without the antipode. 
\end{remark}
\subsubsection{Coefficients}
Having identified Hochschild homology and cohomology as derived
functors, we can assume a different perspective and put coefficients in:
for $M \in \mathsf{Comod}_\erre(\calH_l)$ with coaction $\gD_M$ and $\calH$ flat as right $A_l$-module \rmref{bimod-lmod}, define
\[
H^\bull(\calH, M):=\Cotor^\bull_\calH(M,A_l).
\]
If $M$ is projective as a right $A_l$-module, one may use the cobar complex to compute these groups:
using the isomorphism 
$$
M\bx_{\calH_l}\Cobar^\bull(\calH) \cong M \otimes_{A_l} C^\bull(\calH), \quad m
\otimes h \otimes w \mapsto m \cdot \eps_l(h) \otimes w,
$$
with inverse
$m \otimes w \mapsto \gD_M(m) \otimes w$, where $w \in C^\bull(\calH)$, 
as well as 
the isomorphism
\begin{equation*}
\begin{split}
M \bx_{\calH_l} &\Cobar^\bull(\calH) \cong \\
& \quad \cong \{ m \otimes h \otimes w \in M \otimes_{A_l}
\Cobar^\bull(\calH_l) \mid 
m \otimes h \otimes w = \gD_M(m \cdot \eps_l h) \otimes w\},
\end{split}
\end{equation*}
similarly as above,
one obtains the explicit corresponding
complex with coefficients in $M$.

Likewise, we put for $N\in\mathsf{Mod}(\calH^{\op})$
\[
H_\bull(\calH, N):= \Tor^\calH_\bull(N,A_l).
\]
If $\calH$ is projective as left $A_r$-module \rmref{H-comod}, one may use the bar resolution to write
down the explicit complex computing these groups.

\subsection{The case of commutative and cocommutative Hopf algebroids}
\label{baldvier}
For commutative and cocommutative Hopf algebroids, one of the
respective two
cyclic theories is particularly simple to calculate 
in terms of the associated Hochschild theory. 
This phenomenon is known
for Hopf algebras, cf.\ \cite[Thm.\ 4.1]{KhaRan:ANCMFHA}, 
and originated with Karoubi's computation of the cyclic homology of
$k[G]$ in \cite{Kar:HCDGEDA}, 
where $G$ is a discrete group.
 
In a commutative Hopf algebroid, the underlying left bialgebroid may
serve to define the right bialgebroid structure by means of the prescriptions  
$A_r := A_l$, $s_r := t_l$, $t_r := s_l$, $\gD_r : =\gD_l$, and $\eps_r
:= \eps_l$,
recovering the notion of Hopf algebroids in \cite{Rav:CCASHGOS}.
On the other hand, cocommutativity for Hopf
algebroids is defined as the cocommutativity of the underlying left
bialgebroid $\calH_l$ (which by \rmref{bleistift} implies
cocommutativity for $\calH_r$ as well) and only makes sense 
for commutative $A_l=A_r$ for which $s_l=t_l$ as well as $s_r=t_r$. 

\begin{proposition}\hspace{4cm}
\label{bar-cyclic}
\begin{enumerate}
\item
Let $\calH$ be a commutative Hopf algebroid with
$A_r = A_l$, $s_r = t_l$, $t_r = s_l$, $\gD_r =\gD_l$, and $\eps_r
= \eps_l$.
Then $\Cobar^\bull(\calH)\in\mathsf{Comod}_\elle(\calH_l)$ is a para-cocyclic object 
by means of the cocyclic operator
\[
\tau'_n(h^0\otimes\cdots\otimes h^n)=h^0\cdot(1\otimes\tau_n(h^1\otimes\cdots\otimes h^n)),
\]
where on the right hand side the monoidal structure of
$\mathsf{Mod}(\calH)$ is used.
\item
Let $\calH$ be a cocommutative Hopf algebroid over commutative base
algebra $A$ with invertible antipode $S$. Then $\Br_\bull (\calH)$  is a
para-cyclic $\calH$-module with cyclic operator
\[
t'_n(h_0\otimes\cdots\otimes h_n)=h_0h_1^{(2)}\cdots h_n^{(2)}\otimes 
t_n\big(h_1^{(1)}\otimes\cdots\otimes h_n^{(1)}\big).
\]
\end{enumerate}
In both cases one obtains
cocyclic resp.\ cyclic structures if and only if $S^2=\id$.
\end{proposition}

\begin{proof}
{\it i}) 
Although we view the cobar complex as 
cosimplicial object in $\mathsf{Comod}_\elle(\calH_l)$, it has a natural left
$\calH$-module structure as in \rmref{mod-tensor} from which 
it is also immediate that $\tau'$ is a morphism of graded
left $\calH_l$-comodules. 
Let us now show that $\tau'$ is para-cocyclic:
from the explicit formula \eqref{CoACOp} of the cocyclic operator
$\tau$, one easily shows by induction that
\[
h\cdot\tau^j_n(h^1\otimes\cdots\otimes h^n)=\tau^j_n\left(h^1\otimes\cdots\otimes h^jS^{-1}h\otimes
\cdots\otimes h^n\right),
\]
for all $1\leq j \leq n$. With this equation we can now compute
\[
\begin{split}
{\tau'}^{n+1}_n(h^0\otimes\cdots\otimes
h^n)&={\tau'}_n^n\left(h^0\cdot
(1\otimes\tau_n(h^1\otimes\cdots\otimes h^n))\right)\\
&={\tau'}_n^n(h^0_{(1)}\otimes h^0_{(2)}\cdot\tau_n(h^1\otimes\cdots\otimes h^n))\\
&={\tau'}_n^n(h^0_{(1)}\otimes\tau_n(h^1S^{-1} h^0_{(2)}\otimes\cdots\otimes h^n))\\
&={\tau'}_n^{n-1}(h^0_{(1)}\otimes h^0_{(2)}\cdot\tau^2_n(h^1S^{-1} h^0_{(3)}\otimes\cdots\otimes h^n))\\
&={\tau'}_n^{n-1}(h^0_{(1)}\otimes\tau^2_n(h^1S^{-1} h^0_{(3)}\otimes
h^2S^{-1} h^0_{(2)}\otimes\cdots\otimes h^n))\\
&\hspace{3cm}\vdots\\
&=h^0_{(1)}\otimes h^0_{(2)}\cdot \tau^{n+1}_n\big(h^1S^{-1}
h^0_{(n+2)} \otimes \cdots\otimes h^nS^{-1} h^0_{(3)}\big)\\
&= h^0_{(1)}\otimes h^0_{(2)}\cdot\big(S^{2}(h^1S^{-1}h^0_{(n+2)})\otimes\cdots 
       \otimes S^{2}(h^nS^{-1}h^0_{(3)})\big)\\
&=h^0\otimes S^2 h^1\otimes\cdots\otimes S^2 h^n.
\end{split}
\]
The last equality is 
verified by writing out the expression and using the left co\-mon\-oid identities.
This proves that $\tau'$ generates an action of the cyclic
groups if and only if $S^2=\id$. 
The remaining cocyclic identities, compatibility with the $\delta'_i$ and $\sigma'_i$ that is, are easy to verify.

{\it ii}) Since $s_r = t_r$, the space
$C_\bull(\calH)$ carries a left $\calH_r$-coaction given by
$\Delta_r(h_1 \otimes \cdots \otimes h_n) := h_1^{(2)}\cdots h_n^{(2)}\otimes h_1^{(1)}\otimes\cdots\otimes
h_n^{(1)}$, which appears in the expression of the cyclic
operator. One has 
\[
\begin{split}
\Delta_r(t_n(h_1\otimes\cdots\otimes h_n))
&=S^{-1}\big({h_1}_{(2)} \cdots  {h_{n-1}}_{(2)}
      {h_n}_{(1)}\big) {h_1}^{(2)}_{(1)}  \cdots
{h_{n-1}}^{(2)}_{(1)} \\
&\qquad\otimes S^{-1}(h_{1(3)}\cdots h_{n-1(3)}{h_n}_{(2)}) \otimes
{h_1}_{(1)}^{(1)}\otimes\cdots
\otimes {h_{n-1}}_{(1)}^{(1)}\\
&=S^{-1}{h_n}_{(1)}\otimes t_n(h_1\otimes\cdots\otimes h_{n-1}
\otimes {h_n}_{(2)}).
\end{split}
\]
With this we now compute
\[
\begin{split}
{t'}_{n}^{n+1}(&h_0\otimes\cdots\otimes
h_n)
={t'}^n_n\left(h_0h_1^{(2)}\cdots h_n^{(2)}\otimes
t_n(h_1^{(1)}\otimes\cdots\otimes h_n^{(1)})\right)\\
&={t'}_n^{n-1}\left(h_0h_1^{(2)}\cdots h_{n-1}^{(2)}
h_n^{(2)}S^{-1} {h_n}^{(1)}_{(1)} \otimes
t^2_n(h_1^{(1)}\otimes\cdots\otimes h_{n-1}^{(1)} \otimes
{h_n}^{(1)}_{(2)})\right)\\
&={t'}_n^{n-1}\left(h_0h_1^{(2)}\cdots h_{n-1}^{(2)}
t_r\eps_r(h_n^{(1)})\otimes t^2_n(h_1^{(1)}\otimes\cdots\otimes
h_n^{(2)})\right)\\
&={t'}_n^{n-1}\left(h_0h_1^{(2)}\cdots h_{n-1}^{(2)} \otimes
t^2_n(h_1^{(1)}\otimes\cdots\otimes h_{n-1}^{(1)} \otimes h_n)\right)\\
&\hspace{3cm}\vdots\\
&=h_0\otimes t^{n+1}_n(h_1\otimes\cdots\otimes h_n)\\
&=h_0\otimes S^{-2} h_1\otimes\cdots\otimes S^{-2} h_n,\\
\end{split}
\]
where the vertical dots mean the $(n-1)$-fold repetition of the previous
manipulation. To obtain the fourth line we have used $s_r = t_r$ and
\[
a_1 \cdot t_n(h_1\otimes\cdots\otimes h_n) \cdot a_2 
= t_n (h_1\otimes\cdots\otimes a_2 \cdot h_n \cdot a_1), \qquad a_1,
a_2 \in A_r,
\]
with respect to the respective $(A_r,A_r)$-bimodule structure
\rmref{H-comod}, as follows from \eqref{DualCyc} and 
by exploiting the Takeuchi condition of the left coproduct on page \pageref{taki}.
\end{proof}
\begin{theorem}
\label{cocomm-cycl}
Let $\calH$ be a Hopf algebroid with involutive antipode.
\begin{enumerate}
\item
When $\calH$ is commutative with 
$A_l=A_r$, 
$s_l = t_r$,
$t_l = s_r$, $\Delta_l=\Delta_r$, $\eps_l = \eps_r$, and flat as right $A_l$-module \rmref{bimod-lmod}, one
has
\[
HC^\bull(\calH)\cong\underset{{i \geq 0}}{\textstyle{\bigoplus}} HH^{\bull-2i}(\calH);
\]
\item
when $\calH$ is cocommutative and projective as left $A_r$-module \rmref{H-comod}, there is a natural isomorphism
\[
HC_\bull(\calH) \cong \underset{{i \geq 0}}{\textstyle{\bigoplus}}
HH_{\bull-2i}(\calH).
\]
\end{enumerate}
\end{theorem}
\begin{proof}
{\it i}) Consider 
(cf.\ \cite{Tsy:HOMLAORATHH, Lod:CH}) 
Tsygan's 
double complex 
$CC^\bull\left(\Cobar^\bull(\calH)\right)$ of the cocyclic left $\calH_l$-comodule
$\Cobar^\bull(\calH)$. 
Since $\Cobar^\bull(\calH)$ is a resolution of $A_l$ in the category of left $\calH_l$-comodules,
the double complex 
$CC^\bull\left(\Cobar^\bull(\calH)\right)$ is a resolution (in the sense of hypercohomology, see \cite[Sect.\ 5.7]{Wei:AITHA}) of the cochain complex
\[
A_{l \bull}:\quad 0\rightarrow A_l\rightarrow 0\rightarrow A_l\rightarrow 0\rightarrow\ldots,
\]
with the first $0$ in degree zero. From the explicit form of the
cyclic 
operator in Proposition \ref{bar-cyclic}, one easily
observes that the natural isomorphism
\[
A_l\bx_{\calH_l} \Cobar^\bull(\calH)\cong C^\bull(\calH)
\]
of \rmref{staub} 
is one of cocyclic $k$-modules. This identifies cyclic cohomology of $\calH$ as 
the hyper-derived $\Cotor$, written ${\rm\bf Cotor}$, of $A_l$ with
values in the chain complex 
\nolinebreak
$A_{l\bull}$:
\[
HC^\bull(\calH)={\rm\bf Cotor}^\bull(A_l,A_{l\bull}).
\]
Clearly, any resolution for $A_l$ defines a resolution of the complex
$A_{l\bull}$ 
by putting $0$ in the even degree columns,
and therefore ${\rm\bf Cotor}^\bull(A_l,A_{l\bull})=\bigoplus_{i\geq 0}HH^{\bull-2i}(\calH)$.

{\it ii}) is proved in very much the same fashion, 
this time identifying $HC_\bull(\calH)={\rm\bf Tor}_\bull(A_r,A_{l\bull})$, 
the hyper-derived functors 
of $A_r\otimes_\calH-:\mathsf{Mod}(\calH)\to\mathsf{Mod}(k)$.
\end{proof}
\section{Examples}
\label{Examples}
In this section we discuss examples of Hopf algebroids and compute their cyclic homology and cohomology groups.
\subsection{The enveloping Hopf algebroid of an algebra}
\label{pg}
A very simple example of a Hopf algebroid is given by the 
enveloping algebra $\Ae = A \otimes_k \Aop$ of an arbitrary (unital) $k$-algebra $A$. It is a left
bialgebroid over $A$ by means of the structure maps
$s_l(a) := a \otimes_k 1$, $t_l(b) := 1 \otimes_k b$, 
$\gD_l(a \otimes b) := (a \otimes_k 1) \otimes_A (1 \otimes_k b)$,
$\eps_l(a \otimes_k b) := ab$, 
and a right bialgebroid over $\Aop$ by means of $s_r(b) := 1 \otimes_k b$, $t_r(a) := a \otimes_k 1$, 
$\gD_r(a \otimes b) := (1 \otimes_k a) \otimes_A (b \otimes_k 1)$,
$\eps_r(a \otimes_k b) := ba$. With the antipode $S(a \otimes_k b) := b
\otimes_k a$, these data coalesce to a Hopf algebroid. 
\begin{proposition} 
\label{pg-cycl}
Let $A$ be a $k$-algebra and $\Ae$ its enveloping algebra.
\begin{enumerate}
\item
The Hopf-cyclic cohomology of $\Ae$ is trivial, i.e.,
\[
HC^\bull(\Ae)=\begin{cases} k&\mbox{if} \ {\scriptstyle{\bullet}}=0, \\0&\mbox{else.}\end{cases}
\]
\item
The dual Hopf-cyclic homology of $\Ae$ equals the cyclic
  homology of the \mbox{$k$-algebra $A$:}
\[
HC_\bull(\Ae)=HC_\bull^{\scriptscriptstyle{\rm alg}}(A).
\]
\end{enumerate}
\end{proposition} 
\begin{proof}
{\it i}) was proved in \cite{ConMos:DCCAHASITG}. It actually also follows by 
cyclic duality from {\it ii}). To prove {\it ii}), one just writes out the 
cyclic object associated to $\Ae$; 
it is exactly equal to the cyclic object
$A^{{\scriptscriptstyle{{\rm alg}}}, \natural}$
associated to the algebra $A$.
\end{proof}
Recall 
that, when passing to the periodic theory, 
the right hand side in {\it ii}) yields the noncommutative generalisation of
classical de Rham cohomology, cf.\ \cite{Con:NCDG}. 
\subsection{Etale groupoids}
\label{etale}
{\em Notation.}
Let $E$ and $F$ be vector
 bundles (or more generally, $c$-soft sheaves of vector spaces) over
 two manifolds 
 $X$ and $Y$, respectively. 
Suppose that $f:X\to Y$ is an \'etale map and $\alpha_f:E\cong f^*F$
 an isomorphism of vector bundles over $X$. 
Then the push-forward (or fibre sum) 
of $f$, denoted $f_*:\Gamma_c(X,E)\to\Gamma_c(Y,F)$ is defined by
\[
(f_*s)(y)=\sum_{f(x)=y}\alpha_f(s(x)), 
\]
where $x\in X$, $y\in Y$ and $s\in\Gamma_c(X,E)$. 
This construction is functorial in the obvious sense.

Another class of examples of Hopf algebroids comes from \'etale
groupoids, as essentially already noted
in \cite{Mrc:THAOFOEGATPME,Mrc:ODBEGAHA} (a different way to obtain a (topological) Hopf algebroid from an \'etale groupoid is described in \cite{KamTan:HAASCC}).
A groupoid $\G$, to start with, is a small category in which each arrow is invertible. We denote the space of objects
by $M$ and the space of arrows by $\G$. The structure maps can be organised in the following diagram:
\vspace*{-.2cm}
\begin{equation*}
\xymatrix {
\G_2 \ar^m[r] & \, \G  \ar^{i}[r] & \, \G  \ar^s@<0.5ex>[r]\ar_t@<-0.5ex>[r] & \, M  \ar^u[r] & \, \G  }.
\vspace*{-.2cm}
\end{equation*}
Here $u$ is the unit map, $s$ and $t$ are the source and target of
arrows in $\G$, $i$ is the inversion and $m$ 
the multiplication defined on the space of composable arrows:
\[
\G_2:=\G \qttr \times s t {} {M} \G=\{(g_1,g_2)\in\G\times \G,~s(g_1)=t(g_2)\}.
\]
A Lie groupoid is a groupoid $\G\rightrightarrows M$ for which $\G$ and $M$ are smooth manifolds 
and all structure maps listed above are smooth.
In an \'etale groupoid, these are assumed to be local
diffeomorphisms. For simplicity of exposition, we will assume that
$\G$ is Hausdorff.

Associated to an \'etale groupoid is its convolution algebra $C^\infty_c(\G)$ with product
\begin{equation}
\label{convolution}
(f_1*f_2)(g)=\sum_{g_1g_2=g}f_1(g_1)f_2(g_2),
\end{equation}
where $f_1,f_2\in C^\infty_c(\G)$ and $g,g_1,g_2\in\G$. We shall equip this noncommutative algebra 
with the structure of a Hopf algebroid in the following way: the base algebra is given by the 
commutative algebra $C_c^\infty(M)$ and we put $s_l=t_l=s_r=t_r=u_*$, the push-forward along the inclusion 
of the units. We are left with two $\cinfc M$-actions on $\cinfc \G$ by
left and right multiplication with respect to which we define the
tensor products $\otimes^{ll},~\otimes^{rl}$ 
and $\otimes^{rr}$. The formula 
\[
\Omega(f_1\otimes f_2)(g_1,g_2):=f_1(g_1)f_2(g_2),
\]
with $f_1,f_2\in C^\infty_c(\G)$ and $g_1,g_2\in\G$, induces isomorphisms
\begin{equation}
\label{iso-gpd}
\begin{split}
\Omega_{s,t}:C^\infty_c(\G)\otimes^{rl}C^\infty_c(\G)&\stackrel{\cong}{\longrightarrow}C^\infty_c(\G
\, \qttr
\times s t {}{M} \, \G) =  C^\infty_c(\G_2),\\
\Omega_{t,t}:C^\infty_c(\G)\otimes^{ll}C^\infty_c(\G)&\stackrel{\cong}{\longrightarrow}C^\infty_c(\G
\, \qttr \times t t {}{M} \, \G)\\
\Omega_{s,s}:C^\infty_c(\G)\otimes^{rr}C^\infty_c(\G)&\stackrel{\cong}{\longrightarrow}C^\infty_c(\G
\, \qttr \times s s {}{M} \, \G)\\
%\Omega_{t,s}:C^\infty_c(\G)\otimes^{lr}C^\infty_c(\G)&\stackrel{\cong}
%{\longrightarrow}C^\infty_c(\G\qttr \times t s {}{M} \G).
\end{split}
\end{equation}
That these maps are indeed isomorphisms
can be derived from a more general result on sheaves in
\cite[p.\ 271]{Mrc:ODBEGAHA}. With this, we define the left coproduct 
$\Delta_l:C^\infty_c(\G)\to
C^\infty_c(\G)\otimes^{ll}C^\infty_c(\G)\cong C^\infty_c(\G \, \qttr
\times t t {}{M} \, \G)$
by the formula
\[
\Delta_lf(g_1,g_2):=\begin{cases} f(g_1)&\mbox{if}~g_1=g_2,\\ 0&\mbox{else}.\end{cases}
\]
Alternatively, this is simply the push-forward along the diagonal
inclusion $d^l:\G\to \G \, \qttr \times t t {}{M} \, \G$, ~$g\mapsto (g,g)$. 
In a similar fashion, the right coproduct is defined as $\Delta_r=d^r_*$, where 
$d^r:\G\to\G \, \qttr \times s s {}{M} \, \G$ is again the diagonal.
Left and right counit are defined as the push-forward along the target resp.\ source map:
\[
%\begin{split}
(\epsilon_lf)(x):=\sum_{t(g)=x}f(g) \qquad \mbox{and} \qquad
(\epsilon_rf)(x):=\sum_{s(g)=x}f(g).
%\end{split}
\]
Finally, the antipode $S:C^\infty_c(\G)\to C^\infty_c(\G)$ is given by the groupoid inversion:
\[
(Sf)(g):=f(g^{-1}).
\]
\begin{proposition}
When $M$ is compact, $C^\infty_c(\G)$ is a Hopf
algebroid over $\cinf M$ by means of
the structure maps mentioned above.
\end{proposition}
\begin{proof}
We remark that compactness of $M$ is needed in order to make both
algebras $C^\infty_c(M)$ and $C^\infty_c(\G)$ unital. 
The fact that $(C^\infty_c(\G),C^\infty(M),\Delta_l,\eps_l)$ 
is a left bialgebroid 
having an antipode $S$ with certain properties was already shown in \cite[Prop.\
  2.5]{Mrc:ODBEGAHA}. The right bialgebroid structure follows at once by replacing $\G$ by its opposite $\G^{\op}$.
It remains to verify the Hopf algebroid axioms in which left and right
bialgebroid structures are intertwined: for example,
twisted coassociativity is obvious. As for the second
identity in \rmref{TwAp}, let $f\in C^\infty_c(\G)$ and compute
\[
\begin{split}
(f_{(1)}*S(f_{(2)}))(g)&=\sum_{g_1g_2=g}f_{(1)}(g_1)f_{(2)}(g_2^{-1})
  \\
&=\sum_{g_1g_1^{-1}=g}f(g_1)\\
&=\begin{cases} 
\sum_{t(g_1) = x} f(g_1)  \ & \hbox{if} \ g = 1_x \ \hbox{for some} \ x \in M, \\
         0  \quad & \mbox{else}
\end{cases}\\
&=(s_l\epsilon_lf)(g).
\end{split}
\]
The remaining identities in Definition
\ref{HAlgd} are left to the reader.
\end{proof}
\subsubsection{Cyclic cohomology}
The Hopf-cyclic cohomology of this example is easily computed:
\begin{proposition}
The Hopf-Hochschild cohomology of $C^\infty_c(\G)$ is trivial
except in \mbox{degree 0,} i.e.,
\[
HH^\bull(C^\infty_c(\G))\cong
\begin{cases}
\cinf M &\bullet=0,\\
0&\mbox{else}.
\end{cases}
\]
Hence, for the (periodic) Hopf-cyclic cohomology of $C^\infty_c(\G)$
one has
$$
HP^0(\cinfc \G) \cong \cinf M, \qqquad HP^1(\cinfc \G) \cong 0.
$$
\end{proposition}
\begin{proof}
Generalising a construction in \cite{Cra:CCOHA} for group algebras,
define the following periodic resolution of $\cinf M$ by cofree (left)
$\cinfc \G$-comodules:
$$
I: \quad 0 \longrightarrow \cinf M 
\stackrel{u_*}{\longrightarrow} \cinfc \G 
\stackrel{\ga}{\longrightarrow} \cinfc \G
\stackrel{\gb}{\longrightarrow} \cinfc \G 
\stackrel{\ga}{\longrightarrow} \ldots,
$$
where $\ga(f) := f - f|_M$ and $\gb(f) := f|_M$. According to
 Theorem \ref{hoch-tor}, the Hochschild
cohomology groups are 
computed by $\cinf M
\bx_{\cinfc \G} I$, i.e.\ by means of 
$0 \longrightarrow \cinf M 
\stackrel{0}{\longrightarrow} \cinf M 
\stackrel{\id}{\longrightarrow} \cinf M
\stackrel{0}{\longrightarrow} \cinf M 
\stackrel{\id}{\longrightarrow} \ldots$. Then one has $HH^n(\cinfc
\G)=\cinf M$ for $n=0$ and zero in all other cases. Applying an
$SBI$ sequence argument, the second statement follows.
\end{proof}

\subsubsection{Cyclic homology and groupoid homology}
For the dual homology theory, consider the nerve $\G_\bull :=
\{\G_n\}_{n \geq 0}$ of $\G$ defined as usual
\[
\G_0 := M, \qquad  \G_n :=\{(g_1,\ldots,g_n)\in \G^{\times n},~s(g_i)=t(g_{i+1}),~1\leq i\leq 0\},
\]
equipped with face operators $d_i:\G_n\to\G_{n-1}$ defined by
\begin{subequations}
\begin{equation}
\label{NervD}
d_i(g_1,\ldots,g_n)=
\begin{cases} 
(g_2,\ldots,g_n) & \mbox{if} \ i=0, \\ 
(g_1,\ldots,g_ig_{i+1},\ldots,g_n) &  \mbox{if} \ 1\leq i\leq n-1, \\ 
(g_1,\ldots,g_{n-1}) & \mbox{if} \ i=n,
\end{cases}
\end{equation}
whereas $d_0, d_1: \G_1 \to \G_0$ are given by source and target map,
respectively. Equipped with degeneracies $s_i: \G_n \to
\G_{n+1}$ given as
\begin{equation}
\label{NervS}
s_i(g_1, \ldots, g_n)  = \left\{
\begin{array}{ll}
(1_{t(g_1)}, g_1, \ldots, g_n) 
& \mbox{if} \ i = 0, \\
(g_1, \ldots, g_i, 1_{s(g_i)}, g_{i+1}, \ldots, g_n)  
 & \mbox{if} \ 1 \leq i \leq n,
\end{array}\right.
\end{equation}
\end{subequations}
the nerve is a simplicial manifold whose geometric realisation is a model
for the classifying space $B\G$. Denote by $\tau_n:\G_n\to M$ the
map $\tau_n(g_1,\ldots,g_n)=t(g_1)$. 
Given a representation $E$ of $\G$, that is a vector bundle $E$ over $M$ equipped with an action of $\G$,  define
\[
C_n^{d}(\G;E):=\Gamma_c(\G_n,\tau_n^*E).
\]
This space of chains carries a differential $\partial:C_n^{d}(\G;E)\to C^{d}_{n-1}(\G;E)$ given by
$$
\partial:=\sum_{i=0}^n(-1)^i(d_i)_*,
$$
where the push-forward is defined with respect to the tautological isomorphisms
$\tau_n^*E\cong d_i^* \tau_{n-1}^*E$ for $1\leq i\leq n$ 
and $g_1: E_{s(g_1)}\to E_{t(g_1)}$ is the isomorphism
$\tau_n^*E\cong d_0^*\tau_{n-1}^*E$ at $(g_1,\ldots,g_n)\in\G_n$. 
This defines a differential because of the simplicial identities of
the underlying face maps, and its homology 
is the groupoid homology of $E$, denoted as
$H^{d}_\bull(\G,E)$, cf.\ \cite{CraMoe:AHTFEG}.

\begin{theorem}
\label{cycl-etale}
Let $\G$ be an \'etale groupoid. There are natural isomorphisms
\[
\begin{split}
HH_\bull(C^\infty_c(\G))&\cong H^d_\bull(\G,\underline{\C}),\\
HC_\bull(C^\infty_c(\G))&\cong  \underset{n \geq
  0}{\textstyle{\bigoplus}} H^d_{\bull+2n}(\G,\underline{\C}).
\end{split}
\]
\end{theorem}
\begin{proof}
The obvious generalisation of the isomorphism \eqref{iso-gpd} to
higher degrees yields
\begin{equation}
\label{tonsille}
\gO^n_{s,t}:
C_\bull(C^\infty_c(\G))\stackrel{\cong}{\longrightarrow}
C_c^\infty(\G_\bull)
=C^d_\bull(\G,\underline{\C}),
\end{equation}
where $\underline{\C}$ denotes the trivial representation on the line bundle $M\times \C$.
To identify the differential, remark that the convolution product \eqref{convolution} is simply 
the push-forward along the multiplication map $m:\G_2\to\G$, and right
and left counit the push-forwards along source and target maps, i.e.,
$\eps_r = s_*$, $\eps_r \circ S^{-1} = \eps_l = t_*$. It is then a
straightforward check that the isomorphism \rmref{tonsille}
intertwines the simplicial maps \rmref{DualHH} with
the push-forwards along the face 
operators \eqref{NervD} on $\G_\bull$, and this identifies the differential
with the groupoid homology differential $\partial$. This proves the first assertion. The second follows 
from Theorem \ref{cocomm-cycl}.
\end{proof}
\begin{remark}
In particular, the isomorphism \rmref{tonsille} is an isomorphism of
cyclic modules: the operators $\tilde{t}_n: \G_n \to \G_n$,
\begin{equation}
\label{NervT}
\tilde{t}_n(g_1, \ldots, g_n) :=((g_1 g_2 \cdots g_n)^{-1}, g_1, \ldots, g_{n-1})
\end{equation}
for $n \geq 2$, and $t_1(g) := \rcp g$, $t_0 := \id_{\G_0}$
define a cyclic operation on $\G_\bull$, such that $\cinfc \G$
together with the push-forwards of \rmref{NervD}, \rmref{NervS}, and
\rmref{NervT} becomes a cyclic module. One then has with
respect to the dual Hopf-cyclic operator \rmref{DualCyc}:
\begin{equation*}
\begin{split}
(\gO^n_{s,t} \, & t_n (f^1 \otimes^{rl} \cdots
  \otimes^{rl} f^n))(g_1, \ldots,
  g_{n}) = \\
&= \gO^n_{s,t} \big( \rcp S( f^{1}_{(2)} * \cdots * f^{n-1}_{(2)} * f^n)
  \otimes^{rl}  f^1_{(1)}    
\otimes^{rl} \cdots
  \otimes^{rl} f^{n-1}_{(1)}\big)(g_1, \ldots,
  g_{n}) \\
%&= (f^{1}_{(2)} * \cdots * f^{n-1}_{(2)} * f^n)(g_1^{-1})f^{1}_{(1)}(g_2)
 % \cdots f^{n-1}_{(1)}(g_n) \\ 
&= \sum_{g_1^{-1} = g'_1 \cdots  g'_n } f^{1}_{(2)}(g'_1) \cdots
  f^{n-1}_{(2)}(g'_{n-1}) f^n(g'_n) f^{1}_{(1)}(g_2) \cdots f^{n-1}_{(1)}(g_n) \\ 
& = f^1(g_2) \cdots f^{n-1}(g_n) f^n((g_1 \cdots g_n)^{-1}),
\end{split}
\end{equation*}
and this is exactly the push-forward of $\tilde{t}_n$.
\end{remark}
\begin{remark}
The first isomorphism of the theorem above readily generalises as
follows: let $E$ be a representation of $\G$. 
Then $\mathcal{E}:=\Gamma_c(M,E)$ is a module over $C^\infty_c(\G)$ by the action
\[
(f\cdot \varphi)(x)=\sum_{t(g)=x}f(g)\varphi(s(g)),
\]
where $f\in C^\infty_c(\G)$ and $\varphi\in\Gamma_c(M,E)$. With this module, we have
\[
H_\bull(C^\infty_c(\G);\mathcal{E})\cong H_\bull^d(\G,E).
\]
\end{remark}
\begin{remark}
Analogously as in group theory, a little computation reveals that the
Hopf-Galois map from Lemma \ref{hopf-galois}
%\rmref{Hin} 
and its inverse 
%\rmref{Her} 
are (via the isomorphisms
\rmref{iso-gpd}) the push-forwards of the following maps on the groupoid
level:
$$
\tilde{\varphi}_n: \G_n \to \G^n, \quad (g_1, \ldots, g_n) \mapsto (g_1,
g_1g_2, \ldots, g_1g_2 \cdots g_n),
$$
where $\G^n :=\{(g_1,\ldots,g_n)\in \G^{\times
  n},~t(g_i)=t(g_{i+1}),~1\leq i\leq 0\}$, with inverse
$$
\tilde{\psi}_n: \G^n \to \G_n, \quad ({g}_1, \ldots, {g}_n)
\mapsto 
({g}_1, {g}^{-1}_1 {g}_2, \ldots, {g}^{-1}_{n-1} {g}_n).
$$
\end{remark}
\subsubsection{Relation with the computations of Brylinski-Nistor and Crainic} 
In \cite{BryNis:CCOEG,Cra:CCOEGTGC} the cyclic homology of $C^\infty_c(\G)$ as an
algebra, 
i.e.\ not as a Hopf algebroid, was computed.
Let us show how the present result fits into that computation. 
A fundamental tool in the papers mentioned above was 
the ``reduction to loops''
\begin{equation}
\label{red-loops}
C^\infty_c(\G)^{{\scriptscriptstyle{{\rm alg}}}, \natural}_n\to \Gamma_c\left(B_n,\tau_n^{-1}\calC^\infty_{M^{\times(n+1)}}\right),
\end{equation}
where on the left hand side we have the usual cyclic object associated
to an algebra 
(but using topological tensor products).
The space $B_n$ above is the so-called higher Burghelea space of
closed strings of $n+1$ composable arrows
\[
B_n:=\{(g_0,\ldots,g_n) \in \G^{\times (n+1)} \mid t(g_i) =
s(g_{i - 1}) \ \mbox{for} \  1 \leq i \leq n, \ \mbox{and} \ t(g_0) = s(g_n) \},
\]
and $\tau_n:B_n\to M^{\times(n+1)}$ is here the map $\tau_n(g_0,\ldots,g_n)=(t(g_0),\ldots,t(g_0))$.
This is a simplicial space
by defining face  operators $d'_i: B_n \to B_{n-1}$,
%\[
%\begin{split}
%d_i(g_0, g_1, \ldots, g_n)  &= 
%\begin{cases}
%(g_0, \ldots, g_i g_{i+1}, \ldots, g_n)  
%\quad & \mbox{if} \ 0 \leq i \leq n -1  \\
%(g_n g_0, g_1, \ldots, g_{n-1}) \quad & \mbox{if} \
%i = n,
%\end{cases} \\
%s_i(g_0, g_1, \ldots, g_n)  &= 
%\begin{cases}
%(g_0, \ldots, g_i, 1_{t(g_{i+1})}, g_{i+1}, \ldots, g_n) \quad
%& \mbox{if} \ 0 \leq i \leq n-1 \\
%(g_0, \ldots, g_n, 1_{s(g_n)})  
%\quad & \mbox{if} \ i = n.
%\end{cases}
%\end{split}
%\]
\begin{equation*}
\begin{array}{rcll}
\label{BurDS}
d'_i(g_0, g_1, \ldots, g_n)  &=& \left\{
\begin{array}{l}
(g_0, \ldots, g_i g_{i+1}, \ldots, g_n)  
\\
(g_n g_0, g_1, \ldots, g_{n-1}) 
\end{array}\right. & \begin{array}{l} \mbox{if} \ 0 \leq i
  \leq n -1,   \\ \mbox{if} \ i = n, \end{array}\\
  \end{array}
\end{equation*}
and degeneracy operators $s'_i: B_n \to B_{n+1}$,
\begin{equation*}
\begin{array}{rcll}
s'_i(g_0, g_1, \ldots, g_n)  &=& \left\{
\begin{array}{l}
(g_0, \ldots, g_i, 1_{t(g_{i+1})}, g_{i+1}, \ldots, g_n) 
\\
(g_0, \ldots, g_n, 1_{s(g_n)})  
\end{array}\right. & \begin{array}{l} \mbox{if} \ 0 \leq i
  \leq n-1,  \\ \mbox{if} \ i = n. \end{array}
\end{array}
\end{equation*}
Furthermore, it has a cyclic operator 
$t'_n: B_n \to B_n$
defined by 
\begin{equation*}
\label{BurT}
t'_n(g_0, \ldots, g_n) =(g_n, g_0, \ldots, g_{n-1}),
\end{equation*}
turning $B_\bull$ into a cyclic object in the category of manifolds.
The map \eqref{red-loops} is a morphism of cyclic objects if we equip the right hand side 
with the cyclic structure induced by $B_\bull$, together with the (twisted) cyclic structure of the
cyclic object $(\calC^\infty_M)^\natural$ in the category of sheaves on $M$. This is the diagonal of
a bicyclic complex which is quasi-isomorphic to its total complex. 
On the level of Hochschild homology, this is the Eilenberg-Zilber
theorem (see, for example, \cite[Thm.\ 8.5.1]{Wei:AITHA}) which---in one direction---is implemented by the Alexander-Whitney map.
Applying the HKR~map on the level of sheaves, one eventually finds
\[
HH_\bull(C^\infty_c(\G))\cong
\underset{{p+q=\bull}}{\textstyle{\bigoplus}} H_p\left(\Lambda(\G),\Lambda^qT^*B_0\right).
\]
The groupoid $\Lambda(\G):=B_0\rtimes\G$ is disconnected in general, 
which induces a decomposition of the Hochschild and cyclic 
homology. The component $\G\subseteq\Lambda(\G)$ is called the unit component, and for this one finds
\begin{equation}
\label{loc-unit}
HH_\bull(C^\infty_c(\G))_{[1]}\cong
\underset{p+q=\bull}{\textstyle{\bigoplus}} H_p\left(\G,\Lambda^qT^*M\right).
\end{equation}

We compare this with the Hopf-cyclic
theory as follows:
using the isomorphisms \eqref{iso-gpd}, one has for the fundamental
space from \S \ref{invariants}
\[
B_n(C^\infty_c(\G))\cong C^\infty_c(B_n).
\]
One easily checks that the induced simplicial and cyclic operators are equal to the 
push-forwards along the simplicial and cyclic maps on $B_\bull$ as above. In a similar spirit, the invariant map 
$C^\infty_c(\G_n)\hookrightarrow C^\infty_c(B_n)$ from \eqref{invariant} 
is induced by the morphism 
\[
\G_n \to B_n, \quad (g_1,\ldots,g_n)\mapsto((g_1\cdots g_n)^{-1},g_1,\ldots,g_n).
\]
With this, we now see that the map
\[
HH_\bull^{\scriptscriptstyle\rm alg}(C^\infty_c(\G))\to HH^{\scriptscriptstyle\rm Hopf}_\bull(C^\infty_c(\G))
\]
induced by the projection $C^\infty_c(\G)^\natural\to
B_\bull(C^\infty_c(\G))$ is in turn induced by the projection
onto the degree zero component $\bigoplus_{i\geq 0}\Lambda^i
T^*M\to\underline{\C}$ of representations 
of 
\nolinebreak
$\G$. 

\begin{remark}
As remarked in \cite{Cra:CCACCFF}, the dual cyclic homology of a Hopf
{\em algebra} captures the full
``localisation at units'' of the cyclic homology of the underlying algebra. Here we see explicitly that this is 
not the case for Hopf algebroids: the right hand side of \eqref{loc-unit} has far more components than those 
appearing in Theorem 
\nolinebreak
\ref{cycl-etale}.
\end{remark}
\subsection{Lie-Rinehart algebras}
\label{Lie-Rinehart}
Important examples of Hopf algebroids also arise from Lie-Rinehart algebras as we shall now explain:
\subsubsection{Definitions}
Here we briefly recapitulate the basic definitions and properties of
Lie-Rinehart algebras, cf.\ \cite{Rin:DFOGCA, Hue:PCAQ}.
Let $A$ be a commutative algebra over the ground ring $k$, containing
$\mathbb{Q}$. 
A \textit{Lie-Rinehart algebra} over $A$
is a pair $(A,L)$, where $L$ is a $k$-Lie algebra equipped with an
$A$-module structure 
and a morphism of $k$-Lie algebras
$L\to\der_k A, \ X \mapsto \{a \mapsto X(a)\}$ such that
\begin{equation*}
\begin{array}{rcll}
(aX)(b)&=& a \big(X(b)\big), & \qqquad X \in L, \ a, b \in A, \\
{[X,aY]} &=& a[X,Y]+ X(a)Y,     & \qqquad X,Y \in L, \ a \in A.
\end{array}
\end{equation*}
The morphism $L\to\der_k(A)$ is usually referred to as the {\em
  anchor} of $(A,L)$.
For convenience we shall also assume that $A$ is unital in what follows.

A Lie-Rinehart algebra is the
 algebraic analogue of the notion of a
 Lie algebroid in differential geometry. The 
algebraic geometric generalisation is given by a sheaf of Lie
 algebroids, defined over a locally ringed space. 
In fact, a Lie-Rinehart algebra 
defines such a sheaf over the affine scheme ${\rm Spec}(A)$.

A {\em left $(A,L)$-module} over a Lie-Rinehart algebra is a left
$A$-module $M$ which is also a left Lie algebra module over $L$ with action
$X\otimes_k m \mapsto X(m)$ satisfying
\[
\begin{split}
(aX)(m)&=a \big(X(m)\big),\\
X(am)&= X(a)m+aX(m).
\end{split}
\]
Alternatively, we can view a left $(A,L)$-module as an $A$-module $M$ equipped with a flat left
$(A,L)$-connection: this is a map $\nabla^l:M\to\Hom_A(L,M)$ satisfying
\begin{equation}
\label{LeftConn2}
\nabla^l_X(am) = a \nabla^l_X(m) + X(a)m,
\end{equation}
for all $a\in A$, $X\in L$, and $m\in M$. Flatness amounts to the usual condition
\[
[\nabla^l_X,\nabla^l_Y]=\nabla^l_{[X,Y]},
\]
for all $X,Y\in L$. We write $\mathsf{Mod}(A,L)$ for the category of left $(A,L)$-modules.

The \textit{universal enveloping algebra} of a Lie-Rinehart algebra
$(A,L)$ is constructed as follows \cite{Rin:DFOGCA}:
the direct $A$-module sum $A \oplus L$ can be made into a $k$-Lie algebra by
means of the Lie bracket
\[
[(a_1, X_1),(a_2,X_2)]:=\big(X_1(a_2)-X_2(a_1), [X_1, X_2]\big).
\]
Let $\mathscr{U}(A \oplus L)$ denote its universal enveloping algebra and $\mathscr{U}^+(A
\oplus L)$ the subalgebra generated by the canonical image of $A
\oplus L$ in $\mathscr{U}(A \oplus L)$. For $z \in A \oplus L$, denote by $z'$
its canonical image in $\mathscr{U}^+(A \oplus L)$.
The quotient $\V L:= \mathscr{U}^+(A \oplus L)/I$, where $I$ is the two-sided
ideal in $\mathscr{U}^+(A \oplus L)$
generated by the elements $(az)' - a'z'$, $a \in A$,
is called the universal enveloping algebra of the Lie-Rinehart algebra
$(A,L)$. 
It comes equipped with a $k$-algebra morphism $i_A:A\to \mathscr{V}L$, as well as a morphism
$i_L:L\to {\rm Lie}(\mathscr{V}L)$ of $k$-Lie algebras, subject to the conditions
\[
i_A(a) i_L(X) = i_L(aX), \qquad i_L(X) i_A(a) - i_A(a) i_L(X) =
i_A(X(a)), \quad a \in A, \ X \in L.
\]
It is universal in the following sense: 
for any
other triple $(\mathscr{W}, \phi_L, \phi_A)$ of a $k$-algebra
$\mathscr{W}$ 
and two morphisms $\phi_A:A \to \mathscr{W}$, $\phi_L:L \to
{\rm Lie}(\mathscr{W})$ of $k$-algebras and $k$-Lie algebras, respectively, that for all 
$a \in A, \ X \in L$
obey
$$
\phi_A(a) \phi_L(X) = \phi_L(aX), \qquad \phi_L(X) \phi_A(a) -
\phi_A(a) \phi_L(X) = \phi_A(X(a)), 
$$
there is a unique morphism $\Phi: \V L \to \mathscr{W}$
of $k$-algebras such that $\Phi\circ i_A = \phi_A$ and $\Phi\circ i_L = \phi_L$.
This property shows that the
natural functor
$
\mathsf{Mod}(\mathscr{V}L)\to\mathsf{Mod}(A,L)
$
is an equivalence of categories. 
With this, the Lie-Rinehart cohomology of $(A,L)$ with values in a left $(A,L)$-module $M$ is defined as
\begin{equation}
\label{VLExt}
H^\bull(L,M):=\Ext^\bull_{\V L}(A,M).
\end{equation}
{\em The Poincar\'e-Birkhoff-Witt theorem.} The algebra $\mathscr{V}L$ carries a canonical filtration 
\begin{equation}
\label{pbwf}
\V L_{(0)}\subset \V L_{(1)}\subset \V L_{(2)}\subset\ldots
\end{equation}
by defining $\V L_{(-1)}:= 0$, $\V L_{(0)}:=A$ and $\V L_{(p)}$ to be the left $A$-submodule 
of $\V L$
generated by $i_L(L)^p$, i.e.\ products of the image of $L$ in $\V L$ of
length at most $p$. Since $a D - D a \in \V L_{(p-1)}$ for all
$a \in A$ and $D \in \V L_{(p)}$, left and right $A$-module
structures coincide on $\V L_{(p)}/\V L_{(p-1)}$. It follows that 
the associated graded object ${\rm gr} (\V L)$ inherits the structure of a
graded commutative $A$-algebra.

Let $S^\bull_AL$ be the graded symmetric $A$-algebra of $L$ and $S^p_AL$ its  degree $p$ part.
When $L$ is projective over $A$, 
the Poincar\'e-Birkhoff-Witt theorem (cf.\ \cite{Rin:DFOGCA}, and \cite{NisWeiXu:PDOODG}
in the context of Lie algebroids) states that the canonical $A$-linear epimorphism
$
S_AL \to {\rm gr} (\V L)
$
is an isomorphism of $A$-algebras.
While $i_A$ is always injective, in this case even $i_L$ is
injective and we may identify elements $a \in A$ and $X \in L$ with
their images in $\V L$.
Hence, the symmetrisation
\begin{equation*}
\label{Symm}
\pi: S^p_AL \to \V L_{(p)} \quad X_1 \otimes \cdots \otimes X_p \mapsto \frac{1}{p!}\sum_{\gs \in
S_p} X_{\gs(1)} \cdots X_{\gs(p)}
\end{equation*}
where $X_i \in L$ or $X_i \in A$, induces an isomorphism of left $A$-modules $S_AL \to \V L$.
\subsubsection{The associated Hopf algebroid}
The fact that Lie-Rinehart algebras give rise to left bialgebroids in the sense of Definition \ref{left-bialg} 
 by means of their enveloping algebras has been observed 
before in the literature, cf.\ \cite{Xu:QG, KhaRan:CCOEHA, MoeMrc:OTUEAOALRA}. 
In this section we shall determine the extra datum needed to define a Hopf algebroid structure.

In the previous section, we have discussed the category of left $\V L$-modules, and its interpretation
on the level of the Lie-Rinehart algebra as flat connections \eqref{LeftConn2}. Let us now consider 
{\em right} $\V L$-modules. A {\em
  right} $(A,L)$-connection (cf.\ \cite{Hue:LRAGAABVA}) on an $A$-module $N$ is a map
$\nabla^r: N \to \Hom_k(L,N)$ which fulfills
\begin{eqnarray}
\label{RightConn1}
\nabla^r_X(an) &=& a \nabla^r_X n - X(a)n \\ 
\label{RightConn2} 
\nabla^r_{aX}n &=& a \nabla^r_X n - X(a)n, \qqquad a \in A, \ X \in L,
\ n \in N.  
\end{eqnarray}
Again, the connection is called {\em flat} if one has
$[\nabla^r_X,\nabla^r_Y]=\nabla^r_{[Y,X]}$ for all $X,Y\in L$, 
in which case they integrate to a right $\V L$-module.
If $L$ is
$A$-projective of finite constant rank, then 
by \cite[Thm.\ 3]{Hue:LRAGAABVA} flat right $(A,L)$-connections on $A$ correspond to flat \textit{left} $(A,L)$-connections on
 $\textstyle\bigwedge^{\scriptscriptstyle{\rm max}}_A L$, the 
maximal exterior power of $L$. As such they were introduced in
\cite{Xu:GAABVAIPG} in the context of Lie algebroids to define Lie algebroid homology.
In the more general context of Lie-Rinehart algebras such flat right
$(A,L)$-connections on $A$
need not exist at all, cf.\ Remark \ref{hopf-times}.

For $M$ a right $(A,L)$-module---or, equivalently, a right $\V L$-module---we define Lie-Rinehart homology with 
coefficients in $M$ as
\begin{equation}
\label{VLTor}
H_\bull(L,M) := \Tor^{\V L}_\bull(M,A).
\end{equation}
We will now describe left and right bialgebroid structures on $\V L$: to
start with, set
\begin{equation*}
\label{SEqT}
s_r \equiv t_r \equiv s_l \equiv t_l \equiv i_A: A \hookrightarrow \V L.
\end{equation*}
With this identification at hand, the various $A$-module structures on $\V L$ reduce to 
left and right multiplication in $\V L$. With this, we write $\otimes^{ll}$ for the tensor product 
in $\mathsf{Mod}(\V L)$ and $\otimes^{rr}$ for the one in
$\mathsf{Mod}(\V L^\op)$.
%We then have
\begin{proposition}
\label{FLR}
Flat left and right $(A,L)$-connections on $A$
 correspond to respectively left and
right bialgebroid
structures on $\V L$ over $A$.
\end{proposition}
\begin{proof}
Flat left and right $(A,L)$-connections $\nabla^l$ and $\nabla^r$ on $A$
give rise to resp.\ left and right $\V L$-actions on $A$
which will be denoted, only in this proof, by $(D, a) \mapsto D \cdot a$ and $(a,D) \mapsto a
\cdot D$ for $a \in A$ and $D \in \V L$. 
Define left and right counit by
\begin{equation*}
%\label{piemonte}
\eps_l(D) := D \cdot 1_A, 
\qqquad 
%\mbox{and} 
\qqquad 
\eps_r (D) := 1_A \cdot D, 
\end{equation*}
for $D\in\V L$.
In particular, we have of course $\eps_l(a) = a = \eps_r(a)$ for $a \in
A$. Seen as maps $\V L \to A$, one has by the properties of a left connection
\begin{equation*}
\eps_l(D \eps_l(E)) = (D\eps_l(E)) \cdot 1_A
= D \cdot \eps_l (E)
= D \cdot (E \cdot 1_A) 
= (DE) \cdot 1_A = \eps_l(DE) 
\end{equation*}
with $D,E\in\V L$, and also by \rmref{RightConn1} and \rmref{RightConn2}
\begin{equation*}
\eps_r(\eps_r (D)E) = 1_A \cdot (\eps_r(D)E)
= \eps_r (D) \cdot E
= (1_A \cdot D) \cdot E
= 1_A \cdot (DE)= \eps_r (DE).
\end{equation*}
%hence the respective counit properties \rmref{leftcou} and \rmref{rightcou}.
Define left and right coproduct by setting
on generators $X \in L$, $a \in A$
\begin{equation*}
\begin{array}{rclcrcl}
\label{LCoProdVL1}
\gD_l X &=& 1 \otimes^{ll} X + X \otimes^{ll} 1 -
\eps_l (X) \otimes^{ll} 1, & & \quad \gD_l a &=& a \otimes^{ll}
1, \\
\gD_r X &=& 1 \otimes^{rr} X + X \otimes^{rr} 1 -
\eps_r (X) \otimes^{rr} 1, & & \quad \gD_r a &=& a \otimes^{rr} 1. \\
\end{array}
\end{equation*}
Extending these maps to the whole of $\V L$ by requiring them to corestrict to $k$-algebra
morphisms $\gD_l: \V L \to \V L \times_A \V L$ and $\gD_r: \V L \to \V L
\times^A \V L$ into the respective Takeuchi products (cf.\ page
\pageref{taki}) associated to the
$(A,A)$-bimodule structures \rmref{bimod-lmod} and \rmref{bimod-rmod}, respectively,
one easily checks that $(\V L, A, i_A, \gD_l,
\eps_l)$ is a left and $(\V L, A, i_A, \gD_r,
\eps_r)$ is a right bialgebroid, respectively.
\end{proof}
\begin{remark}
\label{baumeister}
The anchor of a Lie-Rinehart algebra yields a canonical flat left
$(A,L)$-connection and therefore defines a left bialgebroid structure. The associated left counit $\eps_l$
is simply the projection $\V L \to A$, and one has
$\eps_l (X) = 0$ and $\gD_l X = X \otimes^{ll} 1 + 1 \otimes^{ll} X$
for $X \in L$.
This is the left bialgebroid structure on $\V L$ of \cite{Xu:QG, KhaRan:CCOEHA, MoeMrc:OTUEAOALRA},
which we from now on will fix as {\em the} (canonical) left bialgebroid structure on $\V L$. Remark however that
for the {\em right} bialgebroid structure there is {\em no} canonical
choice, and in general $\eps_r(X) \neq 0$ for $X \in L$. 
\end{remark}
Next, we will define an antipode: let $(A,L)$ be a Lie-Rinehart algebra and $\nabla^r$ a right 
$(A,L)$-connection on $A$, and define the operator $\eps_r^L = \eps_r: L \to A, \ X \mapsto \nabla^r_X 1_A$. 
Define a pair of maps $S_L: L \to \V L$ and
$S_A: A \to \V L$ by
\begin{equation}
\label{TwAntip}
S_L(X) = -X + \eps_r(X), \quad S_A(a) = a, \qquad a \in A, \ X\in L. 
\end{equation}
Combining \rmref{RightConn1} with \rmref{RightConn2}, this implies that
$
S_L(aX) = -aX + \nabla^r_X a.
$
\begin{proposition} [Antipodes for Lie-Rinehart algebras]
\label{AntipProp}
The pair $(S_A, S_L)$ 
extends 
to 
a $k$-algebra antihomomorphism $S: \V L \to \V L$ 
if and only if the underlying right $(A,L)$-connection
on $A$ is flat. 
In such a case, $S$ is an involutive antipode with respect to the
canonical left
bialgebroid structure 
%from Remark \ref{baumeister} 
and the right bialgebroid
structure from \mbox{Proposition}~{\rm \ref{FLR}}.

Conversely, given a $k$-module isomorphism $S: \V L \to \V L$ satisfying 
$S(a_1Da_2) = a_2 S(D) a_1$ for all $D \in \V L$, $a_1, a_2 \in A$, and $S(1) = 1$, the assignment
\begin{equation*}
\label{RightConnifS}
\nabla^r: A \ \to \ \Hom_k(L,A), \quad a \ \mapsto \ \{ X \ \mapsto \
\eps_l(S(X) a)\}
\end{equation*}
defines a right $(A,L)$-connection on $A$ 
which is flat if and only if $S$ is a
$k$-algebra antihomomorphism. 
\end{proposition}

\begin{proof}
We use the universal property of $\V L$: clearly 
$S_A: A \to \V L$ is a morphism of $k$-algebras.
Next, compute
 \begin{equation*}
\begin{split}
[S_L X, &S_L Y] = [X,Y] + [Y,\eps_r(X)] - [X, \eps_r(Y)] + \eps_r(X) \eps_r(Y) -
\eps_r(Y) \eps_r(X) \\
&= [X,Y] + Y\eps_r(X) - X\eps_r(Y) \\
&= S_L([Y,X]) - \eps_r([Y,X]) + Y\eps_r(X) - X\eps_r(Y) + \eps_r(X) \eps_r(Y) -
\eps_r(Y) \eps_r(X) \\
&= S_L([Y,X])  - \nabla^r_{[Y,X]} 1_A +   \nabla^r_{X} \eps_r(Y)  -  \nabla^r_Y \eps_r(X) \\
&= S_L([Y,X])  - \nabla^r_{[Y,X]} 1_A +   \nabla^r_{X} \nabla^r_Y 1_A  -
\nabla^r_Y \nabla^r_X 1_A  \\
&= S_L([Y,X]) +\big([\nabla^r_{X}, \nabla^r_Y] - \nabla^r_{[Y,X]}\big)  (1_A). 
 \end{split}
 \end{equation*}
The term between brackets is the curvature of $\nabla^r$, so $S_L:L\to{\rm Lie}(\V L^{\op})$ is a morphism 
of $k$-Lie algebras if and only if $\nabla^r$ is flat. We now check
$$
S_A (a) S_L (X) = - Xa + a \eps_r(X) = -aX + \nabla^r_Xa
= S_L (aX)
$$
and also
$$
S_L (X) S_A (a) - S_A (a) S_L (X) = - aX + a \eps_r(X) + Xa - a \eps_r(X) = S_A (X(a)),
$$
so by the universal property of $\V L$, there exists a unique homomorphism $S:\V L\to \V L^{\op}$ which 
fulfills $S\circ i_A = S_A$ and $S\circ i_L =S_L$. 
If the connection is flat, the antipode axioms including $S^2 = \id$ are straightforward to 
check by considering a PBW basis of $\V L$ and making use of the antihomomorphism property.

For the converse statement, we need to check the properties \rmref{RightConn1} and
\rmref{RightConn2} in order to be a right connection. As for
\rmref{RightConn1}, we compute
\begin{equation*}
\begin{split}
\nabla^r_{X}(ab) &= \eps_l(S(X)ab) = \eps_l((-Xa + a\eps_r(X))b) = \eps_l((-aX - X(a) + a\eps_r(X))b) \\
               &= a\eps_l(S(X)b) - X(a)b = a\nabla^r_X b - X(a)b,
 \end{split}
 \end{equation*}
and \rmref{RightConn2} is left to the
 reader. 
To show flatness if and only if $S$ is a $k$-algebra antihomomorphism,
use again the universal property of $\V L$ to compare
\begin{equation*}
\begin{split}
[\nabla_Y,\nabla_X](a) &= \eps_l(S(Y) \eps_l(S(X)a)) -
\eps_l(S(X) \eps_l(S(Y)a)) \\
&= \eps_l(S(Y)S(X)a)) -
\eps_l(S(X)S(Y)a). 
\end{split}
\end{equation*}
with 
$
\nabla_{[X,Y]}a = \eps_l(S([X,Y])a).
% = \eps_l(S Y S X \cdot a) - \eps_l(S X S Y\cdot a).
$
This completes the proof.
\end{proof}
\begin{remark}
\label{hopf-times}
%{\em The $\times_A$-Hopf structure.} 
{\em The {\rm left} Hopf algebroid structure.} 
Let $\V L \otimes^{rl} \V L$ denote the
tensor product defined with respect to the ideal generated by $\{Da \otimes E - D
\otimes aE, \ D, E \in \V L, \ a \in A\}$. Although the antipode $S$ depends on
the right connection $\nabla^r$, 
the {\em translation map} $\V L \to \V L\otimes^{rl}\V L$, $D \mapsto D_+\otimes D_-:=D^{(1)}\otimes
S(D^{(2)})$
is 
independent of $\nabla^r$. Indeed, evaluated on a PBW basis $D=aX_{i_1}\cdots X_{i_n}$, one finds
\[
D_+\otimes D_-=\sum_{j=0}^n\sum_{i_1<\ldots< i_j,\atop i_{j+1}< \ldots <i_n}(-1)^{n-j} 
aX_{i_1}\cdots X_{i_j}\otimes X_{i_{j+1}}\cdots X_{i_n}.
\]
In \cite{KowKra:DAPIACT} it was proved that this defines a {\em left Hopf algebroid}
($\times_A$-Hopf algebra \cite{Schau:DADOQGHA}) structure on $\V L$, i.e., the Hopf-Galois map
$\beta\!:=\!\varphi_2\!:\! VL \otimes^{rl} VL \to VL \otimes^{ll} VL$, 
$D \otimes E \mapsto D_{(1)} \otimes D_{(2)} E$ (see \S\ref{adarte}) 
is bijective with inverse 
%$\psi_2: VL \otimes^{ll} VL \to VL \otimes^{rl} VL$, 
$D \otimes E \mapsto D_+ \otimes D_- E$. Among others, this implies that the map above satisfies several
identities of which we will only list the 
three needed later on when dealing with jet spaces:
\begin{eqnarray}
\label{Sch1}
\quad D_{+(1)} \otimes D_{+(2)}D_- &\!\!\!\!=&\!\!\!\! D \otimes 1 
\qqquad\qqquad\quad\quad \ \in \V L \otimes^{ll} \V L,\\
\label{Sch2}
\quad D_{+(1)} \otimes D_{+(2)} \otimes D_- &\!\!\!\!=&\!\!\!\! D_{(1)}\otimes D_{(2)+}
\otimes  D_{(2)-} \quad\ \, \in \V L \otimes^{ll} \V L \otimes^{rl} \V L, \\
%\label{Sch2}
%D_{(1)+} \otimes D_{(1)-} D_{(2)}  &=& D \otimes 1 \in \V L\otimes^{ll} \V L\\
%\label{Sch3}
%u_+ \otimes_\Aop u_- & \in
%& U \times_\Aop U\\
%\label{Sch37}
%u_+ \otimes_\Aop u_{-(1)} \otimes_A
%u_{-(2)}
%&=& u_{++}
%\otimes_\Aop u_-
%\otimes_A u_{+-}\\
%\label{Sch4}
%(uv)_+ \otimes_\Aop (uv)_- &=& u_+v_+
%\otimes_\Aop v_-u_-
%\in  {}_\blact U \otimes_\Aop U_\ract
%\\
\label{Sch47}
\quad D_+D_-&\!\!\!\!=&\!\!\!\! \eps_l(D).
%\label{Sch5}
%\eta(a \otimes b)_+ \otimes_\Aop
%\eta(a \otimes b)_- &=&
%\eta(a \otimes 1) \otimes_\Aop
%\eta(b \otimes 1),
\end{eqnarray}
%where in (\ref{Sch3}) we abbreviated
%$$
%		  U \times_\Aop U:=
%		  \Bigl\{\sum_i u_i \otimes_\Aop v_i \in
%		  {}_\blact U \otimes_\Aop U_\ract\,|\,
%		  \sum_i u_i \ract a \otimes_\Aop v_i=
%		  \sum_i u_i \otimes_\Aop a \blact
%		  v_i
%		  \Bigr\}
%$$
%and in (\ref{Sch37}) the tensor product
%over $A^\mathrm{\op}$ links the first and
%third tensor component (cf.~\cite[Equation (3.7)]{Schau:DADOQGHA}).
%By (\ref{Sch1}) and (\ref{Sch3})
%one can write
%\begin{equation}
%\label{pmb}
%\beta^{-1}(u \otimes_A v) = u_+ \otimes_\Aop u_-v.
%\end{equation}
%which is easily checked to be
%well-defined over $A$ with (\ref{Sch4})
%and (\ref{Sch5}).
Hence, if $A$ does not admit a flat right $(A,L)$-connection (see
\cite{KowKra:DAPIACT} for a counterexample), 
$\V L$ is merely a 
left Hopf algebroid,
%$\times_A$-Hopf algebra, 
but not a Hopf algebroid. Since every Hopf algebroid (with bijective antipode) can be described by {\em two} different kinds 
of bijective Hopf-Galois maps (see \cite[Prop. 4.2]{BoeSzl:HAWBAAIAD} for details), 
we thence propose the name {\em left Hopf algebroid} 
rather than $\times_A$-Hopf algebra (see also \cite[\S2.6.14]{Kow:HAATCT} why this is a reasonable terminology, 
apart from solving a pronunciation problem).
\end{remark}

\subsubsection{The cyclic theory of $\V L$}
In this section we present the computations of the Hopf-cyclic
cohomology and dual Hopf-cyclic homology of the 
universal enveloping algebra $\V L$ of a Lie-Rinehart algebra
$(A,L)$.  
Let for the rest of this section $L$ be projective as a left
$A$-module. 
Furthermore, let $\nabla=\nabla^r$ be a flat right $(A,L)$-connection on
$A$ with associated right counit $\eps_r$, and denote $A_\nabla$ for $A$ equipped with this
right $(A,L)$-module structure. By Proposition \ref{AntipProp}, the connection $\nabla$ determines an antipode,
and therefore the cyclic cohomology and homology are defined. In the following we shall write $B_\nabla$ for the 
corresponding cyclic cohomology operator to stress this dependence; remark that the Hochschild operator $b$ is independent of $\nabla$. 
Consider the 
exterior algebra $\textstyle\bigwedge^\bull_AL$ over $A$ equipped with the differential
$\partial:\textstyle\bigwedge^n_AL\to\textstyle\bigwedge^{n-1}_AL$ defined by
\[
\begin{split}
\partial(aX_1\wedge\cdots\wedge X_n):=&\sum_{i=1}^n(-1)^{i+1} 
\eps_r(aX_i)X_1\wedge\cdots\wedge\hat{X}_i\wedge\cdots\wedge X_n\\
&+\sum_{i<j}(-1)^{i+j}a[X_i,X_j]\wedge X_1\wedge\cdots\wedge\hat{X}_i
\wedge\cdots\wedge\hat{X}_j\wedge\cdots\wedge X_n
\end{split}
\]

\begin{theorem}
\label{LRCyclCohom}
Let $(A,L)$ be a Lie-Rinehart algebra with $L$ projective
over $A$ and equipped with a flat right $(A,L)$-connection $\nabla$ on $A$. 
The antisymmetrisation map
\[
X_1 \wedge \cdots \wedge X_n \mapsto
\frac{1}{n!}\sum_{\gs \in S_n} (-1)^\gs X_{\gs(1)} \otimes^{ll}
\cdots \otimes^{ll}  X_{\gs(n)},
\]
defines a quasi-isomorphism of
mixed complexes
\[
{\rm Alt}:\left( \textstyle\bigwedge^\bull_AL,0,\partial\right)\to
\left(C^\bull(\V L),b,B_\nabla\right)
\]
which induces natural isomorphisms 
\[
\begin{split}
HH^\bull(\V L)&\cong \textstyle\bigwedge^\bull_AL,\\
HP^\bull(\V L) &\cong \underset{n \equiv \bull \, {\rm mod} \,
  2}{\textstyle{\bigoplus}} \, H_{n}(L,A_{\nabla}). 
\end{split}
\]
\end{theorem}
\begin{proof}
The isomorphism for the Hochschild groups relies on a similar
consideration for $k$-modules \cite{Car:CDC, Kas:QG} and is also known in the Lie
 algebroid case \cite[Thm.~1.2]{Cal:FFLA}.
 The proof of the algebraic case proceeds analogously: first one checks that the morphism
${\rm Alt}: \textstyle\bigwedge_A^\bull L\to C^\bull(\V L)$ indeed
 commutes with the differentials, $b\circ {\rm Alt}=0$. 
Since the Hochschild cohomology only depends on the $A$-coalgebra structure,
it suffices to prove that the morphism
${\rm gr}({\rm Alt}):\textstyle\bigwedge_A^\bull L\to
 C^\bull\left({\rm gr}(\V L)\right) \cong 
C^\bull\left(S_AL\right)$ 
is a quasi-isomorphism: observe that $S_AL$ can be seen as the
universal enveloping algebra of the 
Lie-Rinehart algebra defined by the $A$-module $L$ equipped with zero bracket and zero anchor. With this, the PBW
map $S_AL \to \V L$ is an isomorphism of $A$-coalgebras.

Assume first $L$ to be
finitely generated projective over $A$. Localising with respect to a maximal ideal
 $\frakm\subset A$, the module $L_\frakm$ is free over $A_\frakm$ of rank (say) $r$, 
and the morphism descends to a cochain morphism
\[
{\rm gr}({\rm Alt})_\frakm:\textstyle\bigwedge_{A_\frakm}^\bull L_\frakm\to C^\bull(S_{A_\frakm} L_\frakm).
\]
We shall prove that this map is a quasi-isomorphism for all $\frakm$. 
%This then proves the first part of the theorem.
 Fix a basis $e_i\in L_\frakm,~i=1,\ldots, r$ over $A_\frakm$
 as well as a dual basis $e^i\in L_\frakm^*$. We then have 
$S_{A_\frakm}L_\frakm\cong A_\frakm[e_1,\ldots,e_r]$. 
The dual Koszul resolution of $A_\frakm$ by left $S_{A_\frakm}L_\frakm$-comodules has the form
\[
K': \quad  A_\frakm\longrightarrow S_{A_\frakm}L_\frakm\stackrel{d}{\longrightarrow}
S_{A_\frakm}L_\frakm\otimes_{A_\frakm}L_\frakm
\stackrel{d}{\longrightarrow}S_{A_\frakm}L_\frakm\otimes_{A_\frakm}
\textstyle\bigwedge^2_{A_\frakm}L_\frakm\stackrel{d}{\longrightarrow}\ldots
\]
with $d=\sum_{i=1}^r\iota_{e^i}\otimes e_i$, that is,
\[
d(D\otimes X_1\wedge\cdots\wedge X_n):=\sum_{i=1}^r \iota_{e^i} D \otimes e_i\wedge X_1\wedge\cdots\wedge X_n.
\]
Here $\iota_\alpha$ denotes the action of $\alpha\in L^* :=
\Hom_A(L,A)$ by
derivations: 
$\iota_\alpha D:=\ga(D_{(1)})D_{(2)}$, $D \in S_AL$, with respect to the coproduct
$\gD_{\scriptscriptstyle{SL}}$ on $S_AL$.
Defining a contracting homotopy by 
$$
s(D\otimes X_1\wedge\cdots\wedge X_n):=\sum_{i=1}^r (-1)^{i+1} DX_i
\otimes X_1\wedge\cdots\wedge \hat{X}_i \wedge \cdots \wedge X_n,
$$
it can be checked that $K'$ yields a cofree resolution in the category
$\mathsf{Comod}_\elle(S_{A_\frakm}L_\frakm)$, hence the resolution is also relative injective. 
To compare this with the cobar resolution, one shows that the natural map
\[
D_0\otimes D_1\otimes\cdots\otimes D_n\mapsto D_0\otimes \pr(D_1)\wedge\cdots\wedge \pr(D_n), 
\]
where $\pr:S_{A_\frakm}L_\frakm\to L_\frakm$ denotes the canonical projection, defines a cochain equivalence.
Indeed this amounts to the identity
\[
(\id\otimes \pr)\Delta_{\scriptscriptstyle{SL}}D =\sum_{i=1}^r\iota_{e^i}D \otimes e_i,
\]
for all $D\in S_{A_\frakm}L_\frakm$, an identity that is easily checked on generators. 
To compute the $\Cotor$ groups, we use the natural isomorphism
\[
A_\frakm\bx_{S_{A_\frakm}L_\frakm}\left(S_{A_\frakm}L_\frakm\otimes_{A_\frakm}
\textstyle\bigwedge^\bull_{A_\frakm}L_\frakm\right)\cong 
\textstyle\bigwedge^\bull_{A_\frakm}L_\frakm,
\]
which induces the zero differential on the right hand side. By the
fact that the 
projection $({S_{A_\frakm}L_\frakm})^{\otimes
  n}\to\textstyle\bigwedge^n_{A_\frakm}L_\frakm$ 
is a left inverse to ${\rm Alt}$, the claim now follows.

In the general case where $L$ is projective over $A$, but not finitely generated, 
there exists as in \cite[Thm.\ 3.2.2]{Lod:CH} a filtered
ordered set $J$ as well as an inductive system of finitely generated projective (or even free) $A$-modules $L_j$ such that 
$$
L \simeq \lim_{\stackrel{\longrightarrow}{j \in J}} L_j.
$$
Since both $HH$ (which is the
derived functor $\Cotor$ here) as well as $S$ commute with inductive
limits over a filtered ordered set, the projective case follows from the
finitely generated projective case.

To prove the second isomorphism, 
we need to show that ${\rm Alt}$ intertwines the cyclic cohomology differential with
$\partial$. The best way to do this is to use localisation onto
coinvariants. Let 
$B_{\nabla}: C^\bull(\V L) \to C^{\bull-1}(\V L)$ 
and $B: B^\bull(\V L) \to B^{\bull-1}(\V L)$ denote the cyclic
cohomology differentials of the mixed complexes associated to the
Hopf-cocyclic module $\V L_\natural$ and the fundamental $A$-coalgebra cocyclic
module $\V L_{\scriptscriptstyle{{\rm coalg}}, \natural}$,
respectively. 
As usual, $B = N \gs_{-1}(1 - \gl)$, where $\gl :=
(-1)^n \tau_n$, $N := \sum^n_{i=0} \gl^i$, and $\gs_{-1} := \gs_{n-1}
\tau_n$. Hence, $B:B^n(\V L)\to B^{n-1}(\V L)$ is given 
explicitly by
\[
\begin{split}
B(D_0\otimes\cdots\otimes D_n)=
&\sum_{i=0}^n\Big((-1)^{ni}\epsilon_l(D_0)D_{i+1}\otimes\cdots\otimes D_n\otimes D_1\otimes\cdots\otimes D_{i-1}\\
&-(-1)^{n(i-1)}\eps_l(D_n)D_{i+1}\otimes\cdots\otimes D_{n-1}\otimes D_0\otimes\cdots\otimes D_{i-1}\Big).
\end{split}
\]
Note that $B^n(\V L) \cong C^{n+1}(\V L)$ as $(A,A)$-bimodules in this example.
From our general considerations in \S\ref{coinvariants}, we have 
$B_{\nabla} \circ \overline{\Psi}_{\scriptscriptstyle{\rm coinv}} =
\overline{\Psi}_{\scriptscriptstyle{\rm coinv}} \circ B$ for the 
morphism $\overline{\Psi}_{\scriptscriptstyle{\rm coinv}}: B^n(\V L)\to
C^n(\V L)$. Using its right inverse \rmref{pollux}, it is seen that 
$$
{\rm Alt}( aX_1 \wedge \cdots \wedge X_n) =
\overline{\Psi}_{\scriptscriptstyle{\rm coinv}}
\Big(\frac{1}{n!}\sum_{\gs \in S_n} (-1)^\gs a
\otimes X_{\gs(1)} \otimes
\cdots \otimes  X_{\gs(n)}\Big).
$$ 
Since $L \subset \ker \eps_l$ and because $\eps_l$ is a
left $A$-module map, we can compute 
\begin{equation*}
\begin{split}
B_{\nabla} &\big({\rm Alt}(aX_1 \wedge \cdots \wedge X_n)\big) = \\
&= B_{\nabla} \overline{\Psi}_{\scriptscriptstyle{\rm
    coinv}}\Big(\frac{1}{n!}
\sum_{\gs \in S_n} (-1)^\gs a
\otimes X_{\gs(1)} \otimes \cdots \otimes X_{\gs(n)}\Big) \\
&= \overline{\Psi}_{\scriptscriptstyle{\rm coinv}} B\Big(\frac{1}{n!}\sum_{\gs \in S_n} (-1)^\gs a
\otimes X_{\gs(1)} \otimes \cdots \otimes X_{\gs(n)}\Big) \\
%&= \bar{\phi}_\pl N \gs_{-1}\big(\frac{1}{n!}\sum_{\gs \in P(n)} \sign \gs \, a
%\otimes^{ll} X_{\gs(1)} \otimes^{ll} \cdots \otimes^{ll} X_{\gs(n)} \\
%& \quad - (-1)^n X_{\gs(1)} \otimes^{ll} \cdots \otimes^{ll} X_{\gs(n)}
%\otimes^{ll} a\big) \\
%&= \bar{\phi}_\pl N (\frac{1}{n!}\sum_{\gs \in P(n)} \sign \gs \, 
%aX_{\gs(1)} \otimes^{ll} \cdots \otimes^{ll} X_{\gs(n)}) \\
&= \overline{\Psi}_{\scriptscriptstyle{\rm coinv}}
\Big(\frac{1}{(n-1)!} 
\sum_{\gs \in S_n} (-1)^\gs a X_{\gs(1)} \otimes
\cdots \otimes  X_{\gs(n)} \Big) \\
&= {{\frac{1}{(n-1)!}}} \sum_{\gs \in S_n} (-1)^\gs
S(aX_{\gs(1)}) \cdot \big(X_{\gs(2)} \otimes
\cdots \otimes  X_{\gs(n)}\big). 
\end{split}
\end{equation*}
Now as an element 
%for $a\in A$, $X\in L$, we have
in $\V L^{\otimes^{ll} n}$, it is easy to see that
\[
%\label{ComplProj2}
\gD^{n-1}_l S(aX) = - \sum^n_{i=1}  \underbrace{1\otimes\cdots\otimes 1}_{\scriptscriptstyle{i-1~{\rm times}}}
\otimes{aX}\otimes \underbrace{1\otimes\cdots\otimes 1}_{\scriptscriptstyle{n-i~{\rm times}}} 
+ \eps_r(aX)  \otimes \underbrace{1\otimes\cdots\otimes 1}_{\scriptscriptstyle{n-1~{\rm times}}}
\]
for $a\in A$, $X\in L$. With this one then obtains
\begin{equation}
\label{duoplast}
\begin{split}
%B_{\nabla} &\big({\rm Alt}(aX_1 \wedge \cdots \wedge X_n)\big) = \\
&\frac{1}{(n-1)!} \sum_{\gs \in S_n} (-1)^\gs
S(aX_{\gs(1)}) \cdot \big(X_{\gs(2)} \otimes
\cdots \otimes  X_{\gs(n)}\big) = \\
&= {{\frac{1}{(n-1)!}}} \sum_{\gs \in S_n} (-1)^\gs 
\eps_r(aX_{\gs(1)})X_{\gs(2)} \otimes \cdots \otimes  X_{\gs(n)} \\
&\quad - {{\frac{1}{(n-1)!}}} 
\sum^n_{i=1}\sum_{\gs \in S_n} (-1)^\gs a X_{\gs(2)} \otimes \cdots \otimes X_{\gs(1)}X_{\gs(i)} \otimes 
\cdots \otimes  X_{\gs(n)} \\
&= {\rm Alt}\big( \sum^n_{i=1} (-1)^{i + 1} \eps_r(aX_i) X_1 \wedge
\cdots \wedge \hat{X}_i \wedge \cdots \wedge X_n  \\
&\quad + \sum_{i<j} (-1)^{i+j} a [X_i, X_j] \wedge X_1
\wedge \cdots\wedge \hat{X}_i \wedge \cdots \wedge \hat{X}_j \wedge \cdots\wedge X_n
\big) \\
&= {\rm Alt} \big(\partial(a \otimes X_1 \wedge \cdots \wedge X_n)\big).
\end{split}
\end{equation}
This completes the proof.
\end{proof}
\begin{theorem}
Let $(A,L)$ be a Lie-Rinehart algebra. 
Under the same assumptions as in Theorem \ref{LRCyclCohom}, there are natural isomorphisms
\[
\begin{split}
HH_\bull(\V L)&\cong H_{\bull}(L,A_{\nabla}),\\
HC_\bull(\V L)&\cong  \underset{i \geq 0}{\textstyle\bigoplus} \, H_{\bull-2i}(L,A_{\nabla}).
\end{split}
\]
\end{theorem}
\begin{proof}
The first isomorphism follows from Theorem \ref{hoch-tor}, together
 with the definition \eqref{VLTor}
 of Lie-Rinehart homology as a $\Tor$ functor. The second isomorphism
 follows from Theorem \ref{cocomm-cycl} {\em ii}).
\end{proof}
\begin{proposition}
The isomorphism of Hochschild homology above is induced by the chain morphism
\[
\varphi^{-1}_n \circ n! \, {\rm Alt} :\left(\textstyle\bigwedge^n_AL,\partial\right)\to \left(C_n(\V L),b\right),
\]
where $\varphi$ is the Hopf-Galois map of Lemma {\rm \ref{hopf-galois}}.
\end{proposition}
\begin{proof}
In view of Theorem \ref{cyclic-duality} it is equivalent to prove that
 the map ${\rm Alt}:\textstyle\bigwedge_A^\bull L\to C^\bull(\V
 L)$
 maps the differential $\partial:\textstyle\bigwedge^n_AL\to\textstyle\bigwedge^{n-1}_A L$ to
\[
\check{b} := \sigma_{n-1}\circ
\tau_n+\sum_{i=0}^{n-1}(-1)^{i+1}\sigma_i,
\]
i.e., ${\rm Alt} \circ \pl = \check{b} \circ {\rm Alt}$ on $C^\bull(\V
L)$.
Since the maps  $\sigma_i$ are just given by the left counit acting on
the $i^{\rm\scriptscriptstyle th}$ slot of the tensor product, 
the second sum is zero when evaluated on the image of ${\rm Alt}$, and
we are left with the term $\sigma_{n-1}\circ\tau_n$, which gives
\[
\sigma_{n-1} \, \tau_n \, {\rm Alt}(aX_1\wedge\cdots\wedge X_n)=
\frac{1}{n!}\sum_{\sigma\in S_n}(-1)^\sigma S(aX_{\sigma(1)})
\cdot\left(X_{\sigma(2)}\otimes^{ll}\cdots\otimes^{ll} X_{\sigma(n)}\right).
\]
Inspection of the calculation \rmref{duoplast}
%By the computation in the previous proof of the $B$-operator, 
shows that this is exactly $\frac{1}{n}{\rm Alt}\big(\partial
(aX_1\wedge\cdots\wedge X_n)\big)$. Hence $\varphi^{-1}_n \circ
n! \, {\rm Alt}$ is a morphism of complexes. 
To prove that it
is a quasi-isomorphism, 
consider the so-called Koszul-Rinehart resolution:
\[
A\stackrel{i_A}{\longrightarrow} \V L\stackrel{b'}{\longrightarrow}
\V L\otimes_A L
\stackrel{b'}{\longrightarrow}\V L\otimes_A\textstyle\bigwedge_A^2L\stackrel{b'}
{\longrightarrow}\ldots,
\]
with $b': \V L \otimes_A \bigwedge^n L \to \V L \otimes_A
\bigwedge^{n-1} L$ given by
\begin{equation*}
\begin{split}
\label{LRHB}
b'(D \otimes X_1 \wedge \cdots \wedge X_n) &=\sum^n_{i=1}(-1)^{i-1}DX_i \otimes  X_1 \wedge \cdots \hat{X_i}
\cdots \wedge X_n \\
&+ \!\! \sum_{1 \leq i < j \leq n} (-1)^{i+j} D \otimes  [X_i, X_j] \wedge
X_1 \wedge \cdots \hat{X_i} \cdots \hat{X_j} \cdots \wedge X_n,
\end{split}
\end{equation*}
where $D \in \V L$ and $X_1,\ldots,X_n\in L$. This is a projective resolution of
$A$ in the category $\mathsf{Mod}(\V L)$. 
By the same computation as above, one shows that the map 
\[
n! \big(\id \otimes \varphi_n^{-1} \circ {\rm
  Alt} \big):\V L\otimes_A\textstyle\bigwedge^n_A L\to {\rm Bar}_n(\V L)
\]
is a homotopy equivalence. Taking $A\otimes_{\V L} -$ on both sides, one
finds the map of the proposition. 
This proves that it is a quasi-isomorphism.
\end{proof}

\subsection{Jet spaces}
\subsubsection{The dual jet space of a Lie-Rinehart algebra}
\vspace*{-.2cm} 
In this section we describe another Hopf algebroid associated to a Lie-Rinehart algebra $(A,L)$, 
the Hopf algebroid of $L$-jets. Some of its structure maps have been
used before in the literature, 
cf.\ \cite{NesTsy:DOSLADOHSSAIT,CalVdB:HCAAC}, here we give a complete description: it is in a
 certain sense the dual of $\V L$. ({\em Note added in proof}: this Hopf algebroid was later independently 
reobtained in \cite[App.\ A]{CalRosVdB:HCFLA}.) 
In general, duality in the category
 of bialgebroids has been 
described in \cite{KadSzl:BAODTEAD} (see \cite{BoeSzl:HAWBAAIAD} for an extension to
Hopf algebroids) assuming that the bialgebroid is finitely generated 
projective over the base algebra. This is clearly not the case for $\V
L$, but each successive quotient $\V L_{(p)}\slash \V L_{(p-1)}$ in 
the Poincar\'e-Birkhoff-Witt filtration \eqref{pbwf} is
projective, provided $L$ is projective
over $A$. 

For the rest of this section, let $L$ be finitely generated projective of constant rank
as an $A$-module.
The space of $p$-jets of $(A,L)$ is then defined as
\[
\mathscr{J}^pL:=\Hom_A(\V L_{\leq p},A),
\]
where $\V L_{\leq p}$ denotes the elements in $\V L$ of degree $p$ or less.
The infinite jet space is defined as the projective limit
\[
\mathscr{J}^\infty L:=\lim_{\longleftarrow}\J^pL.
\]
By definition, $\mathscr{J}^\infty L$ is complete with
respect to the canonical PBW filtration \rmref{pbwf}. In this section we
will therefore always complete tensor products using this filtration (cf.\ \cite{Qui:RHT}).

We are now going to show that this space carries the structure of a Hopf
algebroid over $A$:
first of all, there is a commutative algebra structure that can be
described using the (left) comultiplication on $\V L$:
\[
\phi_1\phi_2(D)=\phi_1(D_{(1)})\phi_2(D_{(2)}), \qquad
\phi_1,\phi_2\in \J^\infty L, \ D\in \V L. 
\]
The unit for this multiplication is given by the left counit
$\epsilon_l:\V L\to A$, since
\[
\epsilon_l\phi(D)=\epsilon_l(D_{(1)})\phi(D_{(2)})=\phi(\epsilon_l(D_{(1)})D_{(2)})=\phi(D).
\]
There are two homomorphisms $s,t:A\to \J^\infty L$ given by
\[
\begin{split}
s(a)(D)&:=\epsilon_l(aD)=a\epsilon_l(D),\\
t(a)(D)&:=\epsilon_l(Da)=D(a), \qquad a \in A, \ D \in \V L,
\end{split}
\]
where we recall that here and in the rest of this section $D(a) :=
\eps_l(Da)$, $D \in \V L$, is the canonical $\V
L$-action on $A$ given by extension of the anchor of $(A,L)$.
A small computation shows that the images commute, and therefore
$(\J^\infty L,A,s,t)$ is an $(s,t)$-ring. 
Next, we consider the coproduct. For this we need the following:
\begin{lemma}
\label{iso-jet}
There is a canonical isomorphism
\[
\J^\infty L\otimes_A \J^\infty L\cong\lim_{\longleftarrow\atop p}
\Hom_A\left((\V L\otimes^{rl}_A\V L)_{\leq p},A\right).
\]
\end{lemma}
\begin{proof}
By definition, $\J^\infty L\otimes_A \J^\infty L$ is the quotient of
 $\J^\infty L\otimes_k \J^\infty L$
by the ideal generated by 
$
\{t(a)\phi_1\otimes\phi_2-\phi_1\otimes s(a)\phi_2, \ \phi_1,
\phi_2\in\J ^\infty L, a \in A \}.
$
The first term in this ideal, evaluated on $D\otimes E\in\V L\otimes_k \V L$, we write out as:
\[
\begin{split}
(t(a)\phi_1\otimes\phi_2)(D\otimes E)&=t(a)\phi_1(D)\otimes\phi_2(E)\\
&=D_{(1)}(a)\phi_1(D_{(2)})\otimes\phi_2(E)\\
&=\phi_1\big(\eps_l(D_{(1)}a)D_{(2)}\big)\otimes\phi_2(E) =\phi_1(Da)\otimes\phi_2(E).
\end{split}
\]
The second term gives
\[
%\begin{split}
(\phi_1\otimes s(a)\phi_2)(D\otimes E) =\phi_1(D)\otimes
a\epsilon_l(E_{(1)})\phi_2(E_{(2)}) =\phi_1(D)\otimes \phi_2(aE).
%\end{split}
\]
Remark that these two expressions use exactly the $(A,A)$-bimodule structure on $\V L$ used in the 
$\otimes^{rl}$-tensor product.
It therefore follows that the map 
\[
\phi_1\otimes\phi_2\mapsto \{D\otimes E\mapsto \phi_1(D\phi_2(E))\}
\]
 induces the desired isomorphism.
\end{proof}
Observe now that the product on $\V L$ descends to a map $m:\V L\otimes^{rl} \V L\to \V L$. 
We can therefore dualise the product to obtain a coproduct \mbox{$\Delta:\J^\infty L\to \J^\infty L\otimes_A \J^\infty L$,}
\begin{equation*}
\label{prod-jet}
\phi(DE)=:\Delta(\phi)(D\otimes^{rl} E)=\phi_{(1)}\left(D\phi_{(2)}(E)\right).
\end{equation*}
Associativity of the multiplication implies that $\Delta$ is coassociative. 
The counit for this coproduct is given by $\eps:
\phi\mapsto\phi(1_{\scriptscriptstyle{\V L}})$. It is now easy to verify that 
$(\J^\infty L,A,s,t,\epsilon,\Delta)$ is a left bialgebroid, and since $\J^\infty L$ is commutative, it is also
a right bialgebroid. Hence, to obtain a Hopf algebroid all we need is an antipode.

As observed in \cite{NesTsy:DOSLADOHSSAIT}, there are two left $\V L$-module structures on $\J^\infty L$.
First there is the ``obvious'' module structure given by
\[
(D\cdot_1\phi)(E):=\phi(ED), \qquad \phi\in\J^\infty L, \ D, E \in \V L.
\]
Second, there is another left $\V L$-module structure constructed as follows: consider the
$A$-module structure defined by left 
multiplication by the source map, i.e., $(a\cdot\phi)(D):=(s(a)\phi )(D)=\phi(aD).$
On this $A$-module, there is a canonical left connection, also called
the {\em Grothendieck connection}, given by
\begin{equation}
\label{GK}
\nabla^l_X(\phi)(D):= X \big(\phi(D)\big)-\phi(XD), 
\qquad \phi\in\J^\infty L, \ D\in\V L, \ X \in L. 
\end{equation}
One easily checks that this connection is flat, and we can
write the induced $\V L$-module structure 
%in terms of a Poincar\'e-Birkhoff-Witt $D=aX_1\cdots X_n$ basis 
as
\[
\begin{split}
(D\cdot_2\phi)(E)
%&=\sum_{{i_1<\ldots<i_j}\atop{i_{j+1}<\ldots< i_n}}
%(-1)^{n-j}aX_{i_1}\cdots X_{i_j}\left(\phi(X_{i_n}\cdots X_{i_{j+1}}E)\right)\\
&=D_+\big(\phi(D_-E)\big), \qquad D, E \in \V L,
\end{split}
\]
where we used the canonical 
left Hopf algebroid
%$\times_A$-Hopf algebra 
structure on $\V L$, cf.\ Remark \ref{hopf-times}.
With respect to the coproduct, these two module structures satisfy:
\begin{equation}
\label{por}
\begin{split}
\Delta(D\cdot_1\phi)&=D\cdot_1\phi_{(1)}\otimes\phi_{(2)}\\
\Delta(D\cdot_2\phi)&=\phi_{(1)}\otimes D\cdot_2\phi_{(2)}
\end{split}
\end{equation}
We now define the antipode on $\J^\infty L$ to be
\[
(S\phi)(D):=\eps(D\cdot_2\phi)=D_+\big(\phi(D_-)\big).
\]
By construction, this is the map that intertwines the two module structures.

\begin{theorem}
\label{jet-ha}
Equipped with this antipode, $\J^\infty L$ is a Hopf algebroid with
involutive antipode in the sense of Definition
{\rm \ref{def-hopf-algbd}}.
\end{theorem}

\begin{proof}
Since $L$ acts on $\V L$ via \eqref{GK} by derivations, $L\to 
\Der_k(\J^\infty L)$ is a morphism of Lie algebras. It therefore
follows from the Poincar\'e-Birkhoff-Witt theorem 
 that
\begin{equation*}
\label{coprd-mod}
D\cdot_2(\phi_1\phi_2)=(D_{(1)}\cdot_2\phi_1)(D_{(2)}\cdot_2\phi_2).
\end{equation*}
Using this property, one finds that $S$ is a homomorphism of commutative algebras:
\[
\begin{split}
S(\phi_1\phi_2)(D)
&=(D\cdot_2(\phi_1\phi_2))(1_{\scriptscriptstyle{\V L}}) \\
&=((D_{(1)}\cdot_2\phi_1)(D_{(2)}\cdot_2\phi_2))(1_{\scriptscriptstyle{\V
    L}}) 
=((S\phi_1)(S\phi_2))(D).
\end{split}
\]
To prove that $S^2 = \id$, one first computes
$
(S^2 \phi)(D) = \eps_l\big(D_+D_{-+}\phi(D_{--})\big),
$
using the properties of a left counit.
Next, to find a simpler expression for $D_+D_{-+} \otimes
D_{--} \in VL \otimes^{rl} VL$, apply the Hopf-Galois map
$\beta$
from Remark \ref{hopf-times} to it:
\begin{equation*}
\begin{split}
\beta(D_+D_{-+} \otimes D_{--}) &= D_{+(1)}D_{-+(1)} \otimes D_{+(2)}D_{-+(2)}D_{--} \\
&= D_{+(1)}D_{-} \otimes D_{+(2)} \\
&= 1 \otimes D \in VL \otimes^{ll} VL, 
\end{split}
\end{equation*}
where \rmref{Sch1} and the fact that $VL$ is cocommutative were used. Hence 
$$
D_+D_{-+} \otimes D_{--} = \beta^{-1}(1 \otimes D) = 1_+
\otimes 1_- D = 1\otimes D \in VL \otimes^{rl} VL,
$$
and therefore $(S^2 \phi)(D) = \phi(D)$.

We now verify the axioms in
Definition \ref{def-hopf-algbd}: since $s = s_l=t_r$, $t = t_l=s_r$, 
the first one is trivially 
satisfied, whereas the second is equivalent to the coassociativity of
$\Delta$, because $\gD = \Delta_l=\Delta_r$. 
For the third one, with \rmref{Sch1}, \rmref{Sch47}, and the Leibniz rule
for the canonical left $\V L$-action on $A$ we compute:
\begin{equation*}
%\begin{split}
%& 
S\big(s(a)\big)(D)= D_+ \big(a \eps_l (D_-)\big) = \eps_l(D_{+(1)}a)\eps_l(D_{+(2)}D_-) =
D(a) = t(a)(D), %\\
%&  S\big(t(a)\big)(D) = D_+ \big(D_-(a)\big)
%=  (D_+ D_-)(a) = a\eps_l(D) = s(a)(D).
%\end{split}
\end{equation*}
for $a\!\in\! A$, $D \!\in\! VL$, and $S \circ t = s$ then follows using $S^2 = \id$.
Finally, since $S$ is an algebra
homomorphism and an involution, it suffices to prove one of the two identities in \rmref{TwAp}. 
For example, with \rmref{Sch2} and \rmref{Sch47} we obtain
\[
\begin{split}
\phi_{(1)}S(\phi_{(2)})(D)&=\phi_{(1)}(D_{(1)})D_{(2)+}\big(\phi_{(2)}(D_{(2)-})\big)\\
&=\phi_{(1)}(D_{+(1)})D_{+(2)}\left(\phi_{(2)}(D_-)\right)\\
&=\phi_{(1)}(D_+\phi_{(2)}(D_-))\\
&=\phi(D_+D_-)\\
&=\eps_l(D)\phi(1) = s(\eps(\phi))(D),
\end{split}
\]
and this is precisely the second identity in \eqref{TwAp}.
This completes the proof that $\J^\infty L$ has the structure of a
Hopf algebroid with involutive antipode.
\end{proof}
\begin{remark}
Theorem \ref{jet-ha} is remarkable in the sense that whereas the universal enveloping algebra 
$\V L$ of a Lie-Rinehart algebra carries no canonical Hopf algebroid structure, its dual $\J^\infty L$
is a Hopf algebroid without making further choices. Close inspection of the preceding proof shows
that the Hopf algebroid structure---more precisely the antipode---depends solely on the 
{\em left} Hopf algebroid
%$\times_A$-Hopf 
structure on $\V L$ which is canonical, i.e.\ does not depend on the choice of a flat right connection.
\end{remark}
\begin{remark}
In the construction of the jet space---now written as
$\J_l^\infty L$---we considered $\V L$ 
as an $A$-module by left multiplication. {\em Right} multiplication
leads to a space $\J^\infty_rL$, {\it a priori} without much structure.
Only after introducing a flat right $(A,L)$-connection on $A$ we can introduce a ring structure using the right 
coproduct $\Delta_r$ on $\V L$, as well as source and target maps 
using the right counit $\epsilon_r$. This does again lead to a Hopf algebroid, but one easily 
proves that the map $\phi\mapsto\phi\circ S$ defines an isomorphism $\J_l^\infty L\to\J^\infty_r L$ of
Hopf algebroids, where $S$ is the antipode on $\V L$ constructed from
the same flat right connection as in 
Proposition \ref{AntipProp}. 
\end{remark}

\subsubsection{The cyclic theory of $\J^\infty L$}
Let $(A,L)$ be a Lie-Rinehart algebra. 
If $L$ is $A$-projective, Lie-Rinehart cohomology with values in $A$ (cf.\ \rmref{VLExt}) can be
computed by the 
complex $\big(\! \Hom_A({\textstyle{\bigwedge^\bull_A}} L, A), d\big)$ with
differential 
$d: {\textstyle{\bigwedge^n_A}} L \to {\textstyle{\bigwedge^{n+1}_A}} L$
defined 
\nolinebreak
by
\begin{equation*}
\label{deRham}
\begin{split}
d\omega(X_0\wedge\cdots\wedge X_n)&= \sum_{i=0}^n(-1)^iX_i\big(\omega(X_0,\ldots,\hat{X}_i,\ldots, X_n)\big)\\
&\quad+\sum_{i<j}(-1)^{i+j}\omega([X_i,X_j],X_0,\ldots,\hat{X}_i,\ldots,\hat{X}_j,\ldots,X_n).
\end{split}
\end{equation*}
\begin{theorem}
Let $(A,L)$ be a Lie-Rinehart algebra, 
where $L$ is finitely generated
$A$-projective of constant rank. There are canonical isomorphisms
\[
\begin{split}
HH^\bull(\J^\infty L)&\cong H^\bull(L,A),\\
HC^\bull(\J^\infty L)&\cong{\textstyle{\bigoplus\limits_{i\geq 0}}} \,
H^{\bull+2i}(L,A).
\end{split}
\]
\end{theorem} 
\begin{proof}
Denote $L^* := \Hom_A(L,A)$. By the given conditions we
have
$\textstyle\bigwedge_A^\bull L^* \cong
\Hom_A({\textstyle{\bigwedge^\bull_A}} L,A)$.
To compute Hochschild cohomology, instead of the cobar resolution
one can use the dual of the Koszul-Rinehart 
resolution given by (cf.\ \cite{NesTsy:DOSLADOHSSAIT}) 
\[
0\longrightarrow A\stackrel{s}{\longrightarrow} \J^\infty
L\stackrel{\nabla}{\longrightarrow}\J^\infty L
\otimes_A\textstyle\bigwedge^1_AL^*\stackrel{\nabla}{\longrightarrow} 
\J^\infty L\otimes_A\textstyle\bigwedge^2_AL^*\stackrel{\nabla}{\longrightarrow}\ldots,
\]
where $\nabla$ is the continuation of the Grothendieck connection,
cf.\ \rmref{GK}:
\[
\begin{split}
\nabla(\phi&\otimes\omega) (X_1,\ldots,X_{n+1}) =\\
&=\sum_{i=1}^{n+1}(-1)^{i-1}\nabla^l_{X_i}\phi\otimes\omega(X_1,\ldots,\hat{X}_i,\ldots,X_{n+1})\\
& \qquad
+\sum_{i<j}(-1)^{i+j}\phi\otimes\omega([X_i,X_j],X_1\ldots,\hat{X}_i,\ldots,\hat{X}_j,\dots,X_{n+1}),
\end{split}
\]
for $\phi\in\J^\infty L$, $\omega\in\textstyle\bigwedge^n_AL^*$ and $X_1,\ldots,X_{n+1}\in L$.
 It follows from \eqref{por} that this is indeed a cofree resolution
 of $A$ in the category of left $\J^\infty L$-comodules 
(remark that $s:A\to\J^\infty L$ is a morphism of left $\J^\infty L$-comodules).
 To compute the $\Cotor$ groups, we take invariants and apply the isomorphism \rmref{staub}:
 \[
 \textstyle\bigwedge^\bull_AL^*\stackrel{\cong}
{\longrightarrow}A\bx_{\J^\infty L}\left(\J^\infty L\otimes_A\textstyle\bigwedge^\bull_A L^*\right),
 \]
given by $X_1\wedge\cdots\wedge X_n\!\mapsto\! 1_{\scriptscriptstyle{A}}\!\otimes\! 1_{\scriptscriptstyle{\J^\infty L}}\!\otimes\! X_1\wedge\cdots\wedge X_n$. 
Since the unit in $\J^\infty L$ is given 
by the left counit $\eps_l\! : \!\V L\!\to\! A$, the induced differential is 
exactly the differential for Lie-Rinehart cohomology. This proves the isomorphism for Hochschild cohomology.
The second isomorphism on cyclic cohomology follows from Theorem~\ref{cocomm-cycl}~{\it i}).
\end{proof}
\begin{theorem} 
\label{jet-hom}
Let $(A,L)$ be a Lie-Rinehart algebra, where $L$ is finitely generated
$A$-projective of constant rank.
There is a natural morphism of mixed complexes
\[
F:\left(C_\bull(\J^\infty L),b,B\right)\to \left(\textstyle{\textstyle\bigwedge_A^\bull L^*},0,d\right)
\] 
defined in degree $n$ by
\[
F(\phi^1\otimes\cdots\otimes \phi^n)(X_1\wedge\cdots\wedge X_n):=
(-1)^n\left(S(\phi^1)\wedge\cdots\wedge S(\phi^n)\right)(X_1,\ldots, X_n),
\]
which induces isomorphisms
\[
\begin{split}
HH_\bull(\J^\infty L)&\cong\textstyle\bigwedge_A^\bull L^*,\\
HP_\bull(\J^\infty L)&\cong \prod_{i\geq 0} H^{\bull+2i}(L,A).
\end{split}
\]
\end{theorem}
\begin{proof}
This statement is very much the dual of Theorem \ref{LRCyclCohom}. 
The dual of the PBW isomorphism gives $\J^\infty L\cong \hat{S}_AL^*$ as commutative algebras.
Similar to Lemma \ref{iso-jet} there is a 
canonical isomorphism 
\[
C_n(\J^\infty L)\cong \lim_{\longleftarrow\atop p}\Hom_A\left(\big(\V L^{\otimes^{ll}n}\big)_{(p)},A\right), 
\]
induced by the map 
\[
(\phi^1\otimes\cdots\otimes \phi^n)(D_1\otimes\cdots\otimes D_n)=S(\phi^1)(D_1)\cdots S(\phi^n)(D_n).
\]
Observe that $C_n(\J^\infty L)$ is defined here with respect to the
tensor product in the category  $\mathsf{Comod}_R(\J^\infty L)$, the dual of $\otimes ^{rr}$,
and the antipode is needed to go from (the duals of) $\V L^{\otimes^{rr}n}$ to 
$\V L^{\otimes^{ll}n}$, to make the map $F$ well-defined. Since
$\J^\infty L$ is a commutative algebra, it maps the Hochschild differential $b$ to zero.

Clearly, $F$ is a morphism of $A$-modules, where $A$ acts on
$C_\bull(\J^\infty L)$ by 
multiplication by $t(a),~a\in A$,
on the first component. We can therefore localise with respect to a
maximal ideal $\frakm\subset A$ 
to prove that $F$ is 
a quasi-isomorphism. As in the proof of Theorem \ref{LRCyclCohom},
$L_\frakm$ is 
free of rank $r$ over $A_\frakm$, and we choose a basis 
$e_i\in L_\frakm,~e^i\in L^*_\frakm,~i=1,\ldots,r$. The Koszul resolution
\[
0\longleftarrow A_\frakm\stackrel{\eps}{\longleftarrow}\J^\infty
L_\frakm\stackrel{\partial'}{\longleftarrow}
\J^\infty L_\frakm\otimes_{A_\frakm} L^*_\frakm
\stackrel{\partial'}{\longleftarrow}
\J^\infty L_\frakm\otimes_{A_\frakm}\textstyle\bigwedge^2_{A_\frakm} 
L^*_\frakm\stackrel{\partial'}{\longleftarrow}\ldots
\]
is a 
free resolution
of $A_\frakm$ in the category $\mathsf{Mod}(\J^\infty L_\frakm)$ with differential
\[
\partial'(\phi\otimes \omega)=\sum_{i=1}^r e^i\phi\otimes\iota_{e_i}\omega.
\]
The natural map $\J^\infty
L_\frakm\otimes_{A_\frakm}\textstyle\bigwedge^\bull_{A_\frakm}
L^*_\frakm\to {\rm Bar}_\bull(\J^\infty L_\frakm)$
given by 
\[
\phi\otimes\alpha_1\wedge\cdots\wedge \alpha_n:=\phi_0\otimes(\alpha_1\circ \pr)\wedge\cdots\wedge(\alpha_n\circ \pr), 
\]
is a morphism of complexes as one easily checks. Since $S(\alpha\circ \pr)=-\alpha\circ \pr$ for $\alpha\in L^*$, the
map $\id\otimes F_\frakm: {\rm Bar}_\bull(\J^\infty L_\frakm)\to\J^\infty
L_\frakm\otimes_{A_\frakm}\textstyle\bigwedge^\bull_{A_\frakm}L^*_\frakm$ 
is a right inverse
and induces the morphism $F$ when taking the tensor product
$A_\frakm\otimes_{\J^\infty L_\frakm} - $ on both sides. 
This proves the first claim.

As for the second, notice that 
one has $B_n(\J^\infty L)\cong C_{n+1}(\J^\infty L)$ 
since $\J^\infty L$ is commutative, and the 
map to invariants $\Psi_{\scriptscriptstyle{\rm inv}}\!:\!C_n(\J^\infty L)\!\to\! C_{n+1}(\J^\infty
L)$ of \S \ref{invariants} is a morphism of cyclic modules. 
Explicitly, this map, when restricted to $L^*$, is given by
\[
\begin{split}
&\Psi_{\scriptscriptstyle{\rm inv}}(\phi^1\otimes \cdots\otimes \phi^n)(X_1\otimes\cdots\otimes X_{n+1})=\\
&=S\phi_{(1)}^1(X_1)\cdots S\phi^n_{(1)}(X_n) \nabla^l_{X_{n+1}}(\phi^1_{(2)}\cdots\phi^n_{(2)})(1)\\
&=\sum_{i=1}^n\big(S\phi^1(X_1)\cdots\widehat{S\phi^i(X_i)}\cdots
S\phi^n(X_n)\big)S\phi_{(1)}^i(X_i)
\big( X_{n+1}(\phi^i_{(2)}(1))-\phi_{(2)}^i(X_{n+1})\big).
\end{split}
\] 
Since the cyclic structure on $C_{\bull+1}(\J^\infty L)$ depends
only on the structure of $\J^\infty L$ as a commutative algebra, 
it is well-known (see, for example, \cite{Lod:CH}) that the morphism
\[
\phi^1\otimes\cdots\otimes\phi^{n+1}\mapsto\phi^{n+1}d\phi^1\wedge\cdots\wedge d\phi^n
\]
induces a morphism of mixed complexes $\left(C_{\bull}(\J^\infty
L)[1],b,B\right)\to \big(\textstyle\bigwedge^\bull_A L^*,0,
d\big)$. 
Composing this morphism with $\Psi_{\scriptscriptstyle{\rm inv}}$ as above, one finds exactly the map 
stated in the theorem. This proves that it intertwines the
$B$-operator with the coboundary operator for Lie-Rinehart cohomology.
Since we already know that this map is a quasi-isomorphism on the
level of Hochschild homology, 
the $SBI$ sequence implies that it is a quasi-isomorphism for cyclic homology. This proves the theorem.
\end{proof}
\subsubsection{Lie groupoids}
Here we explain the relationship between the previous constructions
and so-called {\em formal Lie groupoids} \cite{Kar:FSGOADQ}, 
justifying the name jet spaces. Among others, it gives a natural explanation of the Hopf 
algebroid structure. 
%We start with the following general construction:
  
Let $X\subset Y$ be a closed subset of a smooth manifold $Y$. Its {\em formal neighbourhood} is
the commutative ring
\[
\J^\infty_Y(X):=C^\infty(Y)\slash I^\infty_X,
\]
where $I_X$ denotes the ideal of functions vanishing on $X$, and $I^\infty_X=\bigcap I^{n+1}_X$. 
It has the following functorial property: let $f:(X_1,Y_1)\to(X_2,Y_2)$ be a smooth map from $Y_1$ to
$Y_2$ with the property that $f(X_1)\subset X_2$. This induces  a canonical morphism of rings
$f^*:\J^\infty_{Y_2}(X_2)\to\J^\infty_{Y_1}(X_1)$ by pull-back.

Consider the Lie-Rinehart algebra arising
 from a Lie algebroid $E(\G)$ of a Lie groupoid $s, t:\G\rightrightarrows M$:
this is the vector bundle over $M$ 
defined by the kernel of the derivative of the source map: $E(\G):=\ker(ds)|
_M$. The derivative of the target map, restricted to $M$, provides the
anchor, 
so that the space of sections of $E(\G)$ defines a Lie-Rinehart 
algebra over $A=C^\infty(M)$. Let $\calC^\infty_M$ denote the
structure sheaf of smooth functions on $M$, 
and define the following sheaf on 
$M$:
\begin{equation*}
\label{def-jet}
\J^n_\G:=s_*\left(\calC^\infty_{\G}\slash\mathcal{I}_M^{n+1}\right),
\end{equation*}
where $\mathcal{I}_M$ denotes the sheaf of smooth functions on $\G$
vanishing on $M$. 
This defines a sheaf of commutative algebras on 
$M$  which
has two natural inclusions $\calC^\infty_M\hookrightarrow\J^n_\G$
given by pull-back via $s$ or $t$. 
As above, $\J^\infty_\G$ 
denotes the projective limit of these sheaves. 
The pair $(M,\J^\infty_\G)$ is a locally ringed space, and the ring of global sections $\J^\infty_\G(M)$ 
is the formal neighbourhood of $M$ in $\G$ as defined above.
\begin{remark}
For the so-called {\em pair groupoid} $M\times M\rightrightarrows M$,
source and target map are given by the projection onto the
first resp.\ second component.
% and composition given by $m((x,y)(y,z))=(x,z)$. 
The associated Lie algebroid over $M$ is then nothing but the 
tangent bundle $TM$. Since the unit inclusion is just the diagonal
map, 
the definition above is the standard definition, cf.\ e.g.\ \cite[Ch.\ 1]{KumSpe:LEGT}, of 
the sheaf of jets of smooth functions on $M$.
\end{remark}
\begin{proposition}
There is a canonical isomorphism $\J^p_\G(M)\cong\J^p(E(\G))$.
\end{proposition}
\begin{proof}
On $M$ there is a short exact sequence of vector bundles
\[
0 \longrightarrow E(\G) \longrightarrow T\G|_M\stackrel{ds}{\longrightarrow} TM\longrightarrow 0.
\]
There is therefore a canonical map 
$$
\J^p_\G(M)\to\J^p(E(\G)), \quad f \mapsto \{ D \mapsto D(f) \},
$$ 
% given by
% associating to $f\in \J^p_\G(M)$ the homomorphism
% $$
% D\mapsto D(f),
% $$
where we view $D\in\V E(\G)_{\leq p}$ as a germ of a differential operator on $\G$ of order $\leq p$. 
This map is clearly left $\calC^\infty_M$-linear, so it indeed defines an element in $\J^p(E(\G))$.
Let $(x_1,\ldots,x_s,y_1,\ldots,y_r):U\to\R^{s+r}$ be local
coordinates on $U\subset\G$, 
where $(x_1,\ldots,x_s)$ are defined on $U\cap M$. 
For some $f\in\calC^\infty_\G(U)$ we have
by Taylor's expansion
\begin{equation}
\label{taylor}
f(x,y)=\sum_{|\alpha|\leq p}D_y^\alpha f(x,0)\frac{y^\alpha}{\alpha!}\quad\mbox{mod}\,\,\mathcal{I}_M^{p+1},
\end{equation}
where $\alpha=(\alpha_1,\ldots,\alpha_r)$ denotes a multiindex,
$|\alpha|=\sum_i\alpha_i$, 
$\alpha!=\alpha_1!\cdots\alpha_r!$, and 
$D^\alpha_y=\partial^{|\alpha|}\slash\partial y_1^{\alpha_1}\cdots
\partial y_r^{\alpha_r}$. 
This gives locally a representative of each local 
section of $\J^p_\G$ as a polynomial of degree $\leq p$ in the $y$-coordinates.
A general element $D\in\V L_{\leq p}$ can locally be written as 
\[
D=\sum_{|\alpha|\leq p}c_\alpha(x)D^\alpha_y,
\]
with $c_\alpha\in\calC^\infty_M(U)$, and this shows that the map
defined above is an isomorphism in each degree. 
Taking the projective limit proves the proposition.
\end{proof}
As remarked, the formal neighbourhood $\J^\infty_\G(M)$ comes equipped with two homomorphisms
$s,t:C^\infty(M)\to \J^\infty_\G(M)$ given by pull-back along the
groupoid source resp.\ target map. 
As a commutative algebra, it therefore inherits the structure of an
$(s,t)$-ring. 

\noindent Consider now the inclusion $(u,u)\!:\!M\hookrightarrow
\G_2$, and define the following sheaf on $M$:
\[
\J^\infty_{\G_2}:=\lim_{\underset{p}{\longleftarrow}} s_*\left(\calC^\infty_{\G_2}\slash \mathcal{I}_{M}^{p+1}\right),
\]
where $s:\G_2\to M$ is defined as $s(g_1,g_2)=s(g_1)$.
\begin{proposition}
There is a canonical isomorphism of sheaves
\[
\J^\infty_\G\otimes_{\calC^\infty_M}\J^\infty_\G\stackrel{\cong}{\longrightarrow}\J^\infty_{\G_2}
\]
by which the coproduct $\Delta$ is identified with the pull-back of the multiplication.
\end{proposition}
\begin{proof}
Define the morphism of sheaves as follows:
let $f_1$ and $f_2$ be local sections of $\J^\infty_\G$. Define the
local section $f$ 
of $\J^\infty_{\G_2}$ stalkwise by
\[
[f]_{(g_1,g_2)}:=[f_1]_{g_1}[f_2]_{g_2}.
\]
Clearly, this morphism factors over the ideal generated by 
$(s^*\calC^\infty_M\otimes 1-1\otimes t^*\calC^\infty_M)$
in $\J^\infty_\G\otimes \J^\infty_\G$ defining the tensor product and is therefore well-defined.
With respect to $\G_2$, there is a short exact sequence of vector bundles over $M$
\[
0\longrightarrow E(\G_2)\longrightarrow T\G_2\stackrel{ds}{\longrightarrow} TM\longrightarrow 0,
\]
where $E(\G_2)$ is the vector bundle with fiber at $x\in M$ given by
\[
E(\G_2)_x=\{(X,Y)\in E(\G)_x\oplus E(\G)_x \mid  dt(X)=ds(Y)\}.
\]
It then follows from \eqref{taylor} that the map defined above is an isomorphism.
\end{proof}
Next, we turn to the antipode, given by the dual
of the
groupoid inversion
map, $S:=i^*$. 
Notice that on the level of sheaves  $i:\G\to \G$ induces a morphism
\[
i^*:\J^\infty_\G\to t_*\left(\calC^\infty_{\G}\slash\mathcal{I}_M^{\infty}\right),
\]
but on the level of global sections it defines a homomorphism 
$S:\J^\infty_\G(M)\to\J^\infty_\G(M)$ satisfying $S(s^*f_1\phi
t^*f_2)=s^*f_2S(\phi)t^*f_1$ 
for all $\phi\in\J^\infty_\G(M)$ 
and $f_1,f_2\in C^\infty(M)$. With this antipode, it is easy to check
that all Hopf algebroid axioms in 
Definition \ref{def-hopf-algbd} are satisfied by the fact that $\mathsf{\G}\rightrightarrows M$ is a Lie groupoid.
\begin{remark}
It is clear from the construction above that not the full 
groupoid $\G\rightrightarrows M$ is needed, but rather
its structure in a neighbourhood of $M$ in $\G$. 
Such an object is called a \textit{local groupoid}. Although for a general
Lie algebroid there may be obstructions to integrate to a Lie
groupoid \cite{CraFer:IOLB}, 
one can always find an integrating  
local Lie groupoid, see Cor.\ 5.1 of [{\em loc.\ cit.}]. 
The previous construction gives therefore an alternative proof of Theorem \ref{jet-ha}
for Lie algebroids.
\end{remark}

\begin{remark}[The van Est isomorphism]
Let $G$ be a compact Lie group with Lie algebra $\frakg$. One may
consider $G$ as a Lie groupoid with only one object, 
the unit, and the previous construction defines a Hopf
algebra of jets of functions on $G$ at the unit. 
In this case, $\V\frakg=\mathscr{U}\frakg$, the universal enveloping algebra of $\frakg$.
Therefore $\J^\infty\frakg=\hat{S}\frakg^*$, and the preceding theorem gives
\[
HC_\bull(\J^\infty(\frakg))\cong H^\bull_{\scriptscriptstyle{\rm Lie}}(\frakg,\R).
\]
On the other hand, $C^\infty(G)$ has a Hopf algebra structure by
dualising the structure maps of $G$, 
provided one uses the projective tensor product $\hat{\otimes}$ and its
property
$C^\infty(G) \hat{\otimes} C^\infty(G)\cong C^\infty(G\times G)$, cf.\ \cite{Get:TECCFNCLG}.
For this Hopf algebra one has
\[
HC_\bull(C^\infty(G))\cong {\textstyle{\bigoplus\limits_{i\geq 0}}}H^{\bull-2i}_{\scriptscriptstyle{\rm diff}}(G,\R).
\]
There is an obvious morphism $C^\infty(G)\to\J^\infty(\frakg)$ of Hopf algebras by taking the jet 
of a function at the unit. On the level of cyclic homology, this
induces a map 
$H^\bull_{\scriptscriptstyle{\rm diff}}(G,\R)
\to H^\bull_{\scriptscriptstyle{\rm Lie}}(\frakg,\R)$, which is the van Est map.
\end{remark}
\begin{example}[The coordinate ring of an affine variety] Let $A$ be the coordinate ring of
an affine variety $X$. For the Lie-Rinehart algebra $(A, \Der_k A )$ we have
\[
\J^\infty(\Der_k(A))=\lim_{{\longleftarrow}\atop{p}}\left(A\otimes A\slash\mathfrak{m}^{p+1}\right),
\]
where $\mathfrak{m}\subset A\otimes A$ is the kernel ideal of the multiplication. We can consider this 
Hopf algebroid to be the localisation of the enveloping algebra $\Ae$---viz.\ the pair groupoid---of  \S \ref{pg} 
to the diagonal $X\subset X\times X$. By Theorem \ref{jet-hom} we have
\[
\begin{split}
HH_\bull(\J^\infty(X))&\cong \Omega^\bull X,\\
HP_\bull(\J^\infty(X))&\cong \prod_{i\geq 0}H^{\bull+2i}_{\rm alg}(X).
\end{split}
\]
Since $A$ is commutative, we also have $HC_\bull(A)\cong H^\bull_{\rm alg}(X)$. In view of 
Proposition \ref{pg-cycl}, compare this with the van Est isomorphism of the previous remark.
\end{example}

\bibliographystyle{alpha}

\end{document}